\pdfoutput=1
\documentclass[11pt]{article}

\usepackage{amssymb,amsmath,amsthm,mathtools}

\theoremstyle{plain}
\newtheorem{theorem}             {Theorem}    [section]
\newtheorem{lemma}      [theorem]{Lemma}
\newtheorem{proposition}[theorem]{Proposition}
\newtheorem{corollary}  [theorem]{Corollary}

\theoremstyle{definition}
\newtheorem{definition} [theorem]{Definition}
\newtheorem{assumption} [theorem]{Assumption}

\theoremstyle{remark}
\newtheorem*{remark}             {Remark}

\newcommand{\bbN}{\mathbb{N}}
\newcommand{\bbR}{\mathbb{R}}
\newcommand{\bbZ}{\mathbb{Z}}
\newcommand{\calI}{\mathcal{I}}
\newcommand{\calJ}{\mathcal{J}}
\newcommand{\calV}{\mathcal{V}}
\newcommand{\eps}{\varepsilon}

\DeclareMathOperator{\sign}{sign}
\newcommand{\laplace}{\ensuremath{\Delta}}
\newcommand{\slaplace}[2][\empty]{\ensuremath{%
    \ifx\empty#1%
    \Delta #2%
    \else%
    \Delta_{#1} #2%
    \fi%
  }%
}

\newcommand{\sm}{\setminus}
\newcommand{\set}[2][\empty]{\ensuremath{%
    \left\{%
      \ifx\empty#1%
      \relax%
      \else%
      #1:%
      \fi%
      #2%
    \right\}%
  }%
}

\newcommand{\deriv}[2][\empty]{\ensuremath{%
    \ifx\empty#1%
    \frac{\mathrm{d}}{\mathrm{d}{#2}}%
    \else%
    \frac{\mathrm{d}{#1}}{\mathrm{d}{#2}}%
    \fi%
  }
}
\newcommand{\tderiv}[2][\empty]{{\textstyle\deriv[#1]{#2}}}

\newcommand{\dint}{\mathrm{d}}

\RequirePackage{graphicx,color}
\graphicspath{{figures/}}

\usepackage[unicode,colorlinks=true,citecolor=blue,linkcolor=blue]{hyperref}

\newcommand{\noref}[1]{}

\usepackage[twoside,centering,a4paper,bottom=2.7cm,textwidth=15.24cm]{geometry}

\def\inserttitle{}
\def\insertshorttitle{}
\def\insertauthor{}
\def\insertshortauthor{}

\makeatletter
\renewcommand{\title}[2][]{%
  \def\inserttitle{#2}%
  \def\@title{#2}%
  \ifx&#1&%
  \def\insertshorttitle{#2}%
  \else%
  \def\insertshorttitle{#1}%
  \fi%
  \markboth{\hfill\scshape\insertshortauthor\hfill}%
  {\hfill\scshape\insertshorttitle\hfill}
}
\renewcommand{\author}[2][]{%
  \def\insertauthor{#2}%
  \def\@author{#2}%
  \ifx&#1&%
  \def\insertshortauthor{#2}%
  \else%
  \def\insertshortauthor{#1}%
  \fi%
  \markboth{\hfill\scshape\insertshortauthor\hfill}%
  {\hfill\scshape\insertshorttitle\hfill}
}
\makeatother

\usepackage[font=small,format=plain,labelfont=bf,up]{caption}

\numberwithin{figure}{section}
\numberwithin{table}{section}
\numberwithin{equation}{section}

\title{Mathematical analysis of a coarsening model with local
  interactions}

\author[M.~Helmers, B.~Niethammer and J.~J.~L.~Vel\'azquez]{%
  {Michael Helmers, Barbara Niethammer and Juan J.~L.~Vel\'azquez}\\[1.5ex]
  Institute for Applied Mathematics, University of Bonn\\
  Endenicher Allee 60, 53115 Bonn, Germany\\
  {\tt helmers@iam.uni-bonn.de, niethammer@iam.uni-bonn.de,}\\
  {\tt velazquez@iam.uni-bonn.de}
}

\date{\today}

\begin{document}

\maketitle
\thispagestyle{empty}

%%%%%%%%%%%%%%%%%%%%%%%%%%%%%%%%%%%%%%%%%%%%%%%%%%%%%%%%%%%%%%%%%%%%%%%%%%%%%%%
%%%%%%%%%%%%%%%%%%%%%%%%%%%%%%%%%%%%%%%%%%%%%%%%%%%%%%%%%%%%%%%%%%%%%%%%%%%%%%%
%%%
%%% Abstract, keywords, MSC(2010), table of contents
%%%
%%%%%%%%%%%%%%%%%%%%%%%%%%%%%%%%%%%%%%%%%%%%%%%%%%%%%%%%%%%%%%%%%%%%%%%%%%%%%%%
%%%%%%%%%%%%%%%%%%%%%%%%%%%%%%%%%%%%%%%%%%%%%%%%%%%%%%%%%%%%%%%%%%%%%%%%%%%%%%%

\begin{abstract}
  We consider particles on a one-dimensional lattice whose evolution
  is governed by nearest-neighbor interactions where particles that
  have reached size zero are removed from the system.
  Concentrating on configurations with infinitely many particles, we
  prove existence of solutions under a reasonable density assumption
  on the initial data and show that the vanishing of particles and the
  localized interactions can lead to non-uniqueness.
  Moreover, we provide a rigorous upper coarsening estimate and
  discuss generic statistical properties as well as some non-generic
  behavior of the evolution by means of heuristic arguments and
  numerical observations.
\end{abstract}

{\footnotesize\tableofcontents}

%%%%%%%%%%%%%%%%%%%%%%%%%%%%%%%%%%%%%%%%%%%%%%%%%%%%%%%%%%%%%%%%%%%%%%%%%%%%%%%
%%%%%%%%%%%%%%%%%%%%%%%%%%%%%%%%%%%%%%%%%%%%%%%%%%%%%%%%%%%%%%%%%%%%%%%%%%%%%%%
%%%
%%% Sec 1: Introduction
%%%
%%%%%%%%%%%%%%%%%%%%%%%%%%%%%%%%%%%%%%%%%%%%%%%%%%%%%%%%%%%%%%%%%%%%%%%%%%%%%%%
%%%%%%%%%%%%%%%%%%%%%%%%%%%%%%%%%%%%%%%%%%%%%%%%%%%%%%%%%%%%%%%%%%%%%%%%%%%%%%%

\section{Introduction}
%\label{sec:introduction}

Discrete systems with local interactions are used in a wide range of
models for growth processes.
Examples are the coarsening of sand ripples \cite{HeKr02} or the
clustering in granular gases \cite{MeWeLo01}, where matter is
transported between adjacent ripples or clusters, and certain particle
hopping models \cite{Spitzer70}, where particles occupy cells on a
one-dimensional lattice and jump between neighboring cells with rates
depending on the particle density.

Discrete systems with local interactions also appear as reduced models
of more complex situations such as the coarsening of droplets in
dewetting thin films \cite{GlWi03,GlWi05}, the evolution of phase
domains in the convective Cahn-Hilliard equation \cite{WaOtRuDa03}, or
the motion of grain boundaries in polycrystalline materials
\cite{HeNiOt02}.
Furthermore, local interactions emerge in discrete or discretized
ill-posed diffusion equations such as the Perona-Malik equation used
in image segmentation \cite{PeMa90} or population dynamics
\cite{HoPaOt04}, and there they lead to localized aggregation and
coarsening as well \cite{EsGr09,EsSl08}.

In this paper, we consider a discrete system of particles of sizes
$(x_j)$, $j \in \bbZ$ whose evolution for time $t>0$ is governed by
\begin{equation}
  \label{eq:intro-eq}
  \dot x_j = x_{j+1}^{-\beta} - 2 x_j^{-\beta} + x_{j-1}^{-\beta}
\end{equation}
as long as $x_j>0$; particles that have reached $x=0$ are removed from
the system and the remaining ones are relabeled.
The parameter $\beta$ is a positive real number.
Equation \eqref{eq:intro-eq} is closely related to the first examples
above where our particles represent sand ripples, gas clusters or
lattice cells. It can also be regarded as the discretized
backward-parabolic equation $x_j = \laplace g(x_j)$ with decreasing
mass transfer function $g(s) = s^{-\beta}$ and $\laplace$ denoting the
discrete Laplacian on $\bbZ$.
It differs from the discrete Perona-Malik equation, though, since the
latter has a forward-parabolic region near the origin, whereby small
particles do not vanish but cluster near the boundary to the
backward-parabolic region.
Moreover, due to its rather simple structure we see
\eqref{eq:intro-eq} as a useful toy problem for studying local
correlations in coarsening systems and thus as a step towards
understanding more complex and higher dimensional evolutions such as
grain growth or the Mullins-Sekerka flow with positive volume
fraction.

\begin{figure}[t]
  \centering
  \includegraphics[width=.496\textwidth]{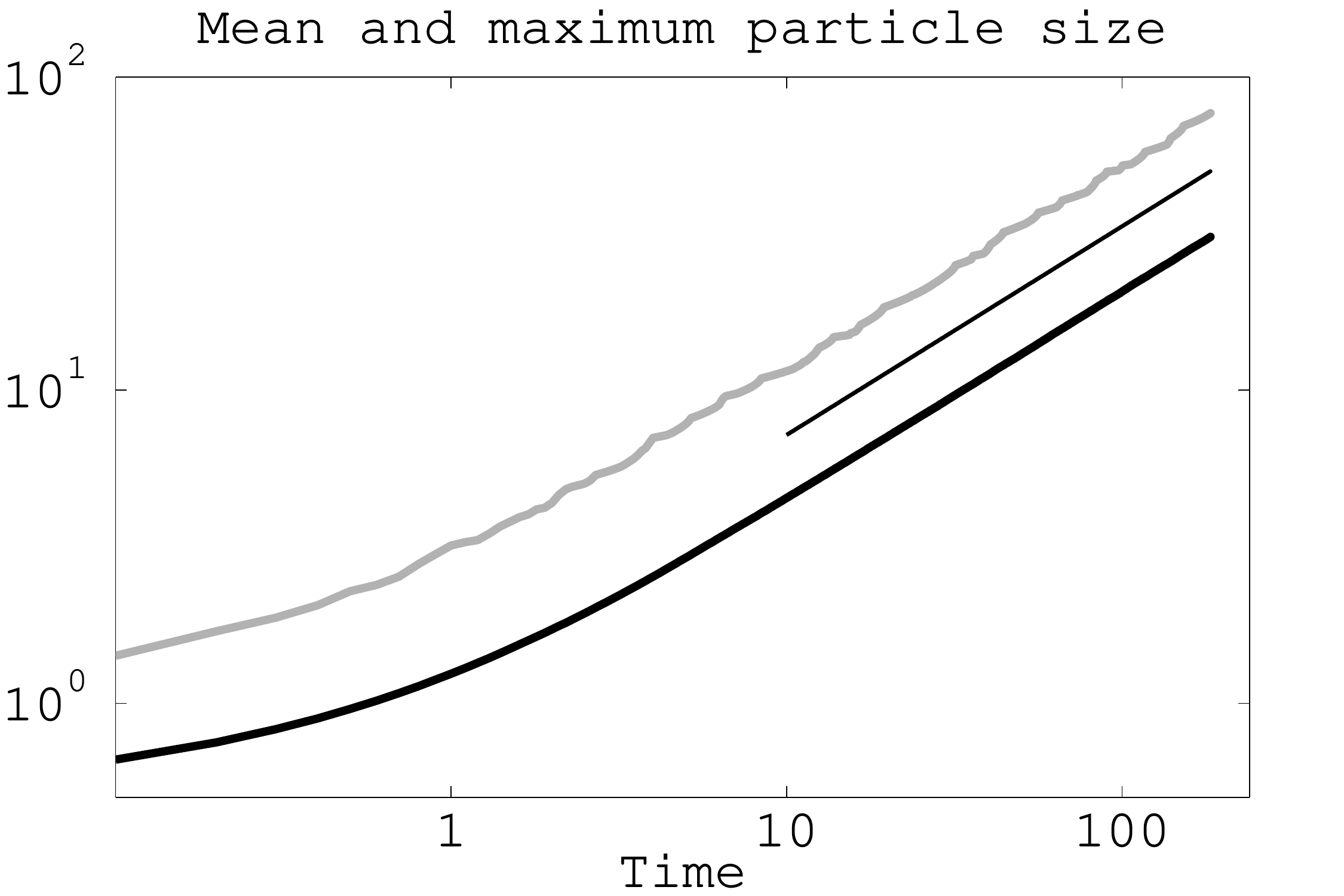}
  \includegraphics[width=.496\textwidth]{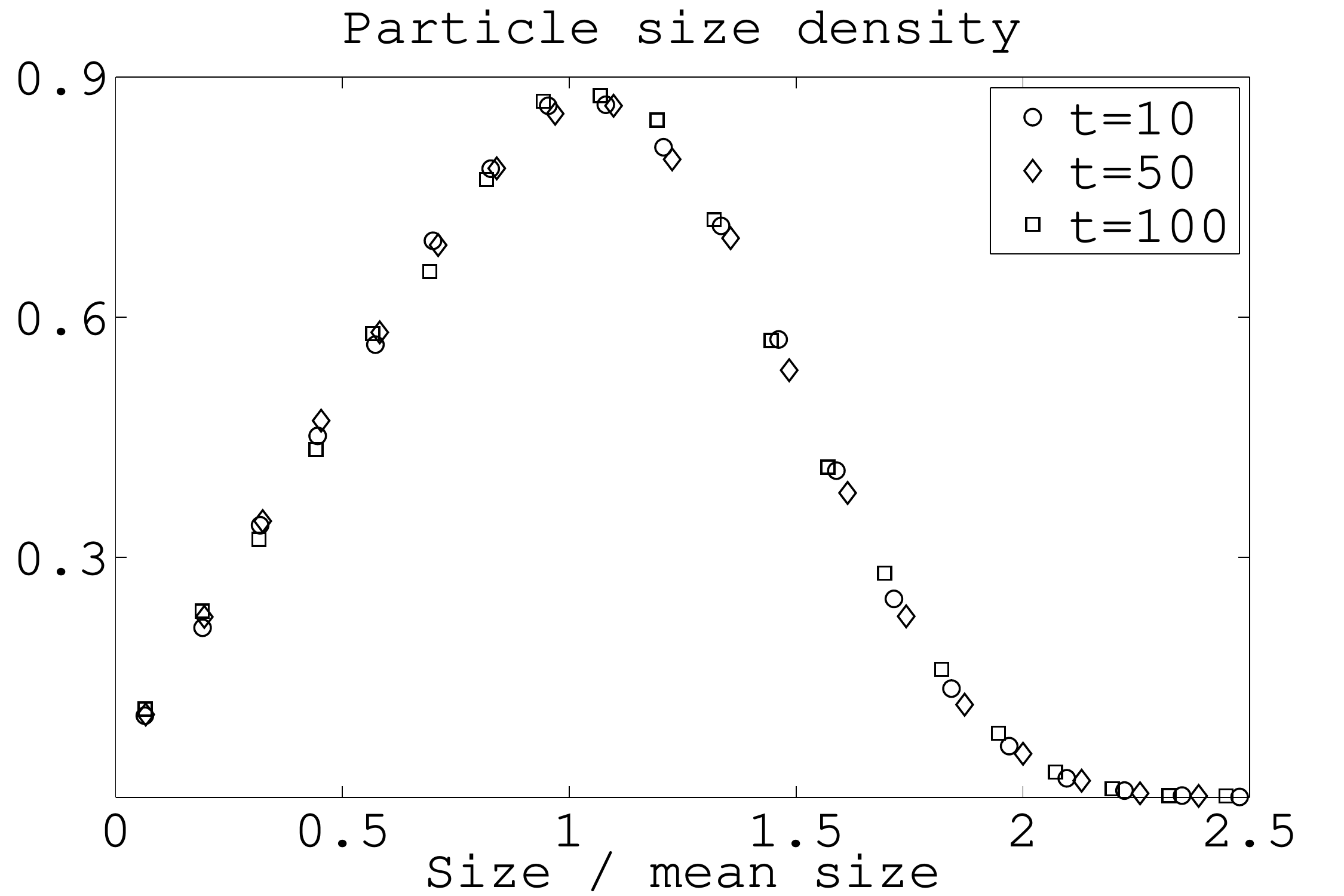}
  \caption{Numerical simulation of \eqref{eq:intro-eq} with
    $\beta=0.5$ and initial particle sizes uniformly distributed in
    the interval $(0,1)$.  \emph{Left.} Mean particle size (black) and
    maximum particle size (gray) over time, both in logarithmic
    scale. The thin line represents the power law $t \mapsto
    t^{1/(\beta+1)}$.  \emph{Right.} Particle size density scaled by
    the mean size at three different times during the evolution.}
  \label{fig:generic}
\end{figure}

The most interesting aspect of \eqref{eq:intro-eq} and the examples
above are the universal or generic statistical properties, which the
evolutions seem to exhibit.
Often, these can be derived formally from the equation or observed in
numerical simulations of large finite systems, and for solutions to
\eqref{eq:intro-eq} the key findings are the following.
\begin{enumerate}
\item Particles disappear in finite time, whereas the total particle
  size $\sum_{j \in \bbZ} x_j$ is formally conserved during the
  evolution. Thus, the typical particle size grows, and a dimensional
  analysis suggests the growth law
  \begin{equation}
    \label{eq:intro-growth-law}
    \text{typical particle size }
    \sim
    ~t^{\frac{1}{\beta+1}}.
  \end{equation}
  The latter is confirmed by numerical results as depicted in Figure
  \ref{fig:generic}.
  
\item The simulations in Figure \ref{fig:generic} also indicate
  self-similarity of the particle size distribution in the sense
  \begin{equation*}
    \text{particle size density}(t,x)
    ~
    \sim
    ~
    t^{-\frac{1}{\beta+1}} p\left(x t^{-\frac{1}{\beta+1}}\right)
  \end{equation*}
  for some profile $p \colon [0,\infty) \to [0,\infty)$.
  
\item Not all solutions display the behavior just described. For
  instance, constant initial data remain constant for all times, and
  as indicated in Figure \ref{fig:non-generic} there are solutions
  where certain particles grow faster than the law
  \eqref{eq:intro-growth-law}.
\end{enumerate}

\begin{figure}[t]
  \centering
  \includegraphics[width=.496\textwidth]{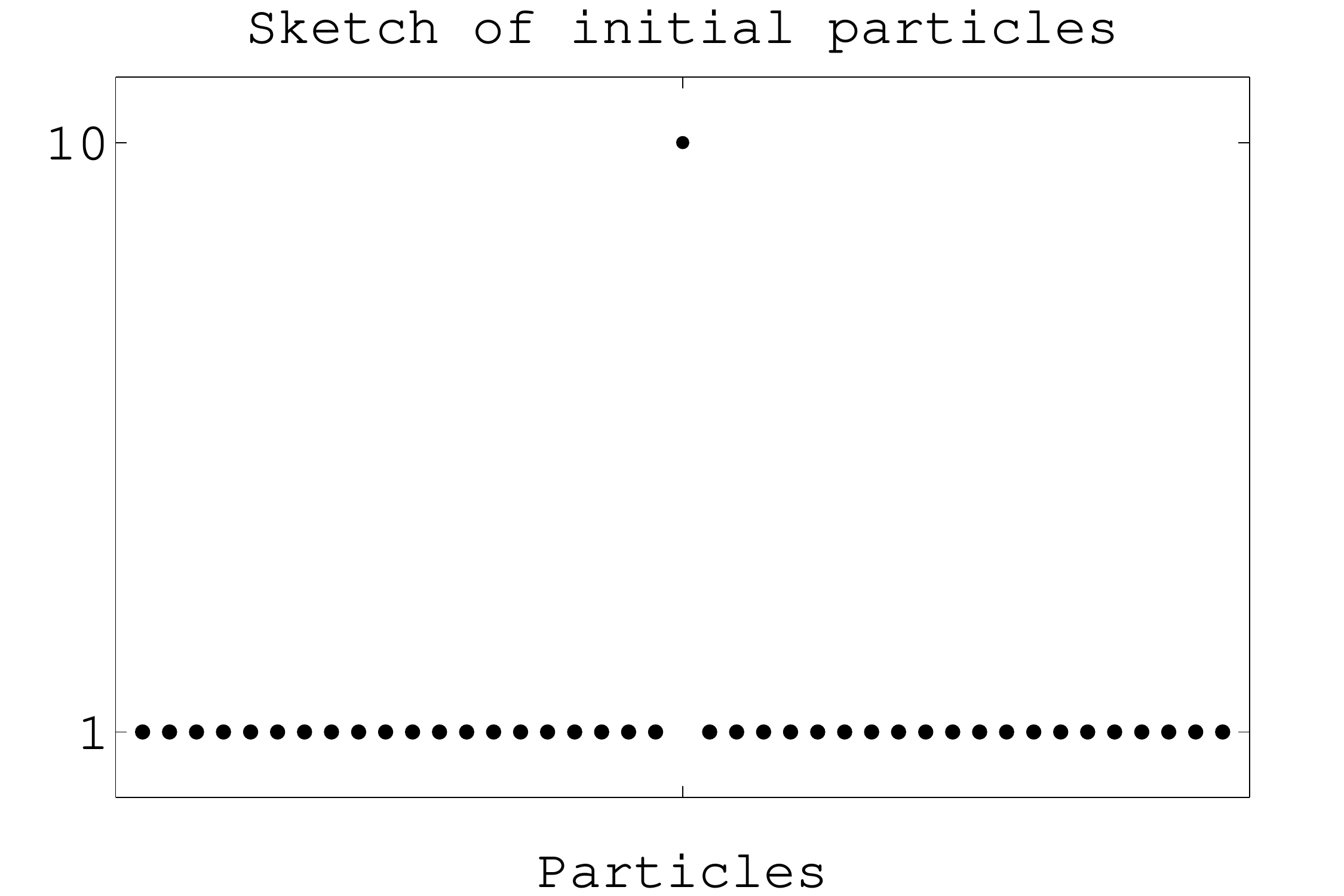}
  \includegraphics[width=.496\textwidth]{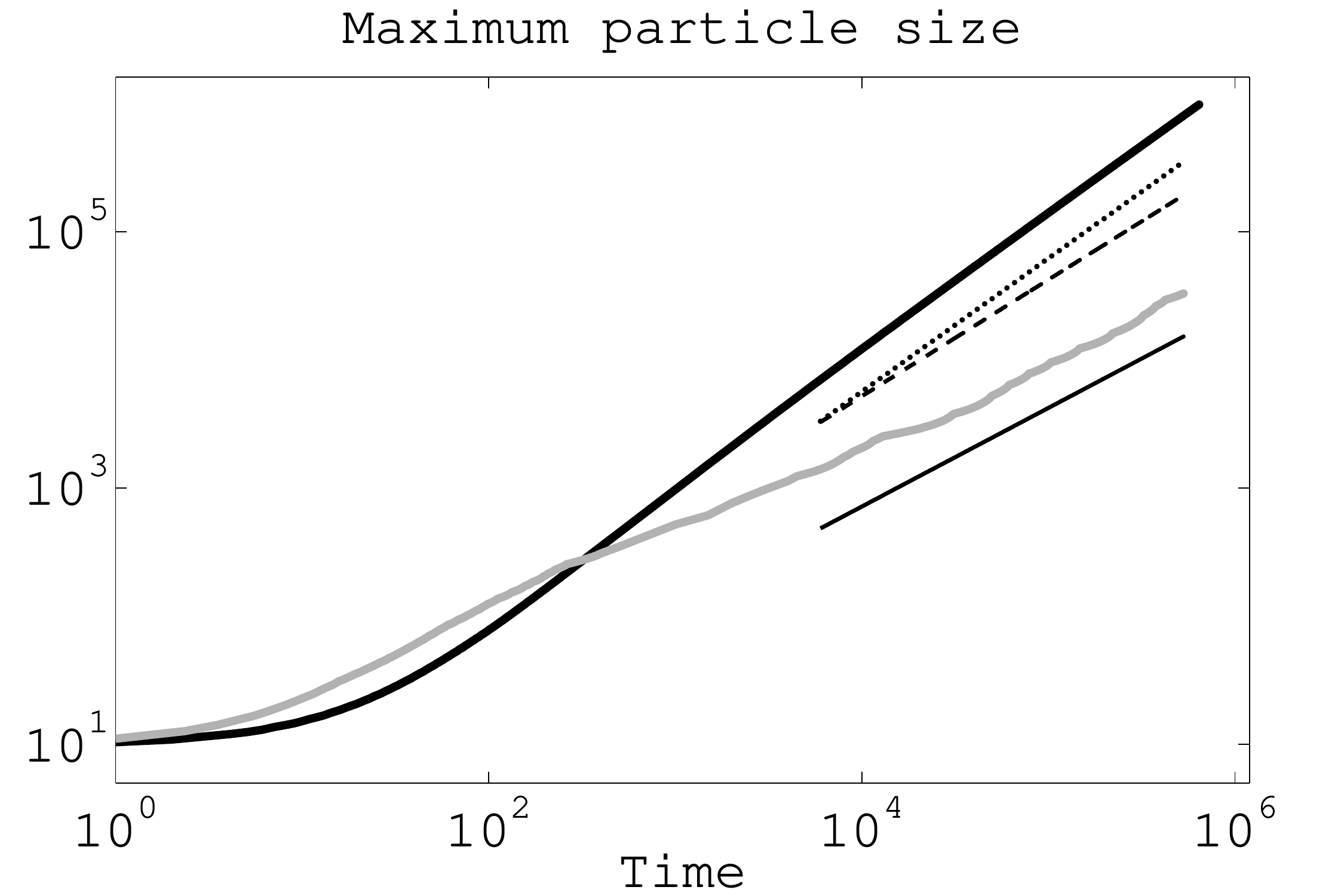}
  \\[1ex]
  \includegraphics[width=.496\textwidth]{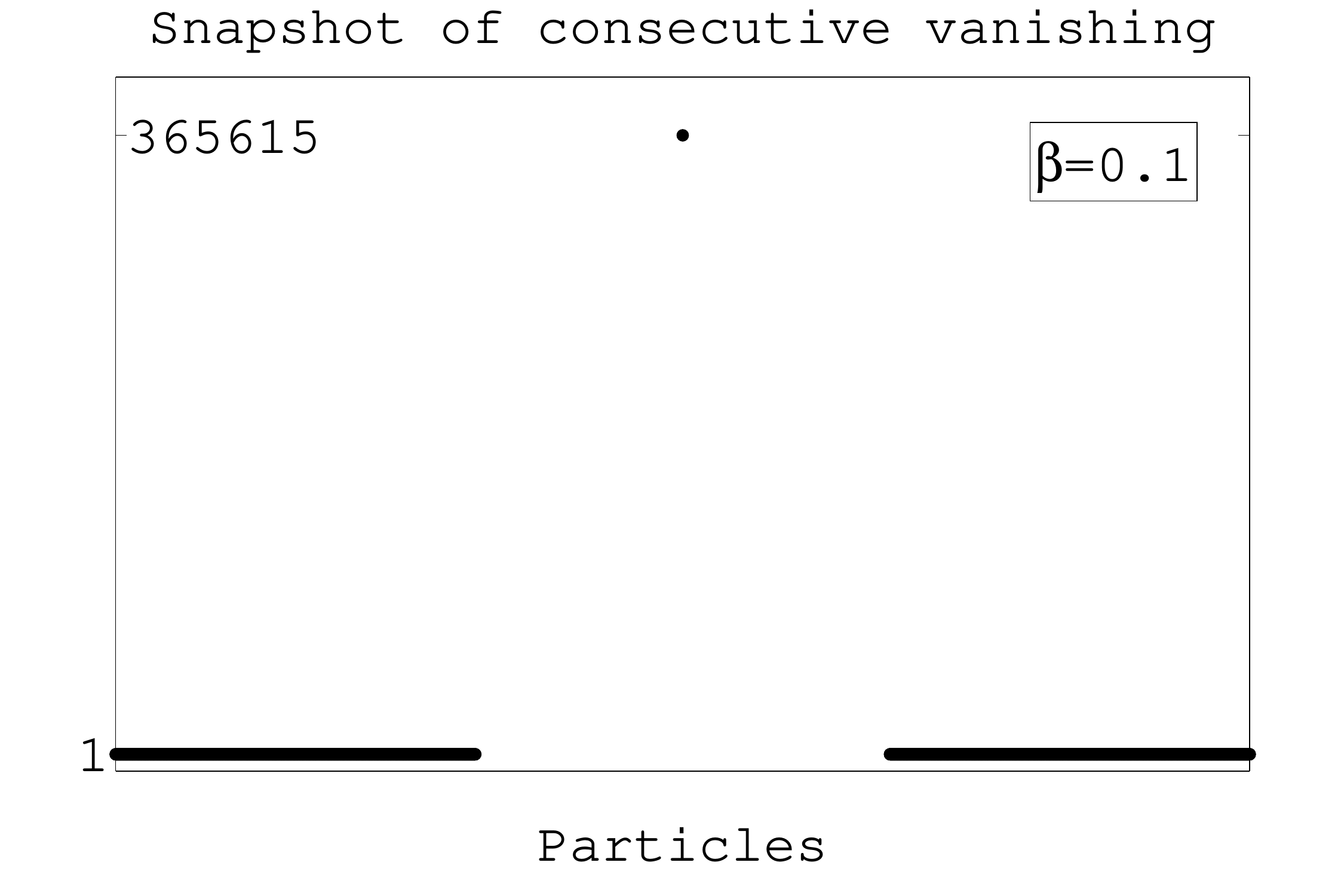}
  \includegraphics[width=.496\textwidth]{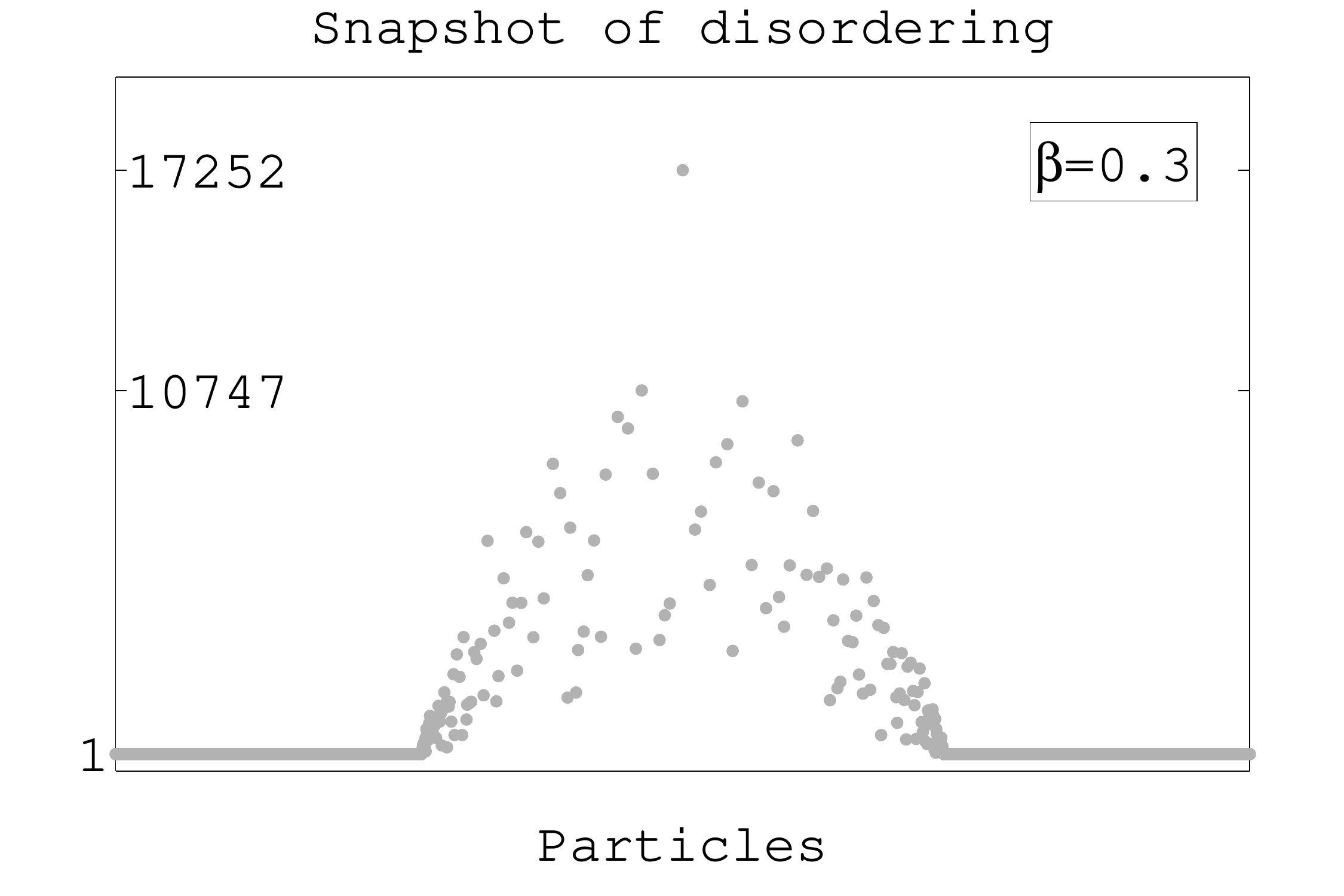}
  \caption{\emph{Top left.} Sketch of initial data where one large
    particle, here $x_0=10$, is surrounded by small particles $x_j=1$.
    \emph{Bottom row.} Snapshots of the evolutions for $\beta=0.1$
    (\emph{left}, black) and $\beta=0.3$ (\emph{right}, gray) at a
    later time. The small black particles vanish sequentially
    beginning at the large one, whereas the small gray particles
    become disordered.  \emph{Top right.}  Maximum particle size for
    both systems over time, in logarithmic scale. For $\beta=0.1$
    the largest particle grows linearly in time (represented by the
    thin dotted line) instead of adhering to the generic growth law
    \eqref{eq:intro-growth-law} (thin dashed line), while for
    $\beta=0.3$ its growth approaches the generic law (thin solid
    line). See Section \ref{sec:non-generic-example} for a more
    detailed discussion of the underlying dynamics.}
  \label{fig:non-generic}
\end{figure}

The analysis of coarsening in equations as \eqref{eq:intro-eq}
therefore comprises two problems: first, to derive an appropriate
upper bound for the (typical) particle size that is true for all
solutions, and second, to find conditions on the initial data that
guarantee the generic behavior of corresponding solutions.  The second
issue is in general very hard to deal with and any mathematical
consideration must be tailored to the specific evolution, whereas
upper coarsening estimates have been obtained for several particle
growth processes using numerical simulations, heuristic arguments or
variants of the method introduced in \cite{KoOt02}.
Usually, however, it is difficult to apply these approaches rigorously
to infinite particle systems.

In the dynamics of \eqref{eq:intro-eq}, on the other hand, finitely
many particles generically disappear in finite time, and due to
vanishings and local interactions there is also no evident equation
for the evolution of a particle size distribution.  In order to
rigorously characterize the statistical properties, our goal is thus
the mathematical analysis of infinite systems.  In the current paper,
we study well-posedness of \eqref{eq:intro-eq} for infinitely many
particles and prove existence of solutions under a density condition
on the initial distribution.  More precisely, assuming that there are
two constants $L>0$ and $d>0$ such that
\begin{equation*}
  \frac{1}{L} \sum_{j \in R} x_j(0) \geq d
\end{equation*}
for any region $R \subset \bbZ$ containing $L$ consecutive particles
we show that there are locally H\"older continuous trajectories $x_j
\colon [0,\infty) \to [0,\infty)$ that satisfy \eqref{eq:intro-eq} in
an integral sense and almost everywhere as long as $x_j>0$. In
particular, a solution contains infinitely many positive particles at
any time $t \geq 0$.

It has been observed that infinite particle systems might have
pathological solutions, for instance in \cite{MR0479206} where the
author discusses a Hamiltonian system of classical particles
interacting by means of hard core potentials. He demonstrates that for
well-prepared initial data it is possible to transfer energy from
infinity to a finite region in space via infinitely many particle
collisions in finite time, and along a similar idea we show that
solutions to \eqref{eq:intro-eq} are not uniquely determined by their
initial data.
To this end, we consider a configuration as sketched in Figure
\ref{fig:non-uniqueness-data} where large particles $x_{3j} \sim 1$
are followed by two small ones $x_{3j+1} \sim x_{3j+2} \sim R_j$ with
$R_j \to 0$ as $j \to \infty$.  Exploiting the effects of vanishings
and the localized interactions we prove that (even in finite systems)
tiny perturbations pass through the configuration and are amplified
during the evolution.  Arbitrarily small initial perturbations may
reach a size of order $1$ if the distance they travel, and thus their
total amplification, is sufficiently large.  In the infinite system,
we then obtain two different solutions by considering unperturbed data
as well as decreasing perturbations at increasing particle indices.

As a first step towards understanding the statistical properties of
\eqref{eq:intro-eq}, we provide a simple upper coarsening estimate
relating the mean particle size of a system and the time that is
necessary to transport a fraction of it from vanishing particles to
remaining ones. We also discuss part of the non-generic behavior
indicated in Figure \ref{fig:non-generic} relying on heuristic
arguments for the mass transfer in the system and numerical
investigations.

\begin{figure}
  \centering
  \includegraphics[width=.8\textwidth]{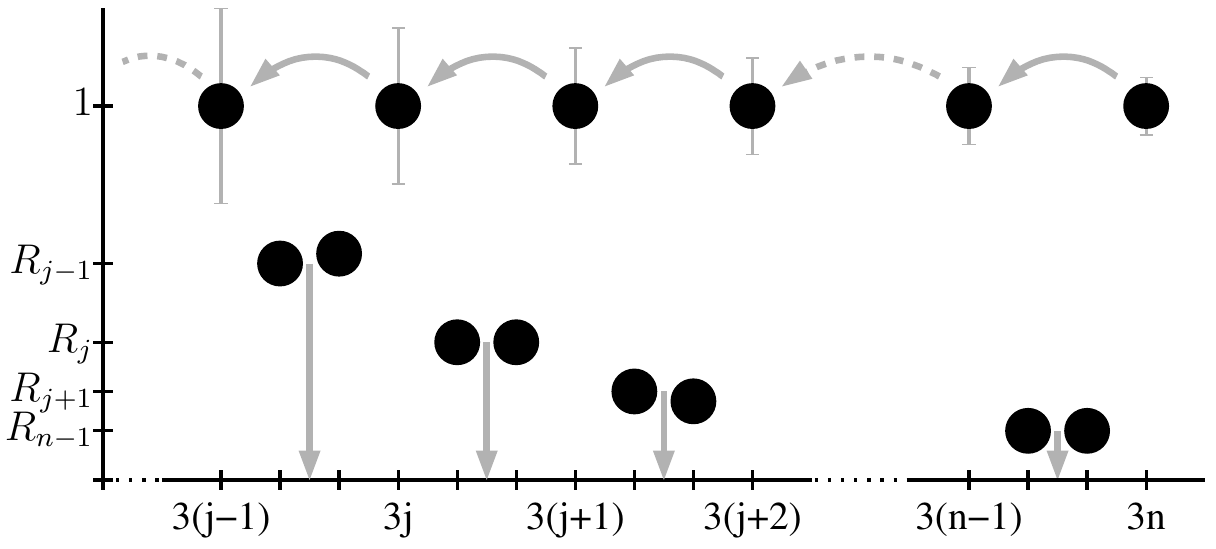}
  \caption{Sketch of the initial data for our non-uniqueness
    result. Large particles $x_{3j} \sim 1$, $j>0$ are followed by
    small ones $R_{3j+1} \sim R_{3j+2} \sim R_j$ where $R_j \to 0$ as
    $j \to \infty$.
    The small particles vanish pairwise from right to left, and
    thereby tiny perturbations of a large particle (thin gray error
    bars) are transported to the left and amplified.
    Non-uniqueness is then a consequence of the infinite number of
    particles, which allows arbitrarily large distances to be
    covered and thus amplifications of order $1$.}
  \label{fig:non-uniqueness-data}
\end{figure}

The rest of the paper is organized as follows. In Section
\ref{sec:setup} we introduce our precise setting, which in particular
comprises how we deal with vanishing particles. We also collect our
coarsening results in Sections \ref{sec:simple-coars-estim} and
\ref{sec:non-generic-example}.
The existence theorem and its proof are contained in Section
\ref{sec:existence}, while Section \ref{sec:non-uniqueness} is devoted
to non-uniqueness.

%%%%%%%%%%%%%%%%%%%%%%%%%%%%%%%%%%%%%%%%%%%%%%%%%%%%%%%%%%%%%%%%%%%%%%%%%%%%%%%
%%%%%%%%%%%%%%%%%%%%%%%%%%%%%%%%%%%%%%%%%%%%%%%%%%%%%%%%%%%%%%%%%%%%%%%%%%%%%%%
%%%
%%% Sec 2
%%%
%%%%%%%%%%%%%%%%%%%%%%%%%%%%%%%%%%%%%%%%%%%%%%%%%%%%%%%%%%%%%%%%%%%%%%%%%%%%%%%
%%%%%%%%%%%%%%%%%%%%%%%%%%%%%%%%%%%%%%%%%%%%%%%%%%%%%%%%%%%%%%%%%%%%%%%%%%%%%%%

\section{Setup and coarsening}
\label{sec:setup}

A precise description of the dynamics of \eqref{eq:intro-eq} requires
that we keep track of particles and their neighborhoods throughout
the evolution. For this reason we consider sequences $x=(x_j)$ in the
state space
\begin{equation*}
  \ell^\infty_+
  =
  \set[x=(x_j)]{0 \leq x_j \leq C \text{ for all } j \in \bbZ
    \text{ and some constant } C>0},
\end{equation*}
which include \emph{vanished} particles $x_j=0$ as well as \emph{living}
ones $x_j>0$; the term \emph{particle} can refer to the size $x_j$ as
well as the index $j$.
Since only adjacent living particles interact with each other, we use
the term \emph{neighborhood} of a particle $x_j$ to denote the two
nearest living particles $x_{\sigma_\pm(j,x)}$ where
\begin{equation}
  \label{eq:sigma_pm1}
  \sigma_-(j,x) = \sup \set[k<j]{x_k>0},
  \qquad
  \sigma_+(j,x) = \inf \set[k>j]{x_k>0};
\end{equation}
in the case that one of the sets in \eqref{eq:sigma_pm1} is empty we
set $\sup \emptyset = -\infty$, $\inf \emptyset = \infty$ and
$x_{\pm\infty}=0$. We also write
\begin{equation*}
  %\label{eq:sigma_pm2}
  \sigma_-(\infty,x) = \sup \set[k]{x_k>0},
  \qquad
  \sigma_+(-\infty,x) = \inf \set[k]{x_k>0}
\end{equation*}
for the largest and smallest particle index in a finite system,
respectively.

Given a function $g \colon [0,\infty) \to [0,\infty)$, we denote by
\begin{equation*}
  %\label{eq:sigma_laplace}
  \slaplace[\sigma]{g(x_j)}
  =
  \slaplace[\sigma(j,x)]{g(x_j)}
  =
  \Big(
    g(x_{\sigma_-(j,x)}) - 2 g(x_j) + g(x_{\sigma_+(j,x)})
  \Big)
  \cdot \chi_{\set{x_j>0}}
\end{equation*}
the \emph{living-particles-Laplacian} of the sequence $g(x)$; for
simplicity of notation we omit the indices $j$ and $x$ in $\sigma$,
$\sigma_\pm$ when the meaning is clear from the context.
In our model problem we have
\begin{equation*}
  g(s) = s^{-\beta} \quad\text{for } 0<s<\infty,
  \qquad
  g(0) = 0,
\end{equation*}
where $\beta>0$, and with the convention $0^{-\beta}=0$ we write
\eqref{eq:intro-eq} as
\begin{equation}
  \label{eq:master-eq}
  \dot x_j(t)
  =
  \slaplace[\sigma] g(x_j(t))
  =
  \slaplace[\sigma] x_j(t)^{-\beta}
  \qquad\text{for}\qquad
  j \in \bbZ, \;t>0
\end{equation}
or
\begin{equation*}
  \dot x(t)
  =
  \slaplace[\sigma] x(t)^{-\beta}
  \qquad\text{for}\qquad
  t>0.
\end{equation*}
By definition of $\slaplace[\sigma]{}$, the right hand side of
\eqref{eq:master-eq} is well-defined for all $j \in \bbZ$. In
particular, we have $\dot x_j(t)=0$ if $x_j(t)=0$ and $\dot x_j(t) =
x_{j-1}(t)^{-\beta} - 2 x_j(t)^{-\beta}$ if $x_{j-1}(t)>0$,
$x_j(t)>0$, but $\sigma_+(j,x(t))=\infty$. Finally, if $x_j(t)>0$ for
all $j \in \bbZ$, then $\slaplace[\sigma]$ is the standard discrete
Laplace operator on $\bbZ$.

% -----------------------------------------------------------------------------
% - Solution
% -----------------------------------------------------------------------------

\subsection{Notion of solution}

We now define our notion of solution. As we will prove existence by
passing to the limit in the number of particles in a system, an
integral formulation is appropriate. On the other hand, for a priori
estimates the differential equation \eqref{eq:master-eq} is more
convenient and hence we briefly discuss equivalence of both
formulations.

\begin{definition}[Integral solution]
  \label{def:integral-solution}
  For $T \in (0,\infty]$ we call $x = (x_j)_{j\in\bbZ} \colon [0,T)
  \to \ell^\infty_+$ a solution to \eqref{eq:master-eq} with initial
  data $x(0)$ if
  \begin{enumerate}
  \item each $x_j$ is continuous in $[0,T)$;
  % \item for any $t \in [0,T)$ there is at least one living particle
  %   $x_{j}(t)>0$, and we have $|\sigma_\pm(k,x(t))|<\infty$ for all $k
  %   \in \bbZ$;
  \item each function $x_j^{-\beta} \chi_{\set{x_j>0}}$, where $0^{-\beta}
    \cdot 0 = 0$, is locally integrable in $[0,T)$;
  \item for any $j \in \bbZ$ the equation
    \begin{equation}
      \label{eq:integral-eq}
      x_j(t_2) - x_j(t_1)
      =
      \int_{t_1}^{t_2} \slaplace[\sigma]{x_j(s)^{-\beta}} \,\dint{s}
    \end{equation}
    is satisfied for all $0 \leq t_1 < t_2 < T$.
  \end{enumerate}
\end{definition}

For a solution $x$ as in Definition \ref{def:integral-solution} and
$m,n \in \bbZ$, $m \leq n$, we denote by
\begin{equation}
  \label{eq:mass}
  M_{m,n}(t)
  =
  \sum_{k=m}^n x_k(t)
\end{equation}
the \emph{mass} of the particles $m,\ldots,n$ at time $t \in [0,T)$.
Using \eqref{eq:integral-eq} and the definition of
$\slaplace[\sigma]{}$, we compute 
\begin{equation}
  \label{eq:integral-change-of-mass}
  M_{m,n}(t_2) - M_{m,n}(t_1)
  =
  \int_{t_1}^{t_2} \left( x_{\sigma_-(m)}^{-\beta}
    - x_{\sigma_+(\sigma_-(m))}^{-\beta}
    - x_{\sigma_-(\sigma_+(n))}^{-\beta}
    + x_{\sigma_+(n)}^{-\beta}
  \right)
  \cdot \chi_{\set{M>0}} \,\dint{s}
\end{equation}
for $0 \leq t_1 < t_2 < T$, where $M_{m,n}(t)>0$ implies that there is
at least one living particle among $m,\ldots,n$ and thus that
$\sigma_-(m) < m \leq \sigma_+(\sigma_-(m)) \leq \sigma_-(\sigma_+(n))
\leq n < \sigma_+(n)$ in the integrand on the right hand side.

The models in the introduction assume that vanished particles do not
reappear at later times, and with the help of
\eqref{eq:integral-change-of-mass} we can show that our solutions
indeed have this property.

\begin{lemma}
  \label{lem:no-undead-particles}
  If $x_j(t_1)=0$ for some $j \in \bbZ$ and $t_1 \in [0,T)$, then
  $x_j(t)=0$ for all $t \in [t_1,T)$.
\end{lemma}

\begin{proof}
  Assume for contradiction that the claim is wrong and denote by
  \begin{equation*}
    t_2
    =
    \inf \set[t>t_1]{x_j(t)>0}
    \in
    [t_1,T)
  \end{equation*}
  the time up to which $x_j$ remains vanished. Suppose first that $m =
  \sigma_-(j,x(t_2))$ and $n = \sigma_+(j,x(t_2))$ are both finite and
  fix $t_3 \in (t_2,T)$ so that by continuity of particle sizes we
  have
  \begin{equation*}
    x_m(t)^{-\beta} \leq c,
    \quad
    x_n(t)^{-\beta} \leq c
    \quad\text{and}\quad
    x_k(t)^{-\beta} \geq 2 c, \; k=m+1,\ldots,n-1
  \end{equation*}
  for all $t \in [t_2,t_3]$ and some constant $c>0$.
  The change of $M_{m+1,n-1}(t)$ then satisfies
  \begin{equation}
    \label{eq:non-undead:change}
    %\sum_{k=m+1}^{n-1} \slaplace[\sigma]{x_k(t)^{-\beta}}
    %=
    x_m(t)^{-\beta} - x_{\sigma_+(m,x(t))}(t)^{-\beta}
    - x_{\sigma_-(n,x(t))}(t)^{-\beta} + x_n(t)^{-\beta}
    \leq
    - 2 c,
  \end{equation}
  for $t \in [t_2,t_3]$ such that $M_{m+1,n-1}(t)>0$, hence
  \eqref{eq:integral-change-of-mass} gives
  \begin{equation*}
    M_{m+1,n-1}(t)
    \leq
    - 2 c \left|\set[s \in (t_2,t)]{M_{m+1,n-1}(s)>0}\right|
  \end{equation*}
  for all $t \in [t_2,t_3]$. This is a contradiction to $0<x_j(t) \leq
  M_{m+1,n-1}(t)$ for $t$ sufficiently close to $t_2$.

  In the cases that $m=-\infty$ or $n=\infty$ or both, the
  corresponding positive term(s) in \eqref{eq:non-undead:change}
  disappear(s). Since, however, the negative contribution
  $-x_{\sigma_+(m)}(t)^{-\beta} -x_{\sigma_-(m)}(t)^{-\beta}$
  remains, we still obtain the contradiction as above.
\end{proof}

\begin{remark}[Vanishing times]
  As a consequence of Lemma \ref{lem:no-undead-particles} each
  particle $x_j$ has a unique \emph{vanishing time} $\tau_j \in [0,T)
  \cup \{\infty\}$ such that
  \begin{equation*}
    x_j(t)>0 \text{ for } t < \tau_j
    \qquad\text{and}\qquad
    x_j(t)=0 \text{ for } t \geq \tau_j,
  \end{equation*}
  where we set $\tau_j=\infty$ if and only if $x_j(t)>0$ for all $t
  \in [0,T)$.
  Moreover, the functions $t \mapsto \sigma_\pm(j,x(t))$ are monotone,
  right-continuous and change their value in vanishing times only.
\end{remark}

\begin{lemma}[Solving the ODE]
  \label{lem:solving-ode}
  For $j \in \bbZ$ let $\calV_j = \set[t]{ \sigma_\pm(j,x(t)) \text{
      is not continuous}}$ be the set of times when a neighbor of
  $x_j$ vanishes. Then $x_j$ is continuously differentiable for all $t
  \in (0,T) \sm ( \{\tau_j\} \cup \calV_j)$ and satisfies
  \eqref{eq:master-eq}.
\end{lemma}

\begin{proof}
  Fix $j \in \bbZ$. For $t > \tau_j$ the claim is trivial and for $t
  \in (0,\tau_j) \sm \calV_j$ it follows from the Fundamental Theorem
  of Calculus applied to \eqref{eq:integral-eq} in a sufficiently
  small time interval around $t$.
  More precisely, with $m=\sigma_-(j,x(t))$ and $n = \sigma_+(j,x(t))$
  fixed, continuity of $x_j$, $x_m$ and $x_n$ provides $\delta>0$ and
  $c>0$ such that
  \begin{equation*}
    \min \{ x_m(s), x_j(s), x_n(s) \colon s \in (t-\delta,t+\delta) \}
    \geq c > 0.
  \end{equation*}
  Hence, $\sigma_-(j,x(s)) \geq m$ and $\sigma_+(j,x(s)) \leq n$ for
  all $s \in (t-\delta,t+\delta)$, which implies that as functions of
  $s$ both change their value only at finitely many times in $\calV_j
  \cap (t-\delta,t+\delta)$. As $t$ is not such a time, we may
  decrease $\delta$ to obtain $\sigma_-(j,x(s)) = m$ and
  $\sigma_+(j,x(s)) = n$ for all $s \in
  (t-\delta,t+\delta)$. Therefore,
  \begin{equation*}
    \frac{x_j(t+h)-x_j(t)}{h}
    =
    \frac{1}{h}
    \int_t^{t+h} x_m(s)^{-\beta} - 2 x_j(s)^{-\beta} + x_n(s)^{-\beta}
    \,\dint{s}
  \end{equation*}
  for all $|h|<\delta$, and the claim follows as stated above.  
\end{proof}

\begin{remark}
  The preceding lemmas show that every solution $x$ has the property
  that
  \begin{enumerate}
  \item for each $j$ there exists a unique vanishing time $\tau_j$;
  \item each $x_j$ is continuously differentiable in $(0,T) \sm
    \calV$, where $\calV = \set[\tau_k]{k \in \bbZ}$, and equation
    \eqref{eq:master-eq} holds.
  \end{enumerate}
  The other way round, if $x \colon [0,T) \to \ell^\infty_+$ satisfies
  the second item above with $\calV$ replaced by an arbitrary null set
  of $[0,T)$ and Definition \ref{def:integral-solution}(1 and 2), then
  also the integral identity \eqref{eq:integral-eq} is true.
  Moreover, uniqueness of vanishing times follows directly from the
  differential equation if the latter holds almost everywhere, since
  the argument in the proof of Lemma \ref{lem:no-undead-particles} is
  easily adapted.
\end{remark}

% -----------------------------------------------------------------------------
% - Coarsening estimate
% -----------------------------------------------------------------------------

\subsection{Upper coarsening estimate}
\label{sec:simple-coars-estim}

Equation \eqref{eq:master-eq} can formally be interpreted as the
$H^{-1}$ gradient flow on living particles of the energy
\begin{equation*}
  E(x) = \frac{1}{1-\beta} \sum_{k \colon x_k>0} x_k^{1-\beta}
  \quad\text{if }
  \beta \not=1
  \qquad\text{or}\qquad
  E(x) = \sum_{k \colon x_k>0} \ln x_k
  \quad\text{if }
  \beta =1,
\end{equation*}
and such a structure has proved successful for obtaining upper
coarsening bounds in similar (finite) systems; see for instance
\cite{EsGr09,EsSl08} for the application of \cite{KoOt02} to a
discrete $H^{-1}$ gradient flow.
However, besides the fact that energy and metric tensor of infinite
systems are in general infinite, this gradient flow structure does not
seem useful here, since (even in finite systems) at each vanishing
time the local correlations in the metric tensor change and for $\beta
\geq 1$ the energy tends to negative infinity.
We thus take a different approach, which regards an upper coarsening
estimate as a lower bound on the time needed to transport a certain
amount of mass through the system and which is derived from the basic
inequality $\dot x_j \geq -2 x_j^{-\beta}$.

For illustration, suppose that $x=(x_k)_{k=-n,\ldots,n}$ is a solution
to~\eqref{eq:master-eq} with $2n+1$ nonzero initial particles whose
mass is $M = \sum_{|k| \leq n} x_k(0)$ and denote by
\begin{equation*}
  T
  =
  \inf \set[t>0]{ \sum_{x_k(t)=0} x_k(0) \geq \frac{1}{2} M}
\end{equation*}
the smallest time such that half of the initial mass has been
transported from vanished particles to living ones.
Then, there must exist a particle $x_j$ with vanishing time $\tau_j
\leq T$ and initial size at least $M/(2(2n+1))$, because otherwise we
would have $\sum_{x_k(T)=0} x_k(0) < M/2$. From $\dot x_j \geq -2
x_j^{-\beta}$ we infer that $x_j(0)^{\beta+1} \leq 2(\beta+1) \tau_j$
and hence obtain the estimate
\begin{equation*}
  \frac{1}{2(2n+1)} M
  \leq
  x_j(0)
  \leq
  \big( 2(\beta+1) \tau_j \big)^{\frac{1}{\beta+1}}
  \leq
  \left[ 2(\beta+1) \right]^{\frac{1}{\beta+1}} T^{\frac{1}{\beta+1}},
\end{equation*}
which bounds $T^{1/(\beta+1)}$ from below by the initial mean particle
size and a constant depending on $\beta$ and the fraction of mass that
is transported.

Assuming that the mean particle size initially exists, we prove a
similar estimate for infinite particle systems.

\begin{proposition}
  \label{pro:coarsening}
  Let $x \colon [0,\infty) \to \ell_+^\infty$ be a solution to
  \eqref{eq:master-eq} which initially has finite mean particle size
  \begin{equation*}
    X = \lim_{n \to \infty} X_n \in (0,\infty),\qquad
    X_n = \frac{1}{2n+1} \sum_{j=-n}^n x_j(0)
  \end{equation*}
  and let
  \begin{equation*}
    T_n = \inf \set[t>0]{ \frac{1}{2n+1}
      \sum_{\substack{j=-n,\ldots,n\\x_j(t)=0}} x_j(0)
      \geq  \frac{1}{2} X_n}
  \end{equation*}
  as well as
  \begin{equation*}
    T
    =
    \inf \set[t>0]{\liminf_{n \to \infty} \frac{1}{2n+1}
      \sum_{\substack{j=-n,\ldots,n\\x_j(t)=0}} x_j(0)
      \geq \frac{1}{2} X}.
  \end{equation*}
  Then there is a constant $c>0$ such that $\liminf_{n \to \infty} T_n
  \geq c X^{\beta+1}$ and $T \geq c X^{\beta+1}$.
\end{proposition}

In Proposition \ref{pro:coarsening} both, $T$ and $\liminf T_n$ are a
measure for the time needed to transport half of the initial mean
particle size from vanishing particles to living ones, but we do not
know which one is better or more useful.
Furthermore, the result only applies when sufficiently many particles
do indeed vanish. We expect that this is the generic case for
long-time evolutions, since the discrete backward-parabolic equation
\eqref{eq:master-eq} increases differences among particle sizes and
since small particles surrounded by larger ones disappear sufficiently
fast; see Lemma \ref{lem:vanish-small-particles}.
On the other hand, there are situations such that no or very few
particles vanish and the possible rate of coarsening is not captured
by our estimate; consider for instance a system of large regions of
equal particles and compare the mass concentration example in Section
\ref{sec:non-generic-example}.

\begin{proof}[Proof of Proposition \ref{pro:coarsening}]
  As for finite systems above we obtain $T_n \geq c X_n^{\beta+1}$,
  and taking the lower limit as $n \to \infty$ proves the first
  inequality.

  By definition of $T$ there is for any $\eps>0$ a time $t_\eps \in
  [T,T+\eps]$ such that
  \begin{equation*}
    \liminf_{n\to\infty} V_n(t_\eps)
    \geq
    \frac{1}{2} X,
    \qquad
    V_n(t) = \frac{1}{2n+1} \sum_{\substack{j=-n,\ldots,n\\x_j(t)=0}} x_j(0).
  \end{equation*}
  Hence, there is $n_0=n_0(\eps)$ such that $V_n(t_\eps) \geq X/2 -
  \eps$ for all $n > n_0$ and the argument for finite systems yields
  \begin{equation*}
    \frac{1}{2} X - \eps
    \leq
    c t_\eps^{\frac{1}{\beta+1}}
    \leq
    c (T+\eps)^{\frac{1}{\beta+1}}.
  \end{equation*}
  Sending $\eps \to 0$ finishes the proof.
\end{proof}

% -----------------------------------------------------------------------------
% - Coarsening estimate
% -----------------------------------------------------------------------------

\subsection{Example: mass concentration versus mass spreading}
\label{sec:non-generic-example}

We conclude this section with a heuristic discussion of the example
shown in Figure \ref{fig:non-generic} where initially a large $x_0(0)
\gg 1$ is surrounded by $x_j(0)=1$, $j \not= 0$.

The particle $x_0$ clearly grows during the evolution, and assuming
that it is sufficiently large we may neglect its contribution to the
particle velocities. Furthermore, due to the symmetry $x_{-j} = x_j$
for $j \in \bbZ$ it suffices to consider only the particles $x_j$, $j
\geq 0$, and the dynamics of the example are effectively governed by
\begin{equation}
  \label{eq:non-gen-1}
  \dot x_0 \sim 2 x_{1}^{-\beta},
  \qquad
  \dot x_1 \sim - 2 x_1^{-\beta} + x_2^{-\beta},
  \qquad
  \dot x_j \sim x_{j-1}^{-\beta} - 2 x_j^{-\beta} + x_{j+1}^{-\beta}
  \quad\text{for}\quad j>1.
\end{equation}
We let
\begin{equation*}
  x_j(t)
  =
  1 + \frac{u_j(s)}{\beta},
  \qquad
  s = \beta t
\end{equation*}
and linearize
\begin{equation*}
  x_j^{-\beta}(t)
  =
  \exp(-\beta\ln x_j(t))
  \sim
  e^{- u_j(s)},
\end{equation*}
so that \eqref{eq:non-gen-1} becomes
\begin{equation}
  \label{eq:non-gen-2}
  \tderiv{s} u_0
  \sim
  2 e^{-u_1},
  \qquad
  \tderiv{s} u_1
  \sim
  - 2 e^{-u_1} + e^{-u_2},
  \qquad
  \tderiv{s} u_j
  \sim
  e^{- u_{j-1}} - 2 e^{-u_j} + e^{-u_{j+1}}
\end{equation}
with initial data $u_0(0) \gg 1$ and $u_j(0)=0$, $j>0$.
The initial conditions immediately show that
\begin{equation*}
  \tderiv{s}u_0(0) \sim 2 > 0,
  \qquad
  \tderiv{s}u_1(0) \sim - 1 < 0,
  \qquad
  \tderiv{s}u_j(0) \sim 0
  \quad\text{for}\quad j>1,
\end{equation*}
and by induction it is easy to see that
\begin{equation*}
  (-1)^j \left(\tderiv{s}\right)^j u_j(0) \sim 1 > 0,
  \qquad
  \left(\tderiv{s}\right)^k u_j(0) = 0
  \quad\text{for}\quad
  k>j>1.
\end{equation*}
Indeed, for $j>0$ equation \eqref{eq:non-gen-2} reads $\deriv[u_j]{s}
\sim \laplace e^{-u_j}$ where $\laplace$ is the discrete Laplacian on
$\bbZ$, and taking the $k$-th time derivative we obtain
\begin{equation*}
  \left(\tderiv{s}\right)^k u_j
  \sim
  \laplace \left( \tderiv{s} \right)^{k-1} e^{-u_j},
\end{equation*}
which enables the induction argument. As a consequence, we find
$(-1)^j u_j(s) > 0$ for small positive times $s$, and using these
inequalities in \eqref{eq:non-gen-2}, we conclude that they remain
true for all $s$.
Moreover, a comparison argument for $u_{j+2}-u_j$ shows that $u_0 \geq
u_2 \geq u_4 \ldots$ and $\ldots \geq u_5 \geq u_3 \geq u_1.$
In terms of the particles $x_j$ this mean that those with an even
index $j$ grow while those with an odd index shrink as long as
\eqref{eq:non-gen-2} is a valid approximation of the original
dynamics.

We now consider two cases, namely $\beta$ being very small and very
large, respectively.
For small $\beta$, the requirement $u_j = O(\beta)$ for $j>1$, which
means that none of the corresponding particles has vanished, and
\eqref{eq:non-gen-2} yield
\begin{equation*}
  \tderiv{s} u_0 \sim 2
  \qquad\text{and}\qquad
  \tderiv{s} u_1 \sim -1.
\end{equation*}
The key findings are the following.
\begin{enumerate}
\item The particles $x_1$ and $x_{-1}$ vanish when $u_1$ reaches a
  value of order $-\beta$, and that happens at a time of order $s \sim
  \beta$ or $t \sim 1$.
\item On this time scale the particles $x_j$, $|j|>1$ hardly change
  size due to $|u_j| = O(\beta)$.  Therefore, after $x_1$ and
  $x_{-1}$ have vanished, we are in the same situation as initially
  where $x_0$ is surrounded by small particles $x_j \approx
  1$, now indexed by $|j|>1$.
\item The mass $x_1(0) + x_{-1}(0) = 2$ is almost completely
  transported to the particle $x_0$, and the latter grows linearly in
  the time $t$.
\end{enumerate}
With regard to the second and third item we refer to the case of small
$\beta$ as \emph{sequential vanishing} or \emph{mass concentration}.

In the second case, when $\beta$ is large, we use that $-u_1$ has to
become large in order that $x_1$ vanishes. Due to $u_2 \geq 0$, we
have
\begin{equation*}
  \tderiv{s} u_1
  \sim
  - 2 e^{-u_1} + e^{-u_2}
  \lesssim
  - 2 e^{-u_1},
\end{equation*}
which implies that the time for $u_1$ to reach a size of order
$-\beta$ is of order $s \sim 1$.
Moreover, we have
\begin{equation*}
  \tderiv{s} u_2
  \sim
  e^{-u_1} - 2 e^{-u_2} + e^{-u_3}
  \sim
  - \tfrac{1}{2} \tderiv{s} u_1 - \tfrac{3}{2} e^{-u_2} + e^{-u_3}
  \gtrsim
  - \tfrac{1}{2} \tderiv{s} u_1 - \tfrac{3}{2}
\end{equation*}
due to $u_2 \geq 0 \geq u_3$, and we conclude that
\begin{equation}
  \label{eq:non-gen-4}
  u_2(s) \gtrsim - \tfrac{1}{2} u_1(s) - \tfrac{3}{2} s.
\end{equation}
The main observations are now the following.
\begin{enumerate}
\item The particles $x_1$ and $x_{-1}$ vanish at a time of order $s
  \sim 1$ or $t \sim \tfrac{1}{\beta}$.
\item The change of the other particles is not known to be smaller. In
  fact, by \eqref{eq:non-gen-4} the change of $u_2$ is of the same
  order as the change of $u_1 \sim -\beta$ at times of order $s \sim
  1$.
\item Due to \eqref{eq:non-gen-4} and the equation for $u_0$ in
  \eqref{eq:non-gen-2}, the mass $x_1(0)=1$ is split between $x_0$ and
  $x_j$, $j>1$ with at least half of it (up to an error $1/\beta$)
  going to $x_j$, $j>1$. Similarly, the mass of $x_{-1}$ is split
  between $x_0$ and $x_j$, $j<-1$.
\end{enumerate}
In view of the third item we call this case \emph{mass spreading}.

\begin{figure}
  \centering
  \hspace{1ex}
  \includegraphics[height=.46\textwidth]{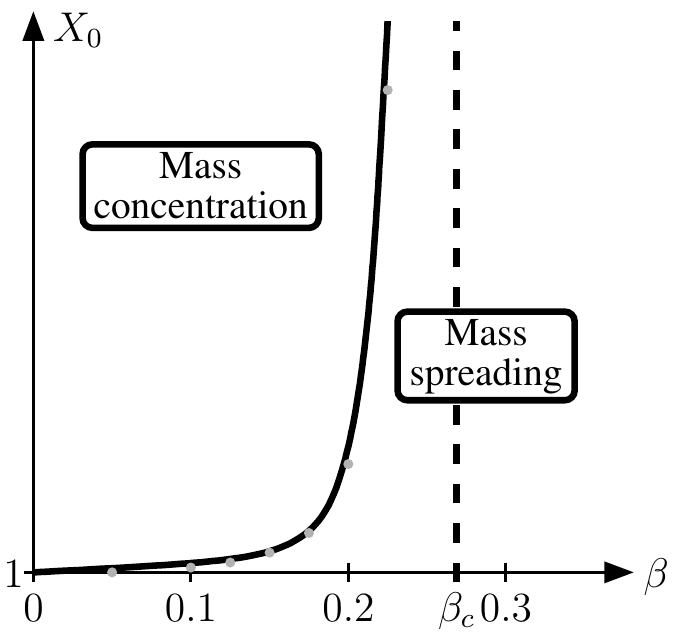}
  \hfill
  \begin{minipage}[b][.46\textwidth][t]{.42\textwidth}
    \begin{tabular}[b]{|r|r|r|}\hline
      $\beta$ & \multicolumn{2}{|c|}{$X_0$}\\ \hline \hline
      & mass & mass \\
      & spreading   & concentration \\ \hline
      .050 &    1.72 &    1.73 \\
      .100 &    8.00 &    8.06 \\
      .125 &   28.00 &   28.20 \\
      .150 &   84.60 &   85.40 \\
      .175 &  188.00 &  189.00 \\
      .200 &  537.00 &  541.00 \\
      .225 & 2456.00 & 2470.00 \\
      \hline
    \end{tabular}
  \end{minipage}
  \hspace{1ex}
  \caption{``Phase diagram'' for the example in Section
    \ref{sec:non-generic-example}. There is a critical $\beta_c$,
    which according to simulations lies between $0.26$ and $0.27$,
    such that for $\beta>\beta_c$ always mass spreading occurs. If
    $\beta<\beta_c$, on the other hand, we have mass spreading for
    small $X_0$ and mass concentration for large $X_0$. Some numerical
    values are shown in the table and as gray dots in the diagram.}
  \label{fig:phase-diagram}
\end{figure}

Our heuristic observations are confirmed by numerical simulations.  In
fact, the latter even yield a refined picture which suggests that
there is a critical value $\beta_c \approx 0.27$ such that for initial
data $x_0 = X_0>1$ and $x_j=1$, $|j|>1$ one of the following three
cases holds; compare Figure \ref{fig:phase-diagram}.
\begin{enumerate}
\item If $\beta>\beta_c$ then for all $X_0>1$ mass spreading occurs.

\item For each $\beta<\beta_c$ there is an $X_\beta$ such that for
  $X_0<X_\beta$ mass spreading takes place.
  
\item For $\beta<\beta_c$ and $X_0>X_\beta$ mass concentration arises.
\end{enumerate}
Moreover, our numerical results also indicate that in the mass
spreading case the particle sizes become disordered over time as
indicated in Figure \ref{fig:non-generic}.
Currently, however, we are not able to provide good heuristics for or
a suitable notion of being disordered.

%%%%%%%%%%%%%%%%%%%%%%%%%%%%%%%%%%%%%%%%%%%%%%%%%%%%%%%%%%%%%%%%%%%%%%%%%%%%%%% 
%%%%%%%%%%%%%%%%%%%%%%%%%%%%%%%%%%%%%%%%%%%%%%%%%%%%%%%%%%%%%%%%%%%%%%%%%%%%%%%
%%%
%%% Existence
%%%
%%%%%%%%%%%%%%%%%%%%%%%%%%%%%%%%%%%%%%%%%%%%%%%%%%%%%%%%%%%%%%%%%%%%%%%%%%%%%%%
%%%%%%%%%%%%%%%%%%%%%%%%%%%%%%%%%%%%%%%%%%%%%%%%%%%%%%%%%%%%%%%%%%%%%%%%%%%%%%%

\section{Existence of solutions for initial data with positive
  density}
\label{sec:existence}

In this section we prove existence of solutions for suitably
distributed initial data which contain no large regions in $\bbZ$ of
vanished or very small particles.
We also suppose that initially there are no vanished particles.

\begin{assumption}
  \label{ass:initial-data}
  All particles in the initial configuration $x=(x_j) \in
  \ell^\infty_+$ are positive. Moreover, there exist two constants
  $d>0$ and $L>0$ such that $\bbZ$ can be subdivided into regions of
  length $L$ with mass density at least $d$; that is, there is $k_0
  \in \set{0,1,\ldots,L-1}$ such that for
  \begin{equation*}
    R(k) := \set{k_0+k L,\ldots,k_0+(k+1)L-1},
    \qquad
    k \in \bbZ
  \end{equation*}
  we have
  \begin{equation}
    \label{eq:ass-box-density}
    \frac{1}{L} \sum_{j \in R(k)} x_{j}
    \geq
    d.
  \end{equation}
  Without loss of generality we assume $k_0=0$.
\end{assumption}

\begin{remark}[Traps]
  Equation \eqref{eq:ass-box-density} implies the existence of indices
  $j_k \in R(k)$ such that $x_{j_k} \geq d$ for all $k \in \bbZ$.  In
  particular, we have $j_k \to \pm \infty$ as $k \to \pm \infty$ and
  $0 < j_{k+1}-j_k < 2L$.
  We call the $x_{j_k}$ and $j_k$ \emph{traps}, and given $j \in
  \bbZ$ we denote by $j_\pm$ the two nearest traps (unequal to $j$
  itself), that is $j_- = \max \set{j_k<j}$ and $j_+ = \min
  \set{j_k>j}$; obviously, we have
  \begin{equation*}
    j_+ - j_- \leq 4 L,
    \qquad
    j_+ - j \leq 2 L,
    \qquad\text{and}\qquad
    j - j_- \leq 2 L
  \end{equation*}
  for all $j \in \bbZ$.
  Finally, if $x$ is a particle configuration with traps $(j_k)$ as
  above, then there is at least one trap among $2 L$ consecutive
  particles. Thus, the existence of traps is up to a factor for $L$
  equivalent to the density inequality \eqref{eq:ass-box-density}.
\end{remark}

An example for a particle ensemble that satisfies Assumption
\ref{ass:initial-data} is a positive $x \in \ell_+^\infty$ such that
the one-sided mean particle sizes
\begin{equation*}
  \lim_{n \to \infty} \frac{1}{n+1} \sum_{j=0}^n x_j,
  \qquad
  \lim_{n \to \infty} \frac{1}{n+1} \sum_{j=-n}^0 x_j
\end{equation*}
exist and are positive. To see this, assume for contradiction that
there are no appropriate traps, that is, for any $\delta>0$ the
indices $(j_k)$ of particles $x_{j_k}\geq\delta$ either are bounded
from above or below or have unbounded distance $j_{k+1}-j_k$ as $k \to
+\infty$ or $k\to-\infty$.
If the latter happens as $k\to\infty$, we can find an index
$k_0=k_0(\delta)$ such that $j_{k+1}-j_k \geq 1/\delta$ for all $k
\geq k_0$, which implies $\#\set[j_k]{j_{k_0} \leq j_k \leq n} \leq
(n+1) \delta$ for $n>j_{k_0}$. Together with $\#\set[j_k]{0 \leq j_k
  \leq j_{k_0}} \leq (j_{k_0}+1)$ and $x_j \leq C$ for all $j \in
\bbZ$ we obtain
\begin{align*}
  \frac{1}{n+1} \sum_{j=0}^n x_j
  &=
  \frac{1}{n+1}
  \Bigg(
  \sum_{\substack{j=0,\ldots,n\\x_j\leq\delta}} x_j
  +
  \sum_{\substack{j=0,\ldots,j_{k_0}\\x_j>\delta}} x_j
  +
  \sum_{\substack{j=j_{k_0}+1,\ldots,n\\x_j>\delta}} x_j
  \Bigg)
  \\
  &\leq
  \delta + \frac{C (j_{k_0}+1)}{n+1} + C\delta
\end{align*}
for $n > j_{k_0}$, and sending $n \to \infty$ gives a contradiction as
$\delta>0$ is arbitrary. The other cases are similar.

The sequential vanishing example in Section
\ref{sec:non-generic-example} provides strong evidence that, at least
for small $\beta$, a condition like Assumption \ref{ass:initial-data}
is necessary for bounded solutions to exist.  Consider an initial
configuration that contains infinitely many regions $R_j$ where a
large particle of size $1$ is surrounded by $L_j$ particles of size
$\eps_j$ on each side, and assuming that $x$ is a solution for such
initial data, rescale time and particle size by means of $y(t) =
x(\eps_j^{\beta+1} t)/\eps_j$. Then, $y$ still satisfies equation
\eqref{eq:master-eq}, particles of size $x=\eps_j$ are transformed to
$y=1$ and those of size $x=1$ to $y=1/\eps_j$.  The particles $y$ in
$R_j$ thus resemble exactly the sequential vanishing situation in
Section \ref{sec:non-generic-example}, provided that $\beta$ and
$\eps_j$ are small and $L_j$ is sufficiently large so that the
particles around the large one in $R_j$ are not influenced by the
adjacent regions $R_{j \pm 1}$.  In particular, the large particle in
$R_j$ gains all the mass of the nearby vanishing ones and its size $y$
grows linearly in time.  In terms of $x$, the large particle grows
like
\begin{equation*}
  \eps_j^{-(\beta+1)} t \eps_j
  =
  t \eps_j^{-\beta},
\end{equation*}
which means that we can arrange an increase of order
$\eps_j^{-\beta/2}$ in a time of order $\eps_j^{\beta/2}$ if $L_j
\eps_j \gg \eps_j^{-\beta/2}$.  Hence, by $\eps_j \to 0$ and choosing
$L_j$ appropriately we obtain an initial configuration $x \in
\ell^\infty$ so that some particles become arbitrarily large in
arbitrarily small time.

% -----------------------------------------------------------------------------
% - A priori estimates
% -----------------------------------------------------------------------------

\subsection{Local-in-time a priori estimates}
\label{sec:a-priori-estimates}

We now derive key properties of solutions with initial data as in
Assumption \ref{ass:initial-data}.
In the following, we let $x \colon [0,T) \to \ell^\infty$ be such a
solution as in Definition \ref{def:integral-solution} and for
simplicity of notation we consider times $0 \leq t < T$ without
indicating it explicitly.
Throughout the paper, we denote by $C$ and $c$ generic constants that
depend on $\beta$ only.

According to Lemma \ref{lem:solving-ode} the differential equation
\eqref{eq:master-eq} is true for all but the vanishing times $\calV =
\set[\tau_k]{k \in \bbZ}$, and dropping the non-negative contributions
of $\slaplace[\sigma]{}$ gives
\begin{equation*}
  \dot x_j(t)
  \geq
  -2 x_j(t)^{-\beta}
\end{equation*}
whenever $x_j>0$ and $t \not\in\calV$. Integrating from $t_1$ to
$t_2$, where $0 \leq t_1 < t_2 < \tau_j$, we find
\begin{equation}
  \label{eq:simple-lower-bound-1}
  x_j(t_2)^{\beta+1} - x_j(t_1)^{\beta+1}
  \geq
  -2 (\beta+1) (t_2-t_1)
\end{equation}
and conclude that
\begin{equation}
  \label{eq:simple-lower-bound-2}
  x_j(t_2)
  \geq
  \Big( x_j(t_1)^{\beta+1} - 2 (\beta+1)(t_2-t_1) \Big)^{\frac{1}{\beta+1}}
\end{equation}
as long as the difference on the right hand side is non-negative.
Hence, we obtain the \emph{half-life estimate}
\begin{equation*}
  x_j(t) \geq \frac{x_j(t_1)}{2}
  \qquad\text{for all}\qquad
  t \leq t_1 + C x_j(t_1)^{\beta+1}
  %\;
  %C = \frac{1-2^{-(\beta+1)}}{2(\beta+1)} 
\end{equation*}
and in particular the \emph{persistence estimate} for traps
\begin{equation*}
  x_{j_k}(t) \geq \frac{d}{2}
  \qquad\text{for all}\qquad
  t \leq T^*:= \min ( C d^{\beta+1}, T ).
\end{equation*}
The latter implies
\begin{equation*}
  \sigma_-(j,x(t)) \geq j_{-},
  \qquad
  \sigma_+(j,x(t)) \leq j_{+}
  \qquad\text{for all}\qquad
  t \leq T^*,
  \;
  j \in \bbZ,
\end{equation*}
and therefore the number of different particles that up to time $T^*$
can appear as neighbors of any fixed particle $x_j$ is at most $4 L$.

Next, we consider the local masses $M_{m,n}(t) = \sum_{k=m}^n x_k(t)$,
$m<n$ as defined in \eqref{eq:mass}. Similar to
\eqref{eq:integral-change-of-mass}, we have
\begin{equation}
  \label{eq:change-of-mass}
  \tderiv{t} M_{m,n}(t)
  =
  x_{\sigma_-(m)}(t)^{-\beta}
  -
  x_{\sigma_-(\sigma_+(m))}(t)^{-\beta}
  -
  x_{\sigma_+(\sigma_-(n))}(t)^{-\beta}
  +
  x_{\sigma_+(n)}(t)^{-\beta}
\end{equation}
whenever $M_{m,n}(t)>0$ and $t \not\in \calV$, and we infer that
\begin{equation}
  \label{eq:change-of-mass-2}
  \tderiv{t} M_{m,n}(t)
  \leq
  x_{\sigma_-(m)}(t)^{-\beta}
  +
  x_{\sigma_+(n)}(t)^{-\beta}
  -
  2 M_{m,n}(t)^{-\beta}
\end{equation}
since $0 < x_{\sigma_-(\sigma_+(m))}(t) \leq M_{m,n}(t)$ and $0 <
x_{\sigma_+(\sigma_-(n))}(t) \leq M_{m,n}(t)$.
Furthermore, if $j_k < j_l$ are two traps, the persistence estimate
and integration of \eqref{eq:change-of-mass} provide
\begin{equation}
  \label{eq:mass-traps-lower}
  M_{j_k,j_l}(t_2)
  \geq
  M_{j_k,j_l}(t_1)
  - 2 \left( \frac{d}{2} \right)^{-\beta} (t_2-t_1)
\end{equation}
and
\begin{equation}
  \label{eq:mass-traps-upper}
  M_{j_k+1,j_l-1}(t_2)
  \leq
  M_{j_k+1,j_l-1}(t_1)
  + 2 \left( \frac{d}{2} \right)^{-\beta} (t_2-t_1)
\end{equation}
for all $0 \leq t_1 < t_2 < T^*$. 
From the inequalities \eqref{eq:mass-traps-lower} and
\eqref{eq:mass-traps-upper} we draw several conclusions that are
essential ingredients of our existence result.
First, the lower bound \eqref{eq:mass-traps-lower} implies that the
mass density inequality \eqref{eq:ass-box-density} in Assumption
\ref{ass:initial-data} is stable up to time $T^*$ in the following
sense.

\begin{lemma}
  \label{lem:box-dens-stability}
  For any $q \in \bbN$, $q>1$ and any $k \in \bbZ$ the mass density in
  the region
  \begin{equation*}
    R_q(k)
    :=
    \set{q k L, \ldots, q (k+1) L -1}
    =
    R(q k) \cup \cdots \cup R(q(k+1)-1)
  \end{equation*}
  of length $q L$ satisfies
  \begin{equation*}
    \frac{1}{q L}
    \sum_{j \in R_q(k)} x_j(t)
    \geq
    \left( 1 - \frac{1+C/L}{q} \right) d
  \end{equation*}
  for all $t \in [0,T^*]$.
\end{lemma}

\begin{proof}
  Each of the $q$ regions $R(q k), \ldots, R(q(k+1)-1)$ of Assumption
  \ref{ass:initial-data} contains a trap, and denoting by $k^+ \in
  R(q(k+1)-1)$ and $k^- \in R(q k)$ the largest and smallest of these
  we infer from \eqref{eq:mass-traps-lower} that
  \begin{equation*}
    \sum_{ j \in R_q(k) } x_j(t)
    \geq
    M_{k^-,k^+}(0) - C d^{-\beta} t
    \geq
    M_{k^-,k^+}(0) - C d
  \end{equation*}
  for all $t \leq T^* \leq C d^{\beta+1}$. Moreover, Assumption
  \ref{ass:initial-data} also implies
  \begin{equation*}
    M_{k^-,k^+}(0)
    \geq
    \sum_{n=q k +1}^{q(k+1)-2} \sum_{j \in R(n)} x_j
    \geq
    (q-1) L d,
  \end{equation*}
  and combining both inequalities finishes the proof.
\end{proof}

The key point of Lemma \ref{lem:box-dens-stability} is the trade-off
between the lower bound and the number of regions in the density
estimate: by choosing $q$ sufficiently large, we can ensure that the
initial bound $d$ is almost preserved.

From the upper mass estimate \eqref{eq:mass-traps-upper} we derive
H\"older continuity and boundedness of the particle trajectories.

\begin{lemma}
  %\label{lem:holder-continuity}
  We have
  \begin{equation*}
    |x_j(t_2)-x_j(t_1)|
    \leq
    C (L + 1) |t_2-t_1|^{\frac{1}{\beta+1}}
  \end{equation*}
  for any $j \in \bbZ$ and all $0 \leq t_1 < t_2 \leq T^*$.
\end{lemma}

\begin{proof}
  Given $j \in \bbZ$, it suffices to consider $t_2<T^*$ such that
  $x_j(t_2)>0$.
  Using \eqref{eq:mass-traps-upper} with the two
  nearest traps $j_\pm$ and $T^* \leq C d^{\beta+1}$ we have
  \begin{align*}
    M_{j_{-}+1,j_{+}-1}(t_2)
    &\leq
    M_{j_{-}+1,j_{+}-1}(t_1) + 2(d/2)^{-\beta} {T^*}^\frac{\beta}{\beta+1}
    (t_2-t_1)^\frac{1}{\beta+1}
    \\
    &\leq
    M_{j_{-}+1,j_{+}-1}(t_1)
    + C (t_2-t_1)^\frac{1}{\beta+1},
  \end{align*}
  and separating $x_j$ from both mass terms yields
  \begin{equation}
    \label{eq:holder-1}
    x_j(t_2) - x_j(t_1)
    \leq
    \sum_{\substack{j_-<k<j_+\\k \not= j}} \big( x_k(t_1)-x_k(t_2) \big)
    + C(t_2-t_1)^\frac{1}{\beta+1}.
  \end{equation}
  We split the sum in \eqref{eq:holder-1} into large and small
  particles, respectively, and consider the summation over
  \begin{align*}
    U_L
    &=
    \set[k \in \{j_{-}+1,\ldots,j_{+}-1\}\sm\{j\}]{x_k(t_1)>\eps},
    \\
    U_S
    &=
    \set[k \in \{j_{-}+1,\ldots,j_{+}-1\}\sm\{j\}]{x_k(t_1)\leq\eps},
  \end{align*}
  separately, where $\eps$ is defined by $\eps^{1+\beta} =
  4(\beta+1)(t_2-t_1)$.
  For the small particles we simply estimate
  \begin{equation}
    \label{eq:holder-2}\noref{eq:holder-2}
    \sum_{k \in U_S} x_k(t_1)-x_k(t_2)
    \leq
    \eps (j_+-j_-)
    \leq
    C L (t_2-t_1)^{\frac{1}{\beta+1}}
  \end{equation}
    using $x_k(t_2) \geq 0$, $j_+-j_- \leq 2L$ and the definition of
  $\eps$.
  For $k \in U_L$, on the other hand, we find $x_k(t_1)^{\beta+1} -
  2(\beta+1)(t_2-t_1) > 0$, hence we may use
  \eqref{eq:simple-lower-bound-2} and the elementary inequality
  $(a^{\beta+1} - b^{\beta+1}) \geq (a-b)^{\beta+1}$ for $\beta>0$ and
  $0 \leq b \leq a$ to get
  \begin{equation*}
    x_k(t_2)
    \geq
    x_k(t_1) - C (t_2-t_1)^{\frac{1}{\beta+1}}
  \end{equation*}
  and conclude
  \begin{equation}
    \label{eq:holder-3}
    \sum_{k \in U_L} x_k(t_1)-x_k(t_2)
    \leq
    C L (t_2-t_1)^{\frac{1}{\beta+1}}.
  \end{equation}
  Combination of \eqref{eq:holder-1}--\eqref{eq:holder-3} yields the
  H\"older estimate in the case $x_j(t_2) \geq x_j(t_1)$, while for
  $x_j(t_2) < x_j(t_1)$ the result is an immediate consequence of
  \eqref{eq:simple-lower-bound-1} and once more the elementary
  inequality $(a^{\beta+1} - b^{\beta+1}) \geq (a-b)^{\beta+1}$.
\end{proof}

An upper bound on the particle sizes now follows from either the mass
estimate \eqref{eq:mass-traps-upper} or the H\"older inequality.

\begin{corollary}
  %\label{cor:upper-bound}
  We have
  \begin{equation*}
    x_j(t)
    \leq
    M_{j_-+1,j_+-1}(0) + C d
    \qquad\text{and}\qquad
    x_j(t)
    \leq
    x_j(0) + C L d
  \end{equation*}
  for all $t \in [0,T^*]$ and any $j \in \bbZ$.
\end{corollary}

Finally, we characterize the behavior of vanishing
particles. Inequality \eqref{eq:simple-lower-bound-1} is in case
$\tau_j<\infty$ also true for $t_2 \geq \tau_j$ and immediately
provides
\begin{equation}
  \label{eq:vanishing-upper-bound}
  x_j(t)
  \leq
  C (\tau_j-t)^{\frac{1}{\beta+1}}
  \qquad\text{for all}\qquad
  0 \leq t \leq \tau_j,
\end{equation}
which can be understood in two ways. On the one hand,
\eqref{eq:vanishing-upper-bound} is an upper bound on the vanishing
particle $x_j(t)$ in terms of $\tau_j-t$, and on the other hand,
setting $t=0$ gives a lower bound on the vanishing time $\tau_j$ in
terms of $x_j(0)$.
In the following two lemmas we prove the opposite assertion by showing
first that the size of a vanishing particle is estimated from below by
a power law and second that sufficiently small particles do indeed
vanish.

\begin{lemma}
  \label{lem:vanish-lower-power-law}
  There is a constant $K=K(\beta,d,L)$ such that
  \begin{equation*}
    x_j(t) \geq K (\tau_j-t)^{\frac{1}{\beta+1}}
  \end{equation*}
  for all $t \in [0,\tau_j]$ and any $j \in \bbZ$ with $\tau_j \leq T^*$.
\end{lemma}

\begin{proof}
  We prove the claim by contradiction. If it is wrong there exist
  solutions $x^n: [0,T^n) \to \ell^\infty_+$, $n \in \bbN$, such that
  each has a particle $j^n$ with vanishing time $\tau^n_{j^n} \leq
  \min(T^n, C d^{\beta+1})$ and a \emph{bad time} $t_b^n <
  \tau^n_{j^n}$ with
  \begin{equation*}
    \frac{x^n_{j^n}(t_b^n)}{\big(\tau^n_{j^n}-t_b^n\big)^\frac{1}{\beta+1}}
    \to
    0
    \qquad\text{as}\qquad
    n \to \infty.
  \end{equation*}
  We may shift the particle index of $x^n$ so that $0=j^n_- < j^n <
  j^n_+ \leq 4 L$ and then select a subsequence of $n \to \infty$ in
  which $j^n$ and $j^n_\pm$ are constant.
  In the following, we restrict our attention to this subsequence of
  $n$ and to the particles $k = 0,\ldots,j^n_+$; for simplicity of
  notation we drop the index $n$ in the notation.

  Rescaling time and size we consider
  \begin{equation*}
    y_k(s)
    =
    \frac{x_k(t)}{\big(\tau_j-t_b\big)^\frac{1}{\beta+1}},
    \qquad
    t
    =
    \tau_j + (\tau_j-t_b) s,
  \end{equation*}
  whereby the bad times $t_b$ are translated into $s=-1$ and the
  vanishing times $\tau_k$ of $x_k$ into vanishing times $\sigma_k$ of
  $y_k$; in particular, we have $\sigma_j = 0$ and deduce
  \begin{equation*}
    y_j(s)>0
    \text{ for all } s \in [-1,0),
    \qquad
    y_j(-1) \to 0
    \text{ as } n \to \infty.
  \end{equation*}
  It is clear that $\deriv{s} y_k(s) = \slaplace[\sigma]
  y_k(s)^{-\beta}$ whenever $x_k(t)$ is differentiable at the time $t$
  corresponding to $s$. Moreover, since $x_k$ is H\"older continuous
  we also obtain
  \begin{equation*}
    %\label{eq:vanish-rescaled-holder}
    |y_k(s_2)-y_k(s_1)| \leq C |s_2-s_1|^\frac{1}{\beta+1}
  \end{equation*}
  for all $s_1,s_2 \in [-1,0]$ and $k=0,\ldots,j_+$, where $C$ depends
  on $\beta$ and $L$.

  To achieve a contradiction we study the rescaled mass in a
  neighborhood of $j$ in which the particles vanish as fast as $x_j$.
  For that purpose, we let $B \subset \set{1,\ldots,j_+-1}$ be the
  largest set of consecutive indices containing $j$ such that
  \begin{equation*}
    \liminf_{n \to \infty} y_k(-1) = 0
    \qquad
    \text{for all } k \in B;
  \end{equation*}
  writing $B = \set{m+1,\ldots,l-1}$ and restricting ourselves to a
  subsequence of $n$ we may assume that the lower limits are attained
  as limits and that $y_m(-1) \geq c$, $y_l(-1) \geq c$ for all $n$
  and some constant $c>0$. H\"older continuity then implies
  \begin{equation*}
    y_k(s)
    \geq
    c - C \delta^{\frac{1}{\beta+1}}
    \geq
    \frac{c}{2}
    \qquad
    \text{for } k=m,l,
  \end{equation*}
  $s \in [-1,-1+\delta)$ and sufficiently small $\delta > 0$, whereas
  the rescaled mass $M(s) = \sum_{k \in B} y_k(s)$ satisfies
  \begin{equation*}
    M(s)
    \leq
    M(-1) + 4 L C \delta^{\frac{1}{\beta+1}}
    \leq
    2^{-\frac{1}{\beta}} \frac{c}{2}
  \end{equation*}
  for $s \in [-1,-1+\delta)$, sufficiently small $\delta \in (0,1)$ and
  all sufficiently large $n$.
  Choosing $\delta=\delta(\beta,L,c)$ and $n=n(\beta,L,c,\delta)$ so
  that these two inequalities are true, we obtain
  \begin{equation*}
    \tderiv{s} M(s)
    =
    \sum_{\substack{k \in B\\y_{k}(s)>0}} \laplace_\sigma y_k(s)
    \leq
    y_m(s)^{-\beta} + y_l(s)^{-\beta} - 2 M(s)^{-\beta}
    \leq
    - M(s)^{-\beta}
  \end{equation*}
  for all $s \in (-1,-1+\delta) \sm \set[\sigma_k]{k \in B}$ and
  therefore
  \begin{equation*}
    0
    <
    y_j(s)^{\beta+1}
    \leq
    M(s)^{\beta+1}
    \leq
    M(-1)^{\beta+1} - (\beta+1)\delta.
  \end{equation*}
  The contradiction now follows from increasing $n$ until
  $M(-1)^{\beta+1} < (\beta+1)\delta$.
\end{proof}

\begin{lemma}
  \label{lem:vanish-small-particles}
  For any $t_0 \in (0,T^*)$ there is a constant $\eta =
  \eta(\beta,d,L,t_0)$ such that $x_j(t) \leq \eta$ for some $t
  \in [0,t_0]$ and $j \in \bbZ$ implies $\tau_j \leq t +
  (x_j(t)/K)^{\beta+1}$, where $K=K(\beta,d,L)$ is the constant from
  Lemma \ref{lem:vanish-lower-power-law}.
\end{lemma}

\begin{proof}
  Lemma \ref{lem:vanish-lower-power-law} yields $\tau_j \leq t +
  (x_j(t)/K)^{\beta+1}$ provided that $\tau_j \leq T^*$; hence, we
  only have to prove the latter inequality in the case $x_j(t)>0$.
  Furthermore, by a shift in time it is sufficient to consider $t=0$.

  Following the same idea as in the proof of Lemma
  \ref{lem:vanish-lower-power-law}, we assume for contradiction that
  the claim is wrong. Then there are a sequence $\eta_n \to 0$ as $n
  \to \infty$ and solutions $x^n \colon [0,T^n)$ such that $T^n \geq
  t_0$ and
  \begin{equation*}
    x^n_{j_n}(0) \leq \eta_n,
    \qquad
    x^n_{j_n}(t)>0 \text{ for all } t < \min(T^n, C d^{\beta+1})
  \end{equation*}
  for some particle $j^n$. Again, by shifting the particle index we
  select a subsequence of $n \to \infty$ such that $j^n_-=0$, $j^n$
  and $j^n_+$ are constant, drop the index $n$ in the notation, and
  restrict our attention to the particles $k=0,\ldots,j_+$.

  We now rescale time and particle sizes according to
  \begin{equation*}
    y_k(s) = \frac{x_k(t)}{\eta},
    \qquad
    t = \eta^{\beta+1} s
  \end{equation*}
  and denote by $B = \set{m+1,\ldots,l-1} \subset \set{0,\ldots,j_+}$
  the largest set of consecutive particles containing $j$ such that
  \begin{equation*}
    \liminf_{n\to\infty} y_k(0) = \alpha_k \in [0,\infty)
    \qquad
    \text{for all } k \in B;
  \end{equation*}
  restriction to a subsequence of $n$ allows us to assume that the
  lower limits are attained as limits and that for $k=m,l$ we have
  $y_k(0) \to \infty$ as $n \to \infty$.
  As before, H\"older continuity implies 
  \begin{equation*}
    y_k(s)
    \geq
    c - C s^\frac{1}{\beta+1}
    \qquad
    \text{for } k=m,l,
    \qquad
    c = \min_{k=m,l} y_k(0),
  \end{equation*}
  and $s \in [0,S)$ where $S = T \eta^{-(\beta+1)}$, while the
  rescaled mass $M(s) = \sum_{k \in B} y_k(s)$ satisfies
  \begin{equation*}
    M(s)
    \leq
    M(0) + 4 L C s^\frac{1}{\beta+1}.
  \end{equation*}
  Since $S \to \infty$, $c \to \infty$ and $M(0) \to \sum_{k \in B}
  \alpha_k$ as $n \to \infty$, we can for any $\delta>0$ find an $n$
  such that $\delta<S$ and
  \begin{equation*}
    y_k(s)
    \geq
    c - C \delta^\frac{1}{\beta+1}
    \geq
    2^{1/\beta} \big( M(0) + 4 L C \delta^\frac{1}{\beta+1} \big)
    \geq
    2^{1/\beta} M(s)
  \end{equation*}
  for $k=m,l$ and all $s \in [0,\delta]$. Hence, we obtain
  \begin{equation*}
    \tderiv{s} M(s)
    \leq
    y_m(s)^{-\beta} + y_l(s)^{-\beta} - 2 M(s)^{-\beta}
    \leq
    - M(s)^{-\beta}
  \end{equation*}
  for all $s \in (0,\delta) \sm \{ \eta^{-(\beta+1)} \tau_k \colon k
  \in B\}$ and
  \begin{equation*}
    0
    <
    y_j(\delta)^{\beta+1}
    \leq
    M(\delta)^{\beta+1}
    \leq
    M(0)^{\beta+1} - (\beta+1)\delta.
  \end{equation*}
  Choosing $\delta$ and $n$ sufficiently large gives a contradiction.
\end{proof}

\begin{remark}
  If $\tau_j>T^*$, then $x_j(t) \geq c > 0$ for all $t \in
  [0,T^*]$. This and Lemma \ref{lem:vanish-lower-power-law} if $\tau_j
  \leq T^*$ imply that every $x_j^{-\beta} \chi_{\set{x_j>0}}$ is
  integrable in $[0,T^*]$. Hence, for initial data as in Assumption
  \ref{ass:initial-data} any solution in $[0,T^*)$ to the differential
  equation \eqref{eq:master-eq} as in the remark after
  Lemma~\ref{lem:solving-ode} is a solution according to Definition
  \ref{def:integral-solution}.
\end{remark}

% -----------------------------------------------------------------------------
% - Existence 
% -----------------------------------------------------------------------------

\subsection{Existence of solutions}
%\label{sec:existence-solutions}

We now prove existence of a solution to \eqref{eq:master-eq} by
truncating the initial data at a particle index $n$ and sending $n \to
\infty$. To this end, given $x^\infty \in \ell_+^\infty$ as in
Assumption \ref{ass:initial-data},
% we let
% \begin{equation*}
%   x^n_j =
%   \begin{cases}
%     x^\infty_j &\text{for } |j| \leq n, \\
%     0          &\text{for } |j| > n,
%   \end{cases}
%   \qquad
%   n \in \bbN
% \end{equation*}
we first look for a solution of the finite dimensional initial value
problem
%and look first for a solution $x^n \colon [0,\infty) \to
%\ell^\infty_+$ of the $2n+1$ ordinary differential equations 
\begin{equation}
  \label{eq:truncated-eq}
  \dot x^n_j(t) = \slaplace[\sigma] x^n_j(t)^{-\beta}
  \quad
  \text{if } t>0
  \qquad\text{and}\qquad
  x^n_j(0) = x^\infty_j
  \qquad\text{for}\qquad
  j=-n,\ldots,n.
\end{equation}
Recall that \eqref{eq:truncated-eq} is well-defined by definition of
$\slaplace[\sigma]$, in which contributions from beyond the two
outermost living particles are dropped.

By a solution to \eqref{eq:truncated-eq} we mean a continuous $x^n
\colon [0,\infty) \to X^n$, where
\begin{equation*}
  X^n
  =
  \set[x=(x_j) \in \bbR^{2n+1}]%
  { 0 \leq x_j \leq C \text{ for } j=-n,\ldots,n
    \text{ and some constant } C>0},
\end{equation*}
such that $x^n$ attains the initial data in \eqref{eq:truncated-eq}
and has the following property: For each $j=-n,\ldots,n$ there exists
a vanishing time $\tau^n_j \in (0,\infty]$ as defined before, and
between $t=0$, consecutive vanishing times and $t=\infty$ each $x^n_j$
is continuously differentiable and satisfies the differential equation
$\dot x^n_j = \slaplace[\sigma] (x^n_j)^{-\beta}$.
In short, $x^n$ solves the same differential equation as in the full
problem, but with truncated initial data, and consequently, the
results of Section \ref{sec:a-priori-estimates} are true for $x_j^n$
provided that $n$ is larger than the particles involved. In
particular, we have
\begin{enumerate}
\item persistence of traps
  \begin{equation*}
    x^n_{j_k}(t) \geq \frac{d}{2}
  \end{equation*}
  for all $t \in [0,T^*]$ and $j_k, n$ such that $n>|j_k|$;
\item the upper bound
  \begin{equation*}
    x^n_j(t) \leq C
  \end{equation*}
  for $t \in [0,T^*]$ and all $j, n$ such that $n>|j_\pm|$ where $C$
  depends on $\beta, d, L$ and $\sup_{j \in \bbZ} x^\infty_j$;
\item the H\"older estimate
  \begin{equation*}
    | x^n_j(t_2) - x^n_j(t_1) |
    \leq
    C |t_2-t_1|^{\frac{1}{\beta+1}}
  \end{equation*}
  for all $0 \leq t_1 < t_2 \leq T^*$ and all $j, n$ such that
  $n>|j_\pm|$, where $C=C(\beta,L)$; and
\item the vanishing estimate
  \begin{equation*}
    x^n_j(t) \leq \eta \text{ for some } t \leq t_0
    \qquad\Longrightarrow\qquad
    \tau^n_j \leq t + C x^n_j(t)^{\beta+1}
    \leq
    T^*
  \end{equation*}
  for $t_0 \in (0,T^*)$, sufficiently small $\eta=\eta(\beta,d,L,t_0)$
  and all $j, n$ such that $n>|j_\pm|$, which in turn implies
  \begin{equation*}
    x^n_j(t) \geq C (\tau^n_j-t)^\frac{1}{\beta+1}
  \end{equation*}
  for $t \in [0,\tau_j^n]$, where $C=C(\beta, d,L)$.
\end{enumerate}
Furthermore, equation \eqref{eq:change-of-mass-2} for the total mass
$M_{-n,n}^n(t) = \sum_{j=-n}^n x_j(t)$ of the truncated system
provides
\begin{equation}
  \label{eq:xn-bound}
  x^n_j(t)
  \leq
  M_{-n,n}^n(t)
  \leq
  M_{-n,n}^n(0),
\end{equation}
which means that \emph{all} $x_j$, $j=-n,\ldots,n$ are bounded for
fixed $n$ and not only those between two traps.

\begin{lemma}[Existence for the truncated system]
  \label{lem:existence-truncated}
  For any $n \in \bbN$ and initial data $x^n(0) =
  (x^n_j(0))_{j=-n,\ldots,n}$ such that $0 < x^n_j(0) < \infty$ there
  is a solution to \eqref{eq:truncated-eq}.
\end{lemma}

\begin{proof}
  Since \eqref{eq:truncated-eq} is an ordinary differential equation
  with locally Lipschitz continuous right hand side in
  $(0,\infty)^{2n+1}$, Picard's Theorem yields a time $\theta_0>0$ and
  a unique solution $x^n \in C^1([0,\theta_0); (0,\infty)^{2n+1})$. By
  inequality \eqref{eq:xn-bound}, the $x^n_j$ are bounded, hence the
  solution can be extended either indefinitely or until a time
  $\theta_1$ at which one or more of the $x^n_j$ vanish.
  In the latter case, we have $x^n_j \in C^1((0,\theta_1)) \cap
  C^0([0,\theta_1])$ for all $j=-n,\ldots,n$, and the evolution can be
  continued by restriction to the living particles at time $\theta_1$.
  By iteration we thus obtain a solution, which after at most $2n+1$
  steps is identically $0$.
\end{proof}

\begin{theorem}[Local existence]
  \label{thm:local-existence}
  Given initial data as in Assumption \ref{ass:initial-data} and $T_1
  < T^*$, where $T^* = C d^{\beta+1}$ as in the persistence estimate,
  there exist a subsequence of $(x^n)$, not relabeled, and a solution
  $x \colon [0,T_1) \to \ell^\infty$ to \eqref{eq:master-eq} such that
  $x^n_j \to x_j$ uniformly in $[0,T_1)$ for any $j \in \bbZ$.
\end{theorem}

\begin{proof}
  Let $x^n$, $n \in \bbN$ be the solution from Lemma
  \ref{lem:existence-truncated} and set $x^n_j \equiv 0$ for $|j|>n$.
  Then for any fixed $j \in \bbZ$ the sequence $(x^n_j)_{n \in \bbN}$
  is uniformly bounded in $C^{0,1/(\beta+1)}([0,T^*],[0,\infty))$ and
  by the Arzel\`a-Ascoli Theorem compact in $C^{0,\alpha}([0,T^*])$
  for $0 \leq \alpha < 1/(\beta+1)$. Without relabeling we may thus
  extract a subsequence of $n \to \infty$ for $j=0,\pm1,\pm2,\ldots$
  successively and take a diagonal sequence to obtain $x \colon
  [0,T^*] \to \ell_+^\infty$ such that $x^n_j \to x_j$ in
  $C^{0,\alpha}([0,T^*])$ and $x_j \in C^{0,1/(\beta+1)}([0,T^*])$ for
  any $j \in \bbZ$.
  By a further subsequence and diagonal argument we may assume that
  each sequence $(\tau^n_j)_{n \in \bbN}$ converges to some $\theta_j
  \in (0,\infty]$ as $n \to \infty$.

  We aim to prove that $x$ is a solution to \eqref{eq:master-eq} by
  passing to the limit $n \to \infty$ in the integral equation
  \begin{equation*}
    x^n_j(t_2) - x^n_j(t_1)
    =
    \int_{t_1}^{t_2} \slaplace[\sigma] x^n_j(s)^{-\beta} \,\dint{s}
  \end{equation*}
  for $0 \leq t_1<t_2 \leq T_1$ and any $j \in \bbZ$.
  To this end, we fix $\eta=\eta(\beta,d,L,T_1)>0$ such that the
  vanishing estimate holds with $t_0=T_1$ and infer that for
  each $j$ and $n>|j_\pm|$ we have either
  \begin{align}
    \label{eq:local-existence-case1}
    \tau^n_j > T^*
    \qquad&\text{and}\qquad
    x^n_j(t) > \eta \text{ for all } t \in [0,T_1]
    \\
    \intertext{or}
    \label{eq:local-existence-case2}
    \tau^n_j \leq T^*
    \qquad&\text{and}\qquad
    x^n_j(t) \geq C (\tau^n_j-t)^\frac{1}{\beta+1}
  \end{align}
  Hence, we may choose another subsequence of $n \to \infty$ such that
  for each $j$ exactly one of \eqref{eq:local-existence-case1},
  \eqref{eq:local-existence-case2} holds for all large $n$.
  In the first case we trivially have $x^n_j \geq \eta$ for all large
  $n$ and obtain
  \begin{equation*}
  (x^n_j)^{-\beta} \chi_{\{x^n_j>0\}} \to x_j^{-\beta}
  \chi_{\{x_j>0\}}
  \qquad\text{in } L^1(0,T_1).  
  \end{equation*}
  In the second case we find that $\theta_j \leq T^*$ is the unique
  vanishing time of $x_j$ in $[0,T^*]$, because $0 = x^n_j(t) \to
  x_j(t)$ for $t \geq \theta_j$ and
  \begin{equation*}
    x_j(t) \geq \lim _{n \to \infty} C (\tau^n_j-t)^{\frac{1}{\beta+1}}
    =
    C (\theta_j - t)^{\frac{1}{\beta+1}}
    > 0
  \end{equation*}
  for $t<\theta_j$. Furthermore, we obtain convergence of
  $(x^n_j)^{-\beta} \chi_{\{x^n_j>0\}}$ in $L^1(0,T_1)$ from pointwise
  convergence of $\chi_{\{x^n_j>0\}}$ to $\chi_{\{x_j>0\}}$, the upper
  bound
  \begin{equation*}
    x^n_j(t)^{-\beta} \chi_{\{x^n_j>0\}}
    \leq
    C (\tau^n_j-t)^{-\frac{\beta}{\beta+1}} \chi_{\{t<\tau^n_j\}},
  \end{equation*}
  and the generalized Dominated Convergence Theorem.
  Finally, the persistence of traps and uniqueness of vanishing times
  imply that each $x_j$ has no more than $4 L$ different neighbors in
  the time interval $[0,T_1]$, and convergence of $\slaplace[\sigma]
  (x^n_j)^{-\beta}$ in $L^1(0,T_1)$ follows.
\end{proof}

We conclude this section with our global existence result, which
follows from combining Theorem \ref{thm:local-existence} and the lower
density estimate in Lemma \ref{lem:box-dens-stability}.

\begin{theorem}[Global existence]
  %\label{thm:global-existence}
  For initial data as in Assumption~\ref{ass:initial-data} there
  exists a solution $x \colon [0,\infty) \to \ell^\infty_+$ to
  \eqref{eq:master-eq}, which has the following properties:
  \begin{enumerate}
  \item for any $T>0$ there is a constant $C=C(\beta,d,L,T)$ such
    that $x_j(t) \leq x_j(0) + C$ for all $t \in [0,T]$ and all $j \in \bbZ$;
  \item for any $T>0$ each $x_j$ is H\"older continuous in $[0,T]$
    with exponent $1/(\beta+1)$ and $\| x_j \|_{C^{0,1/(\beta+1)}([0,T])}
    \leq C = C(\beta,d,L,T)$;
  \item if $\tau_j < \infty$ then $x_j(t) \leq C ( \tau_j -
    t)^{1/(\beta+1)}$ for all $t \in [0,\tau_j]$, where $C=C(\beta)$;
  \item for any $T>0$ there is a constant $K=K(\beta,d,L,T)$ such that
    $\tau_j < T$ implies $x_j(t) \geq K (\tau_j-t)^{1/(\beta+1)}$ for
    all $t \in [0,\tau_j]$ and for any $j \in \bbZ$.
  \end{enumerate}
\end{theorem}

\begin{proof}
  Set $T_1 = T^*/2 = C d^{\beta+1}/2$, where $C=C(\beta)$, and let $x
  \colon [0,T_1) \to \ell^\infty_+$ be a solution from Theorem
  \ref{thm:local-existence}. For $\eps_1 \in (0,1)$ to be chosen
  below, Lemma \ref{lem:box-dens-stability} provides a sufficiently
  large $q_1 = q_1(\beta,L,\eps_1)$ so that $x(T_1)$ satisfies the
  lower density estimate in Assumption \ref{ass:initial-data} with
  $R(k)$ replaced by $R_{q_1}(k)$, $L$ replaced by $L_1 := q_1 L$ and
  $d$ replaced by $d_1 := (1-\eps_1) d$.
  Therefore, we may apply the local existence result once more to
  extend the solution up to the time $T_2 = T_1 + C
  d_1^{\beta+1}/2$. Moreover, the a priori estimates are true in the
  whole interval $[0,T_2)$, once the constants that depend on the
  traps are adapted appropriately; this, of course, makes them
  time-dependent.

  With $0<\eps_n<1$, $n \in \bbN$ we may iterate the preceding
  argument to find
  \begin{equation*}
    d_n = d \prod_{k=1}^n (1-\eps_k),
    \qquad
    L_n = L \prod_{k=1}^n q_k,
    \qquad
    T_n = T_1 + \frac{C}{2} \sum_{k=1}^{n-1} d_k^{\beta+1},
  \end{equation*}
  and a solution up to any of the times $T_n$. Choosing $\eps_n$ such
  that $d_n \to 0$ sufficiently slowly, we can arrange that $T_n \to
  \infty$ as $n \to \infty$. Finally, since any $T>0$ is reached by
  the sequence $(T_n)$ after finitely many steps, the claimed
  estimates follow once more from adapting the constants of the local
  ones.
\end{proof}

%%%%%%%%%%%%%%%%%%%%%%%%%%%%%%%%%%%%%%%%%%%%%%%%%%%%%%%%%%%%%%%%%%%%%%%%%%%%%%% 
%%%%%%%%%%%%%%%%%%%%%%%%%%%%%%%%%%%%%%%%%%%%%%%%%%%%%%%%%%%%%%%%%%%%%%%%%%%%%%%
%%%
%%% Non-uniqueness
%%%
%%%%%%%%%%%%%%%%%%%%%%%%%%%%%%%%%%%%%%%%%%%%%%%%%%%%%%%%%%%%%%%%%%%%%%%%%%%%%%%
%%%%%%%%%%%%%%%%%%%%%%%%%%%%%%%%%%%%%%%%%%%%%%%%%%%%%%%%%%%%%%%%%%%%%%%%%%%%%%%

\section{Non-uniqueness}
\label{sec:non-uniqueness}

In this section we show that solutions to \eqref{eq:master-eq} are in
general not uniquely determined by their initial data.  To this end,
we consider once more truncated problems, this time with two slightly
different sequences of initial particles that have the same limit but
yield two different solutions.
More precisely, for $N \in \bbN$ we consider initial data $\alpha^N
\in \ell_+^\infty$ containing particles
\begin{equation*}
  %\label{eq:alpha1}
  \alpha^N_j = 0
  \quad\text{for}\quad |j| > 3 N,
  \qquad
  \alpha^N_j = 1
  \quad\text{for}\quad -3 N \leq j <0,
  \qquad
  \alpha^N_{3 N} = 1,
\end{equation*}
and 
\begin{equation*}
  %\label{eq:alpha2}
  \alpha^N_{3 j}   = 1,\quad
  \alpha^N_{3 j+1} = R^N_{j,1},\quad
  \alpha^N_{3 j+2} = R^N_{j,2}
  \quad\text{for}\quad j = 0,1,\ldots,N-1,
\end{equation*}
where $(R^N_{j,p})_{j = 0,\ldots,N-1}$, $p=1,2$ are two decreasing
sequences of positive numbers to be chosen; compare Figure
\ref{fig:non-uniqueness-data}.
By $x^N$ we denote the solution to \eqref{eq:master-eq} with initial
data
\begin{equation}
  \label{eq:idata}
  x_j^N(0) =
  \begin{cases}
    \alpha_j^N &\text{if } j \not= 3N,\\
    \alpha_j^N + \eps^N &\text{if } j = 3N,
  \end{cases}
\end{equation}
where $(\eps^N)$ is another sequence to be specified; in case 
$\eps^N=0$ we write $\bar x^N$.

We will prove the following result.

\begin{theorem}
  \label{thm:non-uniqueness}
  Let
  \begin{equation}
    \label{eq:beta}
    \beta \geq \beta_* := \frac{\ln 4 - \ln 3}{\ln 3 - \ln 2}
    \qquad
    \Longleftrightarrow
    \qquad
    2^{\frac{1}{\beta+1}} \leq \frac{3}{2}.
  \end{equation}
  Then there are sequences $(R^N_{j,p})_{j=0,\ldots,N-1}$, $N \in
  \bbN$, $p=1,2$ and $(\eps^N)_{N \in \bbN}$ with the following
  properties.
  \begin{enumerate}
  \item For any $j=0,1,\ldots$ and $p=1,2$ the sequence $(R^N_{j,p})$
    converges to some $R_{j,p}>0$ as $N \to \infty$.
  \item There are a time $T>0$, functions $\bar x \colon [0,T]
    \to \ell_+^\infty$, $x \colon [0,T] \to \ell_+^\infty$, and a
    subsequence of $N\to \infty$, not relabeled, such that for all $j
    \in \bbZ$ we have $\bar x^N_j \to \bar x_j$ and $x^N_j \to
    x_j$ in $C([0,T])$ as $N \to \infty$.
  \item Both, $\bar x$ and $x$ solve \eqref{eq:master-eq} in
    $[0,T)$ with initial data $\bar x(0) = x(0) = \alpha$, where
    $\alpha_j = 1$ for $j < 0$ and $\alpha_{3 j}=1$, $\alpha_{3 j+1} =
    R_{j,1}$, $\alpha_{3 j+2} = R_{j,2}$ for $j \geq 0$, but we
    have $\bar x \not= x$.
  \end{enumerate}
\end{theorem}

\begin{remark}
  Inequality \eqref{eq:beta} is a technical condition that we use to
  study in a simple way the phase portrait of an equation associated
  with two adjacent small particles. We believe that it can be removed
  or at least improved by more careful considerations; see the proof
  of Lemma \ref{lem:matching} and the remark thereafter.
\end{remark}

% -----------------------------------------------------------------------------
% - Heuristics
% -----------------------------------------------------------------------------

\subsection{Overview of the proof}

The proof of Theorem \ref{thm:non-uniqueness} is rather technical, so
we first outline the main ideas.
To this end, we denote as above by $\bar x^N$ and $x^N$ two solutions
that have initial data as in \eqref{eq:idata} with $\eps^N=0$ and
$\eps^N >0$ to be specified, respectively; for simplicity of notation
we drop the upper index $N$ where possible.
Existence of $\bar x$ and $x$ as well as convergence to solutions of
\eqref{eq:master-eq} along a subsequence of $N \to \infty$ follows
from the a priori estimates and compactness arguments in Section
\ref{sec:existence}. Furthermore, since we will chose $\eps^N \to 0$
as $N \to \infty$, both limit solutions will have the same initial
data.

Theorem \ref{thm:non-uniqueness} will follow from the fact that the
differences
\begin{equation}
  \label{eq:differences}
  |x_{j}(t) - \bar x_{j}(t)|
\end{equation}
are uniformly positive for some fixed index $j=3 N_*$, some $t>0$ and
all sufficiently large $N$. To prove this fact, we will on the one
hand derive detailed asymptotic formulas for \eqref{eq:differences} in
case $j = 3N_*,3(N_*+1),\ldots,3N$.
On the other hand, we will require a precise description of solutions
near vanishing times, because many small particles between $3N_*$ and
$3N$ will have vanished at any positive time $t>0$.

For the analysis to come we require that non-adjacent small
particles of $\bar x$ and $x$ vanish at well-separated
times. Moreover, the details will become simpler if we
choose the sequences $(R^N_{j,p})$ such that
the two adjacent small particles of one solution, say $\bar x_{3j+1}$ and
$\bar x_{3j+2}$ vanish simultaneously for each $j=N_*,\ldots,N-1$.
In Section \ref{sec:choice-small-part} we show that such a choice is
possible, and to that aim it is necessary to study the dependence of
solutions on $(R^N_{j,p})$.
The main technical ingredients here are a continuity result for
vanishing times in Lemma \ref{lem:van-times-continuity} and estimates
of the vanishing times in a local problem where two small particles
are surrounded by two large ones in Lemma \ref{lem:local-lemma-I}.
The existence of $(R^N_{j,p})$ with the desired properties then
follows from a degree theory type argument in Proposition
\ref{pro:idata-equal-van-times}.

As a consequence of Section \ref{sec:choice-small-part}, we obtain
initial data such that for $j=N_*,\ldots,N-1$ the particles $\bar
x_{3j+p}$ and $x_{3j+p}$ have finite vanishing times $\bar
\tau_{3j+p}$ and $\tau_{3j+p}$, respectively.
Furthermore, $T_j = \max\{\bar \tau_{3j+1}, \tau_{3j+1},
\tau_{3j+2}\}$ defines a sequence such that all small particles to the
right of $3j$ have vanished at time $T_j$ whereas all small particles
to the left of $3j$ have not.
The key step of our proof is to show that the most important
contribution to $x - \bar x$ at time $T_{j-1}$ is $x_{3(j-1)}(T_{j-1})
- \bar x_{3(j-1)}(T_{j-1})$ and stems from $x_{3j}(T_j) - \bar
x_{3j}(T_j)$.
Ultimately, this contribution originates from the initial
difference $x_{3N} - \bar x_{3N} = \eps$, is iteratively propagated
through the solutions and amplified by the vanishing of the small
particles.

To illustrate the iterative step from $T_j$ to $T_{j-1}$ let us assume
that at time $T_j$ both $\bar x$ and $x$ are equal except for the
particle $3j$ and that we have the simplified situation
\begin{equation}
  \label{eq:illustrate-data1}
  \begin{aligned}
    &\bar x_{3k}(T_j) = x_{3k}(T_j) = 1,
    \quad
    \bar x_{3k+p}(T_j) = x_{3k+p}(T_j) = R_{k,p}
    &\text{for } k < j,
    \\
    &\bar x_{3k}(T_j) = x_{3k}(T_j) = 1,
    \quad
    \bar x_{3k+p}(T_j) = x_{3k+p}(T_j) = 0
    &\text{for } k > j
  \end{aligned}
\end{equation}
and
\begin{equation}
  \label{eq:illustrate-data2}
  \bar x_{3j}(T_j) = 1,
  \qquad
  x_{3j}(T_j) = 1+\delta
\end{equation}
for some small $\delta>0$. Then we obtain in Theorem
\ref{thm:iterative-estimate} that the difference $|\bar x_{3j}(T_j) -
x_{3j}(T_j)| = \delta$ is transported to the particles $3(j-1)$ where
we have
\begin{equation}
  \label{eq:main-difference}
  |\bar x_{3(j-1)}(T_{j-1}) - x_{3(j-1)}(T_{j-1})|
  \sim
  \left( R_{j-1}^{4\beta+1} \delta \right)^{\frac{1}{3\beta+1}}
\end{equation}
with $R_{j-1}$ comparable to both $R_{j-1,p}$, $p=1,2$.
Formula \eqref{eq:main-difference} quantifies the amplification effect
provided that $\delta$ is sufficiently small compared to $R_{j-1}$
and we use it in Section \ref{sec:proof-theorem} to choose initial
perturbations $\eps^N$ for $N>N_*$ such that
\begin{equation*}
  \delta_N = \eps^N,
  \qquad
  \delta_{j-1} \sim
  \left( R_{j-1}^{4\beta+1} \delta_j \right)^{\frac{1}{3\beta+1}}
  \quad\text{for } j = N_*+1,\ldots,N
\end{equation*}
indeed leads to an iterative amplification.
More precisely, we show that we can prescribe $\delta_{N_*}>0$
sufficiently small and independent of $N$ such that there is
$(\eps_N)_{N>N_*}$ with the properties that
\begin{equation*}
  \eps^N
  \sim
  \delta_{N_*}^{(3\beta+1)^{N-N_*}}
\end{equation*}
and
\begin{equation*}
  |x^N_{3N_*}(T^N_{N_*}) - \bar x^N_{3N_*}(T^N_{N_*})| = \delta_{N_*}
\end{equation*}
for all large $N$. Then, taking the limit $N \to \infty$ yields the
non-uniqueness result.

The main part of the proof is the analysis of one iteration 
providing \eqref{eq:main-difference}. It is carried out in Section
\ref{sec:auxiliary-results}, where we study the local problem of two
small particles surrounded by a large one on each side, and in Section
\ref{sec:iterative-estimates}, where we deal with the full particle
system.
First, considering again the setting
\eqref{eq:illustrate-data1}--\eqref{eq:illustrate-data2},
we compute the asymptotic formulas
\begin{equation*}
  \bar x_{3(j-1)+p}(t)
  \sim
  \left( \bar \tau_{3(j-1)+1} -t \right)^{\frac{1}{\beta+1}}
  \qquad\text{as}\qquad
  t \to \bar \tau_{3(j-1)+1}
\end{equation*}
by a straightforward local analysis in Lemma
\ref{lem:asymptotics}. Then, since $x(T_j)$ and $\bar x(T_j)$ differ
by an amount of order $\delta$ it is natural to linearize
\begin{equation}
  \label{eq:order-bar-particles}
  D_p(t)
  =
  x_{3(j-1)+p}(t) - \bar x_{3(j-1)+p}(t)
  \sim
  \delta \phi_p(t) + \text{higher order terms}
\end{equation}
for $t>T_j$ and $p=1,2$.
This is done in Proposition \ref{pro:active-small-approx}, where we in
particular obtain that $\phi_p$ becomes singular near the vanishing
times of the particles $3(j-1)+p$ and that its most singular
contribution is of order
\begin{equation}
  \label{eq:order-phi}
  \left( \bar \tau_{3(j-1)+1} -t \right)^{-\frac{3\beta}{\beta+1}}.
\end{equation}
Hence, by \eqref{eq:order-bar-particles} and \eqref{eq:order-phi} the
linearization is valid only up to times such that
\begin{equation*}
  \left( \bar \tau_{3(j-1)+1} -t \right)^{\frac{3\beta+1}{\beta+1}}
  \sim
  \delta,
\end{equation*}
where $\bar x_{3(j-1)+p}$ and $\phi_p$ become comparable, and for
later times we have to study the full nonlinear system with suitable
matching conditions in order to approximate $x_{3(j-1)+p}$.
The local version of the latter problem is analyzed in Lemma
\ref{lem:matching} and the consequences for the vanishing times in the
particle system are given in Corollary
\ref{cor:diff-van-time-estimate}.
In the end, the approximations for $\bar x_{3(j-1)+p}$, $x_{3(j-1)+p}$
as well as $\bar \tau_{3(j-1)+p}$, $\tau_{3(j-1)+p}$ enable us to
calculate \eqref{eq:main-difference} in Lemma \ref{lem:step-x},
Corollary \ref{cor:mass-transfer} and Proposition
\ref{pro:active-large}.

Besides the key steps just described, Section
\ref{sec:iterative-estimates} contains several lemmas that provide
estimates for particles other than $3(j-1)+p$ and their differences,
which are not necessarily zero as in the simplified setting above.
The contributions of these particles, however, are small compared to
\eqref{eq:main-difference}, and their estimates usually follow from
straightforward computations and standard ODE arguments such as
Gronwall's inequality.

% -----------------------------------------------------------------------------
% - Small particles
% -----------------------------------------------------------------------------

\subsection{Choice of the small particles}
\label{sec:choice-small-part}

With $0 < \gamma < 1/3$ to be fixed later we set
\begin{equation}
  \label{eq:R_m}
  R_j = \gamma^j,
  \qquad
  j \in \bbN,
\end{equation}
and let $(R^N_{j,1})_{j=0,\ldots,N-1}$, $(R^N_{j,2})_{j=0,\ldots,N-1}$
satisfy
\begin{equation}
  \label{eq:RN_mp}
  \frac{1}{2} R_j  \leq R^N_{j,p} \leq \frac{3}{2} R_j
\end{equation}
for $j=0,\ldots,N-1$, $N \in \bbN$ and $p=1,2$. Moreover, we assume
that $(\eps^N)_{N \in \bbN}$ is a sequence such that
\begin{equation}
  \label{eq:epsN}
  |\eps^N| < \frac{1}{8}
  \qquad\text{and}\qquad
  \lim_{N\to^\infty} \eps^N = 0.
\end{equation}
Then, the arguments that led to existence of solutions apply here with
$d=3/4$ and $L=3$, and for later reference we summarize the
corresponding results as follows.
Recall that $C,c$ denote generic constants that depend only on
$\beta$, if not stated otherwise.

\begin{proposition}
  \label{pro:existence}
  There is a time $T^*>0$ such that for any $N \in \bbN$, any
  sequences $(R^N_{j,p})$ as in \eqref{eq:RN_mp}, and any $(\eps^N)$
  as in \eqref{eq:epsN} there exists a solution $x^N \colon [0,\infty)
  \to X^{6N+1}$ to \eqref{eq:master-eq} with initial data
  \eqref{eq:idata} and has the following properties.
  \begin{enumerate}
  \item For all $t \in [0,T^*]$ and all initially large particles,
    that is, $k=-3N,\ldots,0$, $k=3j$ with $j=0,\ldots,N-1$, and
    $k=3N$, we have $x^N_k(t) \geq 1/2$.
  \item There is a unique vanishing time $\tau^N_k \in (0,T^*] \cup
    \{\infty\}$ for each particle $k=-3N,\ldots,3N$ and there are two
    constants $C,c$ such that if $\tau^N_k < \infty$ for some $|k|<3N$
    we have
    \begin{equation}
      \label{eq:exist:upper-bound}
      x^N_k(t) \leq C (\tau^N_k - t)^\frac{1}{\beta+1}
    \end{equation}
    and
    \begin{equation}
      \label{eq:exist:lower-bound}
      x^N_k(t) \geq c ( \tau^N_k-t)^\frac{1}{\beta+1}
    \end{equation}
    for all $t \in [0,\tau^N_k]$.
  \item\label{item:small-vanish} For any $t_0 \in (0,T^*)$ there is
    $\eta=\eta(\beta,t_0)$ such that $x_k(t_0) \leq \eta$ for some
    $|k|<3N$ implies $\tau_k \leq t_0 + (x_k(t_0)/c)^{\beta+1}$.
  \end{enumerate}
\end{proposition}

The existence of $(R^N_{j,p})$ such that adjacent small particles of
the solution $\bar x$ vanish simultaneously follows from the
continuity of vanishing times and their asymptotic behavior as
particle sizes become locally small. We consider these issues in the
following two lemmas.

\begin{lemma}[Continuity of vanishing times]
  \label{lem:van-times-continuity}
  There is $N_*>0$ such that for any $N>N_*$ and any solution $x$ with
  $(R_j)$, $(R^N_{j,p})$ and $(\eps^N)$ as in
  \eqref{eq:R_m}--\eqref{eq:epsN} we have
  \begin{equation*}
    c R_j^{\beta+1}
    \leq
    \tau^N_{3j+p}
    \leq
    C R_j^{\beta+1} < T^*
    \qquad
    \text{for}
    \qquad
    j=N_*,\ldots,N-1,\quad
    p=1,2.
  \end{equation*}
  Moreover, the $\tau^N_{3j+p}$ change continuously with
  $\left((R^N_{m,1})_{m=N_*,\ldots,N-1},
    (R^N_{m,2})_{m=N_*,\ldots,N-1}\right) \in \prod_{k=N_*}^{N-1}
  \left( \tfrac{1}{2}R_k,\tfrac{3}{2} R_k \right)^2$.
\end{lemma}

\begin{proof}
  The first claim is an immediate consequence of \eqref{eq:R_m},
  \eqref{eq:RN_mp} and Proposition \ref{pro:existence}.
  To prove continuity, we fix $N>N_*$ and let $x$ be a solution
  for initial data $(\alpha_j^N)$ with $(R^N_{j,p})$ and $\eps^N$ as
  in \eqref{eq:R_m}--\eqref{eq:epsN}.
  Moreover, we denote by $y$ a solution for nearby initial data where
  $R^N_{j,p}$ is replaced by $R^N_{j,p} + \delta_{j,p}$ with
  $|\delta_{j,p}| \leq \delta$ for $j=N_*,\ldots,N-1$ and $p=1,2$.
  We write $\tau$ and $\theta$ for the vanishing times of $x$ and $y$,
  respectively, and order $(\tau_{3j+p})$ by size via
  \begin{equation*}
    0 < \tau_{3j_1+p_1} \leq \tau_{3j_2+p_2} \leq \ldots
    \tau_{3j_{N-N_*}+p_{N-N_*}},
  \end{equation*}
  where $(j_m)_{m=1,\ldots,N-N_*} \subset \set{N_*,\ldots,N-1}$ and
  $(p_m)_{m=1,\ldots,N-N_*} \subset \set{1,2}^{N-N_*}$.

  Given small $\eps>0$, our first step is to show by a continuation
  argument that the differences $D_k(t) = y_k(t)-x_k(t)$, $k=
  -3N,\ldots,3N$ remain arbitrarily small for $t \leq \tau_{3j_1+p_1}
  - \eps$, provided that we choose $\delta$ sufficiently small.
  Indeed, Proposition \ref{pro:existence}(1 and 2) implies
  \begin{equation*}
    x_k(t)
    \geq
    c \eps^{\frac{1}{\beta+1}}
  \end{equation*}
  for all $|k| \leq 3N$ and $t \leq \tau_{3j_1+p_1} - \eps$, and
  assuming that $y_k(t) \geq c \eps^{1/(\beta+1)}/2$ we compute
  \begin{equation*}
    \left| y_k(t)^{-\beta} - x_k(t)^{-\beta} \right|
    =
    \beta \left| \int_{x_k(t)}^{y_k(t)} s^{-(\beta+1)} \,\dint{s} \right|
    \leq
    C \frac{1}{\eps} |D_k(t)|.
  \end{equation*}
  Using the differential equation for $x_k$ and $y_k$, we
  find
  \begin{equation*}
    \deriv{t} \|D_k(t)\|_{\infty}
    \leq
    C \frac{1}{\eps} \|D_k(t)\|_{\infty},
  \end{equation*}
  and by integration we obtain
  \begin{equation}
    \label{eq:dist1}
    |D_k(t)|
    \leq
    \delta
    \exp\left( C \frac{1}{\eps} (\tau_{3j_1+p_1} - \eps)\right)
  \end{equation}
  for all $|k|\leq 3N$ and $t\leq \tau_{3j_1+p_1} - \eps$ such that
  $y_k(t) \geq c \eps^{1/(\beta+1)}/2$.
  By choosing $\delta$ sufficiently small, the latter inequality holds
  up to some positive $t$ and the right hand side of \eqref{eq:dist1}
  can be made arbitrarily small. Thus, for any $0 < \omega < c
  \eps^{1/(\beta+1)}/4$ we obtain
  \begin{equation}
    \label{eq:dist-of-sol}
    |D_k(t)| \leq \omega
  \end{equation}
  and 
  \begin{equation}
    \label{eq:cont-resulting-est}
    y_k(t)
    \geq
    x_k(t) - \omega
    \geq
    \tfrac{3c}{4} \eps^{\frac{1}{\beta+1}}
    > \tfrac{c}{2} \eps^{\frac{1}{\beta+1}}
  \end{equation}
  for all $|k|\leq3N$, all small $\delta$ depending on $\eps$ and
  $\omega$, and all $t$ such that $y_k(t) \geq c
  \eps^{1/(\beta+1)}/2$. By \eqref{eq:cont-resulting-est} the above
  reasoning extends to all times $t \leq \tau_{3j_1+p_1}-\eps$.
  We conclude that $\theta_{k} \geq \tau_{3j_1+p_1} - \eps$, and
  moreover we can use \eqref{eq:dist-of-sol} and
  \eqref{eq:exist:upper-bound} to transfer an upper bound from
  $x_{3j_1+p_1}$ to $y_{3j_1+p_1}$, namely
  \begin{equation*}
    y_{3j_1+p_1}(\tau_{3j_1+p_1}-\eps)
    \leq
    x_{3j_1+p_1}(\tau_{3j_1+p_1}-\eps) + \omega
    \leq
    C \eps^{\frac{1}{\beta+1}} + \omega
    \leq
    C \eps^{\frac{1}{\beta+1}}.
  \end{equation*}
  Proposition \ref{pro:existence}(\ref{item:small-vanish}) with $t_0 =
  \tau_{3j_1+p_1}-\eps$ then yields $\theta_{3j_1+p_1} \leq
  \tau_{3j_1+p_1} + C \eps$.
  
  Let now $s \in \{2,\ldots,N-N_*\}$ be the index such that
  \begin{equation*}
    \tau_{3j_1+p_1} = \tau_{3j_2+p_2} = \ldots = \tau_{3j_{s-1}+p_{s-1}}
    <
    \tau_{3j_s+p_s}.
  \end{equation*}
  Estimating $\theta_{3j_m+p_m}$, $m=1,\ldots,s-1$ as above, we can
  arrange that for given $\eps>0$ we have
  \begin{equation}
    \label{eq:small-times}
    \left| \theta_{3j_m+p_m} - \tau_{3j_m+p_m} \right|
    =
    \left| \theta_{3j_m+p_m} - \tau_{3j_1+p_1} \right|
    \leq
    \eps
    \qquad\text{for}\qquad
    m = 1,\ldots,s-1
  \end{equation}
  and
  \begin{equation*}
    \tau_{3j_m+p_m}
    \geq
    \tau_{3j_s+p_s}
    >
    \tau_{3j_1+p_1} + \sqrt{\eps}
    \qquad\text{for}\qquad
    m = s,\ldots,N-N_*.
  \end{equation*}
  The lower bound \eqref{eq:exist:lower-bound} and inequality
  \eqref{eq:dist-of-sol} imply
  \begin{equation*}
    x_{3j_m+p_m}(\tau_{3j_1+p_1} - \eps)
    \geq
    C \sqrt{\eps}^{1/(\beta+1)}
    \qquad\text{and}\qquad
    y_{3j_m+p_m}(\tau_{3j_1+p_1} - \eps)
    \geq
    C \sqrt{\eps}^{1/(\beta+1)}
  \end{equation*}
  for $m=s,\ldots,N-N_*$ and all sufficiently small $\delta$.  Using
  \eqref{eq:exist:upper-bound} for
  $y_{3j_m+p_m}$ we find
  \begin{equation*}
    \theta_{3j_m+p_m}
    \geq
    \tau_{3j_1+p_1} + C \sqrt{\eps}
    \geq
    \tau_{3j_1+p_1} + \eps,
  \end{equation*}
  and \eqref{eq:exist:lower-bound} gives
  \begin{equation*}
    y_{3j_m+p_m}(t)
    \geq
    C \left( \theta_{3j_m+p_m} - t \right)^{1/(\beta+1)}
    \geq
    C \eps^{1/(\beta+1)}
  \end{equation*}
  for all $t \in [\tau_{3j_1+p_1}-\eps,\tau_{3j_1+p_1}+\eps]$ and
  $m=s,\ldots,N-N_*$.
  Consequently, we have
  \begin{equation*}
    \int_{\tau_{3j_1+p_1}-\eps}^{\tau_{3j_1+p_1}+\eps}
    y_{3j_m+p_m}(t)^{-\beta} \,\dint{t}
    \leq
    C \eps^{-\beta/(\beta+1)} \eps
    =
    C \eps^{1-\frac{\beta}{(\beta+1)}}
  \end{equation*}
  for $m=s,\ldots,N-N_*$, while for $m=1,\ldots,s-1$ the bounds
  \eqref{eq:small-times} and \eqref{eq:exist:upper-bound} imply
  \begin{equation*}
    \int_{\tau_{3j_1+p_1}-\eps}^{\tau_{3j_1+p_1}+\eps}
    y_{3j_m+p_m}(t)^{-\beta} \,\dint{t}
    \leq
    \int_{\tau_{3j_1+p_1}-\eps}^{\theta_{3j_m+p_m}}
    C (\theta_{3j_m+p_m} - t)^{-\frac{\beta}{\beta+1}} \,\dint{t}
    \leq
    C \eps^{1-\frac{\beta}{\beta+1}}.
  \end{equation*}
  Furthermore, the large particles $k<0$ and $k=3j$, $j=0,\ldots,N$
  clearly satisfy
  \begin{equation*}
    \int_{\tau_{3j_1+p_1}-\eps}^{\tau_{3j_1+p_1}+\eps}
    y_{k}(t)^{-\beta} \,\dint{t}
    \leq
    C \eps
    \leq
    C  \eps^{1-\frac{\beta}{\beta+1}}.
  \end{equation*}

  Using these inequalities in the equations for all particles $y_{k}$
  that have not vanished at time $\tau_{3j_1+p_1}+\eps$, that is, for
  all $k \not\in \set[k=3j_m+p_m]{m=1,\ldots,s-1}$, we obtain
  \begin{equation*}
    \left| y_{k}(\tau_{3j_1+p_1}+\eps)
      - y_{k}(\tau_{3j_1+p_1}-\eps) \right|
    \leq
    C \eps^{1-\frac{\beta}{\beta+1}},
  \end{equation*}
  and since a similar computation applies to each $x_{k}$,
  $k \not\in \set[k=3j_m+p_m]{m=1,\ldots,s-1}$ we deduce
  \begin{equation*}
    | D_{k}(\tau_{j_1+p_1}+\eps) |
    \leq
    |D_{k}(\tau_{j_1+p_1}-\eps)| + C \eps^{1-\frac{\beta}{\beta+1}}
    \leq
    \omega + C \eps^{1-\frac{\beta}{\beta+1}}.
  \end{equation*}
  Thus, at time $\tau_{j_1+p_1}+\eps$ we have $x_{3j_m+p_m} =
  y_{3j_m+p_m} = 0$ for $m=1,\ldots,s-1$, while the other particles
  are positive and arbitrarily close to each other if $\delta$ is
  sufficiently small. Repeating the preceding arguments we obtain the
  asserted continuity after at most $2 N$ steps.
\end{proof}

To study the vanishing times as particle sizes become small we
consider the local problem
\begin{equation}
  \label{eq:local-eq}
  \begin{aligned}
    \dot Y_1
    &=
    -2 Y_1^{-\beta} + Y_{2}^{-\beta} + F_1,
    \\
    \dot Y_2
    &=
    -2 Y_2^{-\beta} + Y_{1}^{-\beta} + F_2,
  \end{aligned}
\end{equation}
where $F_1$ and $F_2$ are continuous functions of time. Solutions to
\eqref{eq:local-eq} are understood in the sense that vanished
particles are removed, and their existence for positive initial data
follows as in Section \ref{sec:existence}.

\begin{lemma}[Vanishing times for local equation]
  \label{lem:local-lemma-I}
  Let $(Y_1, Y_2)$ be a solution to \eqref{eq:local-eq} with initial
  data $Y_1(0)=A_1$ and $Y_2(0)=A_2$, where $1/2 \leq A_p \leq 2$ and
  $|F_1|+|F_2| \leq \eta$.
  Then the following properties hold:
  \begin{enumerate}
  \item For all sufficiently small $\eta$ depending only on $\beta$
    both vanishing times, denoted by $\tau_1$ and $\tau_2$, are finite
    independently of $A_p$ and $F_p$. Moreover, $Y_1 + Y_2 \leq 4$ for
    all $t$.
  \item Given $\eps>0$ there exists $\eta_0 = \eta_0(\beta,\eps) \sim
    \eps$ such that for all $\eta<\eta_0$ the conditions
    $A_1-A_2=\eps>0$ and $|F_1|+|F_2| \leq \eta$ imply $\tau_2<\tau_1$
    and $Y_1(t)-Y_2(t) \geq \eps/2$ for all $t \in [0,\tau_2]$.
  \item With $\nu = |A_1-A_2|$ and $\eta$ as above we have
    \begin{equation*}
      \max_{p=1,2} \left| \tau_p -  \frac{A_p^{\beta+1}}{\beta+1} \right|
      \to 0
      \qquad\text{as}\qquad
      (\nu,\eta) \to (0,0)
    \end{equation*}
    independently of $F_p$ and $A_p$, $p=1,2$.
  \end{enumerate}
\end{lemma}

\begin{proof}
  Until the first vanishing time we have
  \begin{equation*}
    \dot Y_1 + \dot Y_2
    =
    - \Big( Y_1^{-\beta} + Y_2^{-\beta} \Big) + F_1 + F_2
    \leq
    - \big( Y_1 + Y_2 \big)^{-\beta} + \eta,
  \end{equation*}
  and since $A_1+ A_2 \leq 4$ the sum $Y_1+Y_2$ decreases for all
  $\eta<4^{-\beta}/2$. In fact, since $\eta<(A_1+A_2)^{-\beta}/2$,
  we have
  \begin{equation*}
    \dot Y_1 + \dot Y_2
    \leq
    -\frac{1}{2} (Y_1+Y_2)^{-\beta},
  \end{equation*}
  thus the first particle vanishes at a time smaller than $C
  (A_1+A_2)^{\beta+1} \leq 4^{\beta+1}C$.  Thereafter, the same
  argument applied to the remaining particle shows that the second
  vanishing time is bounded by $C(A_1+A_2)^{\beta+1} \leq
  4^{\beta+1}C$, too.

  As long $Y_1 > Y_2 > 0$, which by the assumption for the second
  claim is true initially, the difference $Y_1-Y_2$ satisfies
  \begin{equation*}
    \dot Y_1 - \dot Y_2
    =
    -3 \Big( Y_1^{-\beta} - Y_2^{-\beta} \Big) + F_1 - F_2
    \geq
    - \eta,
  \end{equation*}
  and we obtain
  \begin{equation*}
    Y_1(t) - Y_2(t)
    \geq
    \eps - \eta t
    \geq
    \eps - \eta \max (\tau_1,\tau_2).
  \end{equation*}
  Hence, the second assertion follows with $\eta_0 = \eps /
  (2\max(\tau_1,\tau_2))$.

  To prove the third claim, set $Y_{\max}(t) = \max(Y_1(t),Y_2(t))$
  and $Y_{\min}(t) = \min(Y_1(t),Y_2(t))$ and denote the corresponding
  vanishing times by $\tau_{\max}$ and $\tau_{\min}$, respectively.
  Letting $0<\omega< 2^{-(\beta+1)}$ be arbitrary and setting $\tau =
  (1-\omega) Y_{\max}(0)^{\beta+1}/(\beta+1)$ we are going to show
  that $\tau \leq \tau_{\min} \leq \tau_{\max} \leq \tau + O(\omega)$
  as $(\nu,\eta) \to 0$, which yields the result due to
  $|Y_{\max}(0)-A_p| \leq |A_1 - A_2| = \nu$ for $p=1,2$.

  First, since $Y_{\max} \leq 4$ we have $\eta Y_{\max}^\beta \leq
  \omega$ for all $t>0$ and all $\eta < 4^{-\beta} \omega$.
  Using $Y_{\max} \geq Y_{\min}$ we thus find
  \begin{equation*}
    \dot Y_{\max}(t)
    \geq
    - Y_{\max}(t)^{-\beta} - \eta
    \geq
    - (1+\omega) Y_{\max}(t)^{-\beta}
  \end{equation*}
  for almost all $t < \tau_{\min}$ and conclude
  \begin{equation}
    \label{eq:loc-lem-I:est-0}
    Y_{\max}(t)
    \geq
    \left( Y_{\max}(0)^{\beta+1} - (\beta+1) (1+\omega) t
    \right)^{\frac{1}{\beta+1}}
    \geq
    Y_{\max}(0) \omega^{\frac{2}{\beta+1}}
  \end{equation}
  if $t \leq \min(\tau,\tau_{\min})$.
  Next, as $Y_{\min}(0) = Y_{\max}(0) - \nu$ we have
  \begin{equation}
    \label{eq:eq:loc-lem-I-est-00}
    Y_{\min}(t)
    \geq
    \tfrac{1}{2} \left( Y_{\max}(t) - \nu \right)
    \geq
    \tfrac{1}{2} \left( Y_{\max}(0) \omega^{\frac{2}{\beta+1}} - \nu \right)
    \geq
    \tfrac{1}{4} Y_{\max}(0) \omega^{\frac{2}{\beta+1}}
  \end{equation}
  for small $t>0$ and all $\nu \leq Y_{\max(0)}
  \omega^{2/(\beta+1)}/2$. For such $t$ we find
  \begin{align*}
    \dot Y_{\max} - \dot Y_{\min}
    \leq
    3 \left( Y_{\min}^{-\beta} - Y_{\max}^{-\beta} \right) + \eta
    &=
    3 \beta \int_{Y_{\min}}^{Y_{\max}} y^{-(\beta+1)} \,\dint{y} + \eta
    \\
    &\leq
    C \frac{1}{Y_{\max}(0)^{\beta+1}\omega^2}
    \left( Y_{\max}-Y_{\min} \right) + \eta,
  \end{align*}
  and choosing $\nu$ and $\eta$ sufficiently small we infer that
  \begin{equation}
    \label{eq:loc-lem-I:est-1}
    Y_{\max}(t) - Y_{\min}(t)
    \leq C \left( \nu + \eta \omega^2 \right)
    e^{C \frac{\min(\tau,\tau_{\min})}{Y_{\max}(0)^{\beta+1}\omega^2}}
    + C \eta \omega^2
    \leq
    Y_{\max}(0) \omega^{\frac{3}{\beta+1}}.
  \end{equation}
  Combining \eqref{eq:loc-lem-I:est-0} and \eqref{eq:loc-lem-I:est-1}
  we arrive at
  \begin{equation}
    \label{eq:loc-lem_I:est-2}
    Y_{\min}(t)
    =
    Y_{\max}(t) + Y_{\min}(t) - Y_{\max}(t)
    \geq
    Y_{\max}(0) \omega^{\frac{2}{\beta+1}} (1-\omega^{\frac{1}{\beta+1}})
    \geq
    \tfrac{1}{2} Y_{\max}(0) \omega^{\frac{2}{\beta+1}},
  \end{equation}
  which implies that by continuation the above considerations extend
  to $\min(\tau,\tau_{\min})$.  Therefore, $\tau<\tau_{\min}$ and we
  have
  \begin{equation*}
    \frac{Y_{\max}(0)^{\beta+1}}{\beta+1} - O(\omega)
    \leq \tau_{\min}
    \leq \tau_{\max}.
  \end{equation*}
  Moreover, we find
  \begin{align*}
    Y_{\min}^{-\beta} - Y_{\max}^{-\beta}
    &=
    \beta \int_{Y_{\min}}^{Y_{\max}} y^{-(\beta+1)} \,\dint{y}
    \leq
    \beta Y_{\min}^{-(\beta+1)} (Y_{\max}-Y_{\min})
    \\
    &\leq
    \beta \left( \frac{Y_{\max}}{Y_{\min}} \right)^{\beta+1}
    \frac{Y_{\max}-Y_{\min}}{Y_{\max}} Y_{\max}^{-\beta}
    \leq
    C \omega^{\frac{1}{\beta+1}} Y_{\max}^{-\beta}
  \end{align*}
  using \eqref{eq:loc-lem-I:est-1}--\eqref{eq:loc-lem_I:est-2} to
  estimate the first fraction and
  \eqref{eq:eq:loc-lem-I-est-00}--\eqref{eq:loc-lem-I:est-1} for the
  second.  From this and with $\eta Y_{\max}^\beta \leq \omega$ in the
  equation for $Y_{\max}$ we deduce
  \begin{equation*}
    \dot Y_{\max}(t)
    \leq
    - Y_{\max}(t)^{-\beta}
    + C \omega^{\frac{1}{\beta+1}} Y_{\max}(t)^{-\beta} + \eta
    \leq
    - \left( 1 - C \omega^{\frac{1}{\beta+1}} - \omega \right)
    Y_{\max}(t)^{-\beta}
  \end{equation*}
  for $t < \tau$. By integration we obtain
  \begin{equation*}
    Y_{\max}(\tau)^{\beta+1}
    \leq
    Y_{\max}(0)^{\beta+1} - (\beta+1)(1-C \omega^{\frac{1}{\beta+1}}
    - \omega) \tau
    =
    Y_{\max}(0)^{\beta+1} r(\omega)
  \end{equation*}
  where $r(\omega) \to 0$ as $\omega \to 0$.
  The argument that proved the first part of the lemma but with
  $\eta$ depending on $\omega$, now yields $\tau_{\min} \leq \tau_{\max}
  \leq \tau + C Y_{\max}(0)^{\beta+1} r(\omega)^{\beta+1}$.
\end{proof}

Finally, we prove that the sequences $(R^N_{j,p})$ can be chosen such
that neighboring small particles of the solution $\bar x$ for
$\eps^N = 0$ vanish at the same time.

\begin{proposition}
  \label{pro:idata-equal-van-times}
  There exists $N_*>0$ such that for all $N>N_*$ and any sequence
  $(R_j)$ as in \eqref{eq:R_m} there are two sequences
  $(R^N_{j,p})_{j=0,\ldots,N-1}$, $p=1,2$ satisfying \eqref{eq:RN_mp}
  with the following properties.
  \begin{enumerate}
  \item The vanishing times $\bar \tau^N_{3j+p}$, $p=1,2$ are
    identical for $j=N_*,\ldots,N-1$.
  \item We have $R^N_{j,1} = R_j$ for $j=0,\ldots,N-1$ and
    $R^N_{j,2} = R_j$ for $j=0,\ldots,N_*-1$ as well as
    \begin{equation*}
      | R^N_{j,2} - R_j | \leq \omega_j R_j
      \qquad\text{for}\qquad j=N_*,\ldots,N-1,
    \end{equation*}
    where $(\omega_j)$ is a sequence with $\lim_{j\to\infty} \omega_j
    = 0$ that is independent of $(R_j)$.
  \item There is a decreasing sequence $(\widetilde \omega_j)$
    independent of $(R_j)$ with $0 < \widetilde \omega_j <
    1/(4(\beta+1))$ and $\lim_{j\to\infty} \widetilde \omega_j = 0$
    such that
    \begin{equation*}
      \left| \bar \tau^N_{3j+p} - \frac{R_j^{\beta+1}}{\beta+1} \right|
      \leq \widetilde \omega_j R_j^{\beta+1}
    \end{equation*}
    for $j=N_*,\ldots,N-1$ and $p=1,2$.
  \end{enumerate}
\end{proposition}

\begin{proof}
  For $N_*$ larger than in Lemma \ref{lem:van-times-continuity} to be
  fixed later and $N>N_*$ we define a map $\Phi$ from the cube $Q =
  [-1,1]^{N-N_*}$ to $\bbR^{N-N_*}$ as follows.
  Given $\eta = (\eta_k)_{k=N_*}^{N-1} \in Q$ we solve
  \eqref{eq:master-eq} for $\bar x^N$ with $R^N_{j,1} = R_j$ for
  $j=0,\ldots, N-1$, $R^N_{j,2} = R_j$ for $j=0,\ldots,N_*-1$, and
  $R^N_{j,2} = R_j (1 + \eta_j \omega_j)$ for $j=N_*,\ldots,N-1$,
  where $(R_j)$ is as in \eqref{eq:R_m} and $(\omega_j)_{j\geq0}$ a
  sequence of positive numbers also to be fixed later. Then we set
  \begin{equation*}
    \Phi(\eta)
    =
    (\bar \tau^N_{3j+1} - \bar \tau^N_{3j+2})_{j=N_*}^{N-1}.
  \end{equation*}
  By Lemma \ref{lem:van-times-continuity} the function $\Phi$ is
  well-defined, finite and continuous. We aim to show that the
  topological degree $\deg (\Phi,Q,0)$ is nonzero, which implies the
  existence of $\eta_* \in Q$ such that $\Phi(\eta_*)=0$, that is,
  $\bar \tau^N_{3j+1} = \bar \tau^N_{3j+2}$ for
  $j=N_*,\ldots,N-1$.
  For that purpose, we use that the degree is invariant under the
  homotopy
  \begin{equation*}
    \Psi(\lambda,\eta) = -(1-\lambda)\eta + \lambda \Phi(\eta),
    \qquad
    \lambda \in [0,1],
    \quad
    \eta \in Q,
  \end{equation*}
  and thus
  \begin{equation*}
    1 = \deg(-\mathrm{Id},Q,0) = \deg(\Psi(0,\cdot),Q,0)
    = \deg(\Psi(1,\cdot),Q,0) = \deg(\Phi,Q,0),
  \end{equation*}
  provided that $0 \not= \Psi(\lambda,\eta)$ for all $\lambda \in
  [0,1]$ and $\eta \in \partial Q$.
  To compute $\Psi(\lambda,\eta)$ suppose first that $\eta
  \in \partial Q$ satisfies $\eta_m=-1$ for some $m \in
  \{N_*,\ldots,N-1\}$ and $|\eta_k| \leq 1$ otherwise. Setting
  \begin{equation*}
    y_1(t) = \frac{\bar x^N_{3m+1}(R_j^{\beta+1}t)}{R_m},
    \qquad
    y_2(t) = \frac{\bar x^N_{3m+2}(R_j^{\beta+1}t)}{R_m}
  \end{equation*}
  and $A_1 = 1$, $A_2 = 1-\omega_m$ we aim to apply Lemma
  \ref{lem:local-lemma-I} with
  \begin{equation*}
    F_1(t)
    =
    \left( \frac{\bar x^N_{3m}(R_m^{\beta+1}t)}{R_m} \right)^{-\beta},
    \qquad
    F_2(t)
    =
    \left( \frac{\bar x^N_{3m+3}(R_m^{\beta+1}t)}{R_m} \right)^{-\beta}.
  \end{equation*}
  Since $A_1-A_2 = \omega_m$ and $|F_1|+|F_2| \leq C R_m^{\beta}$ by
  Proposition \ref{pro:existence}, we need $C R_m^\beta \leq
  \eta_0(\beta,\omega_m) \sim \omega_m$, which due to $R_j = \gamma^j$
  can be arranged for some decreasing sequence $(\omega_j)$
  independent of $R_j = \gamma^j$ and $\gamma < 1/3$.
  We thus find $\bar \tau^N_{3m+1} > \bar \tau^N_{3m+2}$, that
  is, $\Phi(\eta)_m > 0$ and $\Psi(\lambda,\eta)_m =
  (1-\lambda)+\Phi(\eta)_m > 0$ for $\lambda \in [0,1]$. Finally, the
  same argument yields $\Psi(\lambda,\eta)_m < 0$ if $\eta_m=+1$.

  The sequences $(R^N_{j,p})$ corresponding to $\eta_* \in Q$, whose
  existence is now proven, satisfy the second claim of the proposition
  by definition of $\Phi$.
  The last assertion is a consequence of Lemma
  \ref{lem:local-lemma-I}(3) applied to $y_p(t) =
  x^N_{3j+p}(R_j^{\beta+1}t)/R_j$ for $j=N_*,\ldots,N-1$ and $p=1,2$,
  which yields
  \begin{equation*}
    \left| \frac{\bar\tau^N_{3j+p}}{R_j^{\beta+1}}
      - \frac{1}{\beta+1} \left(\frac{R^N_{j,p}}{R_j}\right)^{\beta+1}
    \right|
    \to 0
    \qquad
    \text{as}
    \qquad
    \omega_j \to 0,
  \end{equation*}
  and finishes the proof.
\end{proof}

\begin{remark}
  Exactly as in the proof above, we can apply Lemma
  \ref{lem:local-lemma-I} also to $x$ and obtain 
  \begin{equation*}
    \left| \tau^N_{3j+p} - \frac{R_j^{\beta+1}}{\beta+1} \right|
    \leq \widetilde \omega_j R_j^{\beta+1}
  \end{equation*}
  for $j=N_*,\ldots,N-1$ and $p=1,2$, where $N_*$ and $\widetilde
  \omega_j$ are as in Proposition \ref{pro:idata-equal-van-times}.
\end{remark}

% -----------------------------------------------------------------------------
% - Auxiliary results
% -----------------------------------------------------------------------------

\subsection{Auxiliary results for the local problem}
\label{sec:auxiliary-results}

In this section, we collect some auxiliary results for the local
problem \eqref{eq:local-eq}, which we use to study the propagation of
the perturbation.
The first two lemmas address the asymptotic behavior of solutions.

\begin{lemma}[Power law for simultaneously vanishing particles]
  \label{lem:asymptotics}
  Given $F_1,F_2 \in C([0,\infty))$, suppose that $Y_1, Y_2 \in
  C^1(0,\bar \tau) \cap C^0([0,\bar\tau])$ solve
  \eqref{eq:local-eq} with initial data $Y_1(0), Y_2(0) \in
  [R/2,3R/2]$ for some $R>0$ and vanish simultaneously at time
  $\bar\tau>0$.
  Then there is a constant $C$ such that with $\eta =
  \|F_1\|_\infty+\|F_2\|_\infty$ we have
  \begin{equation*}
    \left| Y_p(t) - \big( (\beta+1)(\bar\tau-t)
      \big)^{\frac{1}{\beta+1}} \right|
    \leq
    C \eta (\bar\tau-t)
    \leq
    C R^{\beta+1}
  \end{equation*}
  for all $t \in [0,\bar\tau]$ and $p=1,2$.
\end{lemma}

\begin{proof}
  The arguments for estimate \eqref{eq:simple-lower-bound-1} and Lemma
  \ref{lem:vanish-lower-power-law} (with $L=3$ and $d=1/2$) apply to
  $Y_1,Y_2$ and yield
  \begin{equation}
    \label{eq:asymptotics:bounds-on-y}
    0 < c \leq \frac{Y_p(t)}{(\bar\tau-t)^{\frac{1}{\beta+1}}}
    \leq C < \infty
  \end{equation}
  for $p=1,2$ and all $t \in [0,\bar\tau)$.
  We define a new time $s \in [s_0,\infty)$ by
  \begin{equation*}
    R^{\beta+1} e^{-s(\beta+1)} = \bar \tau - t,
    \qquad
    R^{\beta+1} e^{-s_0(\beta+1)} = \bar \tau
  \end{equation*}
  and set
  \begin{equation*}
    W_p(s)
    =
    \frac{Y_p(t)}{(\bar\tau-t)^{\frac{1}{\beta+1}}}
    =
    \frac{1}{R} e^s Y_p(t),
    \qquad
    p=1,2.
  \end{equation*}
  Then \eqref{eq:asymptotics:bounds-on-y} implies that $W_p(s) \in
  [c,C]$ for all $s \geq s_0$ and that $s_0$ is bounded from above and
  below independently of $R$, because $\bar \tau$ is of order
  $R^{\beta+1}$.
  Moreover, $W_p$ satisfies
  \begin{equation*}
    \tderiv{s} W_p(s)
    =
    W_p(s)
    + (\beta+1) \left( -2 W_p(s)^{-\beta} + W_{3-p}(s)^{-\beta} \right)
    + (\beta+1) R^\beta e^{-s\beta} F_p(t)
  \end{equation*}
  for $p=1,2$, and taking the difference of both equations we obtain
  \begin{equation*}
    \tderiv{s} ( W_1-W_2 )
    =
    W_1 - W_2
    - 3 (\beta+1) \left( W_1^{-\beta} - W_2^{-\beta} \right)
    + (\beta+1) R^\beta e^{-s\beta} \left( F_1 - F_2 \right).
  \end{equation*}
  Next, multiplying by $\sign(W_1-W_2)$ and observing that
  $(W_1^{-\beta}-W_2^{-\beta}) \sign(W_1-W_2) \leq 0$ we find
  \begin{equation*}
    \tderiv{s} |W_1-W_2|
    \geq
    |W_1-W_2| - (\beta+1) R^\beta e^{-s\beta} \eta
  \end{equation*}
  where we also used $\|F_1\|_\infty + \|F_2\|_\infty \leq
  \eta$. Integration then yields
  \begin{align}
    \nonumber
    |W_1-W_2|(s_2)
    &\geq
    \left( |W_1-W_2|(s_1) - (\beta+1) R^\beta \eta
      \int_{s_1}^{s_2} e^{-s\beta} e^{-(s-s_1)} \,\dint{s}
    \right)
    e^{s_2-s_1}
    \\
    \label{eq:asymptotics:difference-bound-1}
    &\geq
    \left( |W_1-W_2|(s_1) - (\beta+1) R^\beta \eta e^{-s_1 \beta} \right)
    e^{s_2-s_1}
  \end{align}
  for all $s_2 \geq s_1 \geq s_0$.
  We conclude that
  \begin{equation}
    \label{eq:asymptotics:difference-bound-2}
    |W_1-W_2|(s)
    \leq
    (\beta+1) R^\beta \eta e^{-s \beta}
    \qquad
    \text{for all } s \geq s_0,
  \end{equation}
  because otherwise the existence of $s_1 \geq s_0$ such that
  $|W_1-W_2|(s_1) \geq (1+\nu)(\beta+1) R^\beta \eta e^{-s_1 \beta}$
  for some $\nu>0$ and \eqref{eq:asymptotics:difference-bound-1} imply
  that $|W_1-W_2|(s)$ grows exponentially for $s \geq s_1$ in
  contradiction to $W_p \in [c,C]$.

  We now use \eqref{eq:asymptotics:difference-bound-2} and $W_p \geq
  c$ to linearize and estimate
  \begin{equation*}
    | W_{1}^{-\beta} - W_2^{-\beta} |(s)
    \leq
    C |W_{1}-W_2|(s)
    \leq
    C R^\beta \eta e^{-s\beta}
  \end{equation*}
  for $p=1,2$, and we deduce that
  \begin{equation*}
    \tderiv{s} W_p(s)
    =
    W_p(s) - (\beta+1) W_p(s)^{-\beta}
    +
    R^\beta \eta e^{-s\beta} \widetilde F_p(s)
  \end{equation*}
  where $\widetilde F_p$ is bounded by a constant depending only on
  $\beta$.
  Standard ODE methods imply that the only solution of $\dot W = W
  -(\beta+1) W^{-\beta}$ with $0 < c \leq W \leq C$ is the constant
  function $W_0 \equiv (\beta+1)^{1/(\beta+1)}$, and arguing by
  contradiction as above we obtain
  \begin{equation*}
    |W_p(s) - W_0|
    \leq
    C \|\widetilde F_p \|_\infty \eta R^\beta e^{-s\beta}
  \end{equation*}
  for all $s \geq s_0$ or, equivalently,
  \begin{equation*}
    \left| \frac{Y_p(t)}{(\bar\tau-t)^{\frac{1}{\beta+1}}}
      - (\beta+1)^{\frac{1}{\beta+1}} \right|
    \leq
    C \|\widetilde F_p\|_\infty \eta (\bar\tau-t)^{\frac{\beta}{\beta+1}}
  \end{equation*}
  for all $t \in [0,\bar\tau]$ and $p=1,2$.
\end{proof}

\begin{lemma}[Asymptotics near vanishing]
  \label{lem:matching}
  Suppose that $\beta \geq \beta_*$ where $\beta_*$ is as in
  \eqref{eq:beta}.  There are constants $B_*>0$, $S_* \in (-1,1)$,
  $\eps_0>0$, $\eta_0>0$, $T_0>0$ such that the following result is
  true:
  If $Y_1$, $Y_2$ solve \eqref{eq:local-eq} (with our usual convention
  that $0^{-\beta}=0$ and that $Y_p$ remains $0$ once it has vanished)
  in the interval $(-T,\infty)$, where $F_p$ are continuous such that
  $\|F\|_\infty \leq \eta \leq \eta_0$, $T>T_0$ and $\eta^{1-\eps_0}
  \leq T^{-(4\beta+1)/(\beta+1)}$, and if
  \begin{equation*}
    \left|
      \begin{pmatrix}
        Y_1 \\ Y_2
      \end{pmatrix}
      (-T)
      -
      (\beta+1)^{\frac{1}{\beta+1}} T^{\frac{1}{\beta+1}}
      \begin{pmatrix}
        1 \\ 1
      \end{pmatrix}
      -
      B_* T^{-\frac{3\beta}{\beta+1}}
      \begin{pmatrix}
        1 \\ -1
      \end{pmatrix}
    \right|
    \leq
    C \left( \eta^{\eps_0} + T^{-1} \right) T^{-\frac{3\beta}{\beta+1}},
  \end{equation*}
  then the vanishing times $\tau_p$ of $Y_p$ satisfy
  \begin{equation*}
    | \tau_1 - (S_*+1) | + | \tau_2 - (S_*-1) |
    =
    o(1)_{\eta \to 0} + o(1)_{T \to \infty}.
  \end{equation*}
\end{lemma}

\begin{proof}
  The proof consists of four steps. In the first three steps we
  construct a solution for \eqref{eq:local-eq} and $F_p = 0$
  which has the desired asymptotic behavior as $t \to - \infty$. In
  the final step we show that we reach every solution with
  sufficiently small $\| F_p \|$.

  \emph{Step 1: Backwards solution.}
  The equation $\dot Y_1 = -2 Y_1^{-\beta}$ in $(-1,1)$ with terminal
  data $Y_1(1)=0$ has the unique solution
  \begin{equation*}
    Y_1(t) = \left( 2(\beta+1)(1-t) \right)^{\frac{1}{\beta+1}},
  \end{equation*}
  and we consider 
  \begin{equation*}
    \dot Y_1 = -2 Y_1^{-\beta} + Y_2^{-\beta},
    \qquad
    \dot Y_2 = -2 Y_2^{-\beta} + Y_1^{-\beta}
  \end{equation*}
  in $(-\infty,-1)$ with data $Y_1(-1) = (4(\beta+1))^{1/(\beta+1)}$
  and $Y_2(-1)=0$. A local solution which is positive for $t<-1$
  exists, and since $Y_{\min} = \min(Y_1,Y_2)$ and $Y_{\max} =
  \max(Y_1,Y_2)$ satisfy
  \begin{equation*}
    \dot Y_{\min} \leq - Y_{\min}^{-\beta}
    \qquad\text{and}\qquad
    \dot Y_{\max} \geq - Y_{\max}^{-\beta}
  \end{equation*}
  we obtain
  \begin{equation*}
    \left( (\beta+1)(-1-t) \right)^{\frac{1}{\beta+1}}
    \leq
    Y_{\min}(t)
    \leq
    Y_{\max}(t)
    \leq
    \left( (\beta+1)(3-t) \right)^{\frac{1}{\beta+1}}.
  \end{equation*}
  Thus, global existence follows and we find $Y_1>Y_2$.

  In order to study the asymptotic behavior of $Y_p(t)$ as $-t$
  becomes large, we let
  \begin{equation*}
    W_p(s)
    =
    \frac{Y_p(t)}{(-t)^{\frac{1}{\beta+1}}}
    =
    Y_p(t) e^{-\frac{s}{\beta+1}}
    \qquad\text{where}\qquad
    -t = e^s,
    \quad
    s \geq 0,
  \end{equation*}
  which satisfies
  \begin{equation}
    \label{eq:lem-matching-w}
    \tderiv{s} W_p
    =
    -\frac{W_p}{\beta+1} + 2 W_p^{-\beta} - W_{3-p}^{-\beta}.
  \end{equation}
  Subtracting the equations for $W_1$ and $W_2$ from each other and
  multiplying by $\sign(W_1-W_2)$ we deduce
  \begin{equation*}
    %\label{eq:lem-matching-w-diff}
    \tderiv{s} |W_1-W_2|
    \leq
    - \frac{|W_1-W_2|}{\beta+1}
  \end{equation*}
  and 
  \begin{equation*}
    |W_1(s)-W_2(s)|
    \leq
    |W_1(0) - W_2(0)| e^{-\frac{s}{\beta+1}}
    =
    \left( 4(\beta+1) \right)^{\frac{1}{\beta+1}} e^{-\frac{s}{\beta+1}}
  \end{equation*}
  for all $s \in (0,\infty)$.
  With this estimate and the lower bound
  \begin{equation*}
    W_p(s)
    \geq
    Y_{\min}(t) e^{-\frac{s}{\beta+1}}
    \geq
    C (-1 + e^s)^{\frac{1}{\beta+1}} e^{-\frac{s}{\beta+1}}
    =
    C (1-e^{-s})^{\frac{1}{\beta+1}}
  \end{equation*}
  we linearize $W_1^{-\beta}-W_2^{-\beta}$ to get
  \begin{equation*}
    |W_1^{-\beta} - W_2^{-\beta}|
    \leq
    \beta W_{\min}^{-(\beta+1)} |W_1-W_2|
    \leq
    C \frac{1}{1-e^{-s}} e^{-\frac{s}{\beta+1}}
  \end{equation*}
  and 
  \begin{equation*}
    \tderiv{s} W_p
    =
    -\frac{W_p}{\beta+1} + W_p^{-\beta}
    +
    O \left( 
      \frac{e^{-\frac{s}{\beta+1}}}{1-e^{-s}}
    \right)
  \end{equation*}
  for $s>0$.
  As in the proof of Lemma \ref{lem:asymptotics} it follows that
  $w_p = W_p - (\beta+1)^{\frac{1}{\beta+1}}$ satisfies
  \begin{equation*}
    | w_p(s) |
    \leq
    C \frac{e^{-\frac{s}{\beta+1}}}{1-e^{-s}}
  \end{equation*}
  for $s>0$, and since $w_p(0) = Y_p(-1) - (\beta+1)^{1/(\beta+1)}$ we
  conclude that $|w_p(s)| \leq C e^{-s/(\beta+1)}$ for all $s \geq 0$.
  Thus, linearizing \eqref{eq:lem-matching-w} around
  $(\beta+1)^{1/(\beta+1)}$ we obtain
  \begin{equation}
    \label{eq:lem-matching-eq-w}
    \deriv{s}
    \begin{pmatrix}
      w_1 \\ w_2
    \end{pmatrix}
    = \frac{-1}{\beta+1}
    \begin{pmatrix}
      2\beta+1 & -\beta \\
      -\beta & 2\beta+1
    \end{pmatrix}
    \begin{pmatrix}
      w_1 \\ w_2
    \end{pmatrix}
    +
    O \left( w_1^2+w_2^2 \right).
  \end{equation}
  for $s$ sufficiently large, say $s > s_0$.
  The solution of \eqref{eq:lem-matching-eq-w} can be written as
  \begin{equation*}
    \begin{pmatrix}
      w_1 \\ w_2
    \end{pmatrix}
    (s)
    =
    \Psi(s)
    \begin{pmatrix}
      w_1 \\ w_2
    \end{pmatrix}
    (s_0)
    +
    \int_{s_0}^s \Psi(s-r)
    O \left( w_1(r)^2 + w_2(r)^2 \right) \,\dint{r},
  \end{equation*}
  where 
  \begin{equation*}
    \Psi(s)
    =
    \frac{1}{2} e^{-(s-s_0)}
    \begin{pmatrix}
      1 & 1 \\ 1 & 1
    \end{pmatrix}
    +
    \frac{1}{2} e^{-(s-s_0)\frac{3\beta+1}{\beta+1}}
    \begin{pmatrix}
      +1 & -1 \\ -1 & +1
    \end{pmatrix}
  \end{equation*}
  is the fundamental matrix for the linear part of
  \eqref{eq:lem-matching-eq-w}.
  The contributions from $w_1^2+w_2^2$ are of order
  \begin{align*}
    \int_{s_0}^s e^{-(s-r)} e^{-2 r} \,\dint{r}
    =
    e^{-s} \int_{s_0}^\infty e^{-r} \,\dint{r}
    -
    e^{-s} \int_{s}^\infty e^{-r} \,\dint{r}
    =
    e^{-s_0} e^{-s} + O \left( e^{-2 s} \right)
  \end{align*}
  and
  \begin{align*}
    \int_{0}^s e^{-(s-r)\frac{3\beta+1}{\beta+1}}
    e^{-2r\frac{3\beta+1}{\beta+1}} \,\dint{r}
    =
    \frac{\beta+1}{3\beta+1} e^{-s_0\frac{3\beta+1}{\beta+1}}
    e^{-s\frac{3\beta+1}{\beta+1}}
    +
    O \left( e^{-2s\frac{3\beta+1}{\beta+1}} \right)
  \end{align*}
  with no mixed terms appearing since the eigenvectors corresponding
  to both powers in $\Psi$ are orthogonal. As a consequence we have
  \begin{equation}
    \label{eq:lem-match-asymp-w}
    \begin{pmatrix}
      w_1 \\ w_2
    \end{pmatrix}
    =
    A_0
    \begin{pmatrix}
      1 \\ 1
    \end{pmatrix}
    e^{-s}
    +
    B_0
    \begin{pmatrix}
      1 \\ -1
    \end{pmatrix}
    e^{-s\frac{3\beta+1}{\beta+1}}
    +
    O \left( e^{-2s}
      +
      e^{-2s \frac{3\beta+1}{\beta+1}} \right)
  \end{equation}
  for all $s \geq s_0$, where $A_0,B_0$ are two constants and the
  contribution of order $e^{-2s}$ is present in $w_p$ if and only if
  the contribution of order $e^{-s}$ is.

  \emph{Step 2: Removing the order $e^{-s}$.}
  For a given shift $S_* \in \bbR$ let $\widetilde Y_p(t) =
  Y_p(t-S_*)$ and set
  \begin{equation*}
    \widetilde Y_p(t) e^{-\frac{s}{\beta+1}}
    =
    \widetilde W_p(s)
    =
    (\beta+1)^{\frac{1}{\beta+1}} + \widetilde w_p(s)
  \end{equation*}
  for $t< \min(0,S_*-1)$ and $e^s = -t$.
  Repeating the considerations from Step 1 we find
  \begin{equation}
    \label{eq:lem-matching-w-tilde1}
    \begin{pmatrix}
      \widetilde w_1 \\ \widetilde w_2
    \end{pmatrix}
    =
    A_{S_*}
    \begin{pmatrix}
      1 \\ 1
    \end{pmatrix}
    e^{-s}
    +
    B_{S_*}
    \begin{pmatrix}
      1 \\ -1
    \end{pmatrix}
    e^{-s\frac{3\beta+1}{\beta+1}}
    +
    O \left( e^{-2s}
      + e^{-2s \frac{3\beta+1}{\beta+1}} \right)
  \end{equation}
  for some constants $A_{S_*}$, $B_{S_*}$ and $s \geq s_{S_*}$ where
  the $s_{S^*}$ is chosen appropriately; as above, the contribution of
  order $e^{-2s}$ is present if and only if the term of order $e^{-s}$
  is.  On the other hand, writing $-t+S_* = e^s+S_* = e^s
  (1+S_*e^{-s}) = e^{\widetilde s}$ and linearizing $(1+S_*
  e^{-s})^\alpha = 1 + \alpha S_* e^{-s} + O(e^{-2s})$ we deduce
  \begin{align*}
    \widetilde W_p(s)
    &=
    \widetilde Y_p(t) e^{-\frac{s}{\beta+1}}
    =
    Y_p(t-S_*) e^{-\frac{\widetilde s}{\beta+1}}
    e^{\frac{\widetilde s-s}{\beta+1}}
    =
    W_p(\widetilde s) (1+S_* e^{-s})^{\frac{1}{\beta+1}}
    \\
    &=
    W_p(\widetilde s) \left(
      1+ \frac{1}{\beta+1} S_* e^{-s} + O(e^{-2s})
    \right)
  \end{align*}
  and thus
  \begin{equation*}
    \widetilde w_p(s)
    =
    (\beta+1)^{-\frac{\beta}{\beta+1}} S_* e^{-s}
    +
    w_p(\widetilde s) + \frac{1}{\beta+1} S_* e^{-s} w_p(\widetilde s)
    + O(e^{-2s}).
  \end{equation*}
  Formula \eqref{eq:lem-match-asymp-w} for $w_p$ and the definition of
  $\widetilde s$ imply
  \begin{align*}
    \begin{pmatrix}
      w_1 \\ w_2
    \end{pmatrix}
    (\widetilde s)
    &=
    \frac{A_0 e^{-s}}{1+ S_* e^{-s}}
    \begin{pmatrix}
      1 \\ 1
    \end{pmatrix}
    +
    \frac{B_0 e^{-s\frac{3\beta+1}{\beta+1}}}
    {(1+S_* e^{-s})^{\frac{3\beta+1}{\beta+1}}}
    \begin{pmatrix}
      1 \\ -1
    \end{pmatrix}
    +
    O \left( e^{-2s}
      +
      e^{-2s \frac{3\beta+1}{\beta+1}}
    \right)
    \\
    &=
    w_p(s)
    + O \left(
      e^{-2 s} + e^{-s \left( \frac{3\beta+1}{\beta+1} +1 \right)}
      %+ e^{-2 s \frac{3\beta+1}{\beta+1}}
    \right),
  \end{align*}
  so that by combining the previous two equations we obtain 
  \begin{multline}
    \label{eq:lem-matching-w-tilde2}
    \begin{pmatrix}
      \widetilde w_1 \\ \widetilde w_2
    \end{pmatrix}
    =
    \left( (\beta+1)^{-\frac{\beta}{\beta+1}} S_* + A_0 \right) 
    \begin{pmatrix}
      1 \\ 1
    \end{pmatrix}
    e^{-s}
    +
    B_0
    \begin{pmatrix}
      1 \\ -1
    \end{pmatrix}
    e^{-s\frac{3\beta+1}{\beta+1}}
    \\
    +
    O \left( e^{-2s} + e^{-3s}
      + e^{-s \left( \frac{3\beta+1}{\beta+1} +1 \right)}
    \right)
  \end{multline}
  as a second representation for $\widetilde w$.
  With $S_* = - (\beta+1)^{\beta/(\beta+1)} A_0$ and $B_*=B_0$ the
  term of order $e^{-s}$ in \eqref{eq:lem-matching-w-tilde2} vanishes
  and thus \eqref{eq:lem-matching-w-tilde1} implies that there is no
  contribution of order $e^{-2s}$, either. Comparing the orders of the
  remaining error terms we arrive at
  \begin{equation*}
    %\label{eq:lem-matching-w-tilde3}
    \begin{pmatrix}
      \widetilde w_1 \\ \widetilde w_2
    \end{pmatrix}
    =
    B_*
    \begin{pmatrix}
      1 \\ -1
    \end{pmatrix}
    e^{-s\frac{3\beta+1}{\beta+1}}
    \\
    +
    O \left(
      e^{-s \left( \frac{3\beta+1}{\beta+1} +1 \right)}
    \right)
  \end{equation*}
  for $s \geq s_{S_*}$ or, equivalently, 
  \begin{equation}
    \label{eq:lem-matching-step-2-res}
    \begin{pmatrix}
      \widetilde Y_1 \\ \widetilde Y_2
    \end{pmatrix}
    =
    (\beta+1)^{\frac{1}{\beta+1}} (-t)^{\frac{1}{\beta+1}}
    \begin{pmatrix}
      1 \\ 1
    \end{pmatrix}
    +
    B_*
    (-t)^{-\frac{3\beta}{\beta+1}}
    \begin{pmatrix}
      1 \\ -1
    \end{pmatrix}
    +
    O \left( (-t)^{-\frac{3\beta}{\beta+1}-1} \right)
  \end{equation}
  for $t \leq - e^{s_{S_*}}$.

  \emph{Step 3: Properties of $B_*$ and $S_*$.}
  Since $\widetilde Y_1 > \widetilde Y_2$, we obviously have $B_* \geq
  0$. Moreover, if $B_*=0$ then \eqref{eq:lem-matching-eq-w} implies
  $\widetilde w \equiv 0$, which contradicts the construction of
  $\widetilde Y$. Thus, we have $B_*>0$.

  Next, we show $S_*<1$ by comparing $\widetilde Y$ with $\bar
  Y(t) = [(\beta+1)(-t)]^{1/(\beta+1)}$.
  Suppose for contradiction that $S_* \geq 1$ so that $\widetilde
  Y_1(0) > 0$ and $\widetilde Y_2(0) \geq 0$. Then we have $\widetilde
  Y_1(t) > \bar Y(t)$ for small $(-t)$, but also $\widetilde
  Y_2(t) > \bar Y(t)$, because either $\widetilde Y_2(0)>0$ or
  \begin{equation*}
    \tderiv{t} \widetilde Y_2
    =
    -2 \widetilde Y_2^{-\beta} + \widetilde Y_1^{-\beta}
    <
    - \tfrac{3}{2} \widetilde Y_2^{-\beta}
  \end{equation*}
  for small $(-t)$ such that $\widetilde Y_2^\beta < \widetilde
  Y_1^\beta /2$.
  Assume now that there is a first time $t_0<0$ such that $\widetilde
  Y_1(t_0) = \bar Y(t_0)$ or $\widetilde Y_2(t_0) = \bar
  Y(t_0)$. Then clearly $\widetilde Y_1(t_0) \not= \widetilde
  Y_2(t_0)$, because otherwise uniqueness for the ODE yields
  $\widetilde Y_1(t) = \widetilde Y_2(t) = \bar Y(t)$ for all
  $t<0$ in contradiction to the asymptotics of $\widetilde Y_p$ proved
  above. However, if $\widetilde Y_2(t_0) = \bar Y(t_0) <
  \widetilde Y_1(t_0)$ then we find
  \begin{equation*}
    0
    \leq
    \tderiv{t}( \widetilde Y_2 - \bar Y)(t_0)
    =
    \widetilde Y_1(t_0)^{-\beta} - \bar Y(t_0)^{-\beta}
    <
    0,
  \end{equation*}
  and in the same way the case $\widetilde Y_1(t_0) = \bar Y(t_0)
  < \widetilde Y_2(t_0)$ is excluded. Thus, we conclude that
  $\widetilde Y_1(t) > \bar Y(t)$ and $\widetilde Y_2(t) >
  \bar Y(t)$ for all $t<0$, which again contradicts the
  asymptotics of $\widetilde Y$.

  To prove $S_*>-1$, we consider the phase portrait of $(Z_1,Z_2)$
  where $Z_p(s) = \widetilde W_p((\beta+1)s) / (\beta+1)^{1/(\beta+1)}$,
  $p=1,2$. From the previous steps we know that
  \begin{equation}
    \label{eq:Zeq}
    \begin{aligned}
      \tderiv{s} Z_1 &= - Z_1 + 2 Z_1^{-\beta} - Z_2^{-\beta},
      \\
      \tderiv{s} Z_2 &= - Z_2 + 2 Z_2^{-\beta} - Z_1^{-\beta},
    \end{aligned}
  \end{equation}
  and that the trajectory $(Z_1,Z_2)(s)$ approaches the point $(1,1)$
  parallel to the vector $(-1,1)$ as $s \to \infty$. On the other
  hand, if $S_* \leq 1$ we have $Z_2(s_0)=0$ and $0<Z_1(s_0)\leq
  2^{1/(\beta+1)}$ where $s_0$ is the time corresponding to $t=S_*-1$,
  and a calculation shows that on the line which connects $(1,1)$ to
  $(2^{1/(\beta+1)},0)$ the dynamics of \eqref{eq:Zeq} point to the
  left of that line if $\beta \geq \beta_*$; compare Figure
  \ref{fig:phase_portrait}. As this contradicts the asymptotics of
  $(Z_1,Z_2)$, we must have $S_*>-1$.

  \emph{Step 4: Adjusting to initial data and right hand side.}
  For $|\nu|$ and $\theta$ such that
  \begin{equation*}
    %\label{eq:lem-matching-nu-theta}
    |\nu| \leq (\beta+1)^{\frac{\beta}{\beta+1}}
    T^{-\frac{3\beta+1}{\beta+1}},
    \qquad
    |\theta^{-\frac{3\beta+1}{\beta+1}}-1| \leq B_*^{-1} T^{-1}
  \end{equation*}
  we set
  \begin{equation*}
    \widehat Y_p(t; \nu,\theta)
    =
    \theta^{-\frac{1}{\beta+1}}
    \widetilde Y_p(\theta(t-\nu)).
  \end{equation*}
  Then, $\widehat Y$ solves the same equation as $\widetilde Y$ and with
  \eqref{eq:lem-matching-step-2-res} we find
  \begin{multline}
    \label{eq:lem-matching-shiftscale}
    \widehat Y(-T; \nu,\theta)
    -
    (\beta+1)^{\frac{1}{\beta+1}} T^{\frac{1}{\beta+1}}
    \begin{pmatrix}
      1 \\ 1
    \end{pmatrix}
    -
    B_* T^{-\frac{3\beta}{\beta+1}}
    \begin{pmatrix}
      1 \\ -1
    \end{pmatrix}
    \\
    \begin{aligned}
      &=
      \left[ \nu (\beta+1)^{-\frac{\beta}{\beta+1}} T^{-\frac{\beta}{\beta+1}}
        + O \left(\nu^2 T^{-\frac{\beta}{\beta+1}-1}\right)
      \right]
      \begin{pmatrix}
        1 \\ 1
      \end{pmatrix}
      \\
      &\quad+
      \left[
        \left( \theta^{-\frac{3\beta+1}{\beta+1}}-1 \right)
        B_* T^{-\frac{3\beta}{\beta+1}}
        + O \left( \nu \theta^{-\frac{3\beta+1}{\beta+1}}
          T^{-\frac{3\beta}{\beta+1}-1} \right)
      \right]
      \begin{pmatrix}
        1 \\ -1
      \end{pmatrix}
      + O \left( T^{-\frac{3\beta}{\beta+1}-1} \right)
      \\
      &=
      O \left( T^{-\frac{3\beta}{\beta+1}-1} \right).
    \end{aligned}
  \end{multline}
  Moreover, the vanishing times of $\widehat Y$ are $\widehat \tau_1 =
  \nu + \theta (S_*+1)$ and $\widehat \tau_2 = \nu + \theta (S_*-1)$.

  Let now $Y_1, Y_2$ be a solution of 
  \begin{equation*}
    \dot Y_1 = -2 Y_1^{-\beta} + F_1
  \end{equation*}
  in $(\widehat \tau_2, \widehat \tau_1)$ with terminal data
  $Y_1(\widehat \tau_1)=0$ and of \eqref{eq:local-eq} in
  $(-\infty,\widehat \tau_2)$ with $Y_2(\widehat \tau_2) = 0$ and
  $Y_1(\widehat \tau_2)$ given by its evolution in $(\widehat \tau_2,
  \widehat \tau_1)$.
  If $\eta_0$ is sufficiently small, existence and uniqueness follow
  from the same arguments as in Step 1, and a Gronwall argument
  implies
  \begin{equation*}
    | Y_p(t) - \widehat Y_p(t) |
    \leq
    C (-t) \eta
  \end{equation*}
  for $t< \widehat \tau_2$, because $Y_p$ and $\widehat Y_p$ behave
  like a power law near their vanishing time and are bounded from
  below once they have reached a certain size depending on $\beta$ and
  $\eta$.
  Using \eqref{eq:lem-matching-shiftscale}, we therefore find
  \begin{equation*}
    Y(-T)
    -
    (\beta+1)^{\frac{1}{\beta+1}} T^{\frac{1}{\beta+1}}
    \begin{pmatrix}
      1 \\ 1
    \end{pmatrix}
    -
    B_* T^{-\frac{3\beta}{\beta+1}}
    \begin{pmatrix}
      1 \\ -1
    \end{pmatrix}
    =
    O \left( \eta T + T^{-\frac{3\beta}{\beta+1}-1} \right).
  \end{equation*}
  The claim now follows from the inequality between $\eta$ and $T$ and
  from the fact that varying $\nu$ and $\theta$ sweeps a neighborhood
  of $[(\beta+1)T]^{1/(\beta+1)} (1,1) - B_* T^{-3\beta/(\beta+1)}
  (1,-1)$.
\end{proof}

\begin{figure}
  \centering
  \includegraphics[width=.5\textwidth]{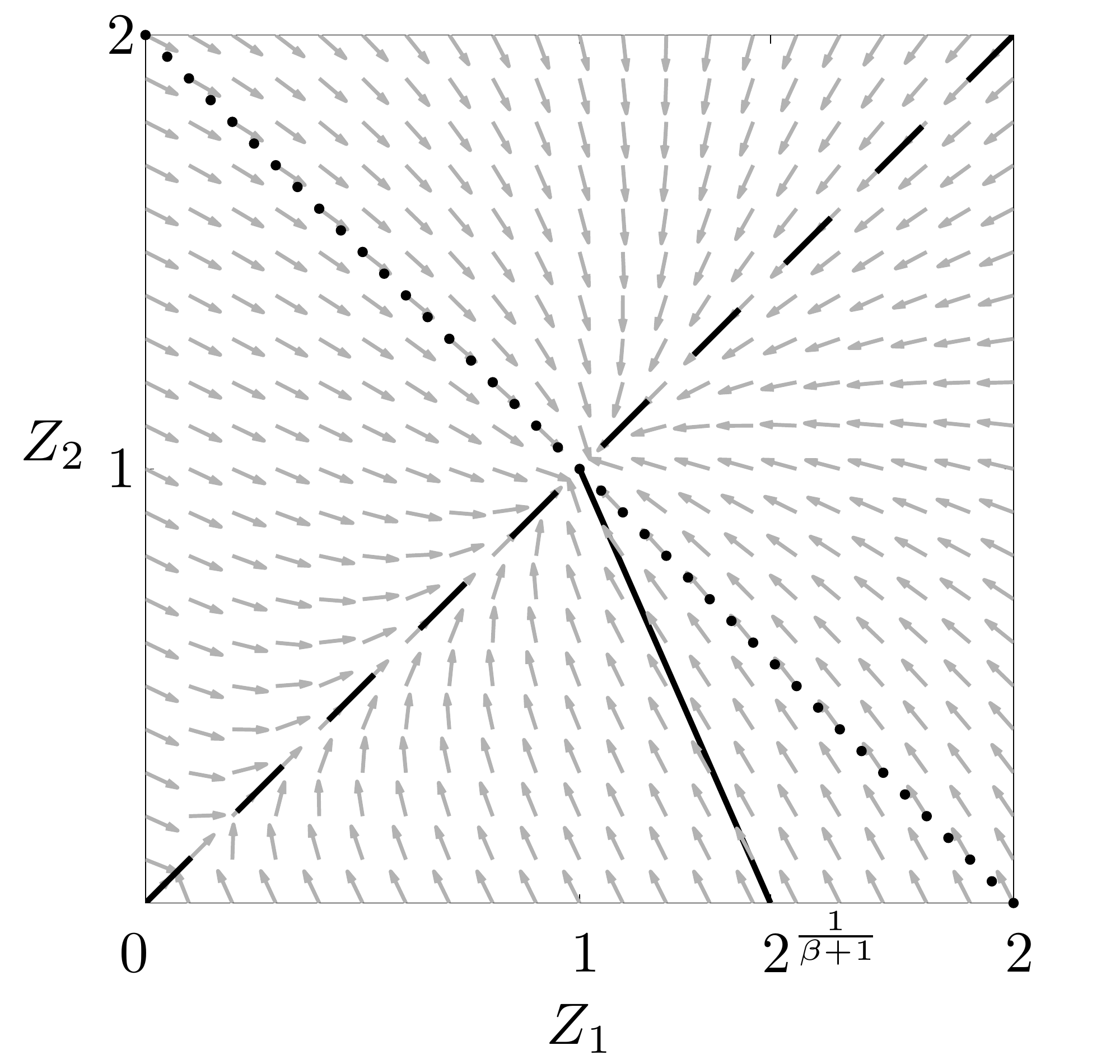}
  \caption{Phase portrait for \eqref{eq:Zeq} and $\beta\geq\beta_*$.
    Every solution $(Z_1,Z_2)(s)$ of \eqref{eq:Zeq} approaches the
    point $(1,1)$ as $s\to\infty$ either along the dashed line or the
    dotted line. In particular, the solution constructed in the proof
    of Lemma \ref{lem:matching} runs along the dotted line.
    This implies $S_*>-1$ because for $S \leq -1$ the initial values
    are in the region left of the solid line connecting $(1,1)$
    and $(2^{1/(\beta+1)},0)$, where the dynamics of \eqref{eq:Zeq}
    lead to the wrong asymptotics.}
  \label{fig:phase_portrait}
\end{figure}

\begin{remark}
  Showing that $S_*>-1$ is the only part of the non-uniqueness result
  where we use $\beta\geq\beta_*$. Indeed, our argument breaks down
  for $\beta<\beta_*$, because near the axis $Z_2=0$ the dynamics of
  \eqref{eq:Zeq} do not point to the left of the line connecting
  $(1,1)$ and $(2^{1/(\beta+1)},0)$. However, formal considerations as
  well as the phase portrait close to $(1,1)$ indicate that still only
  trajectories for $S_*>-1$ reach the point $(1,1)$ along the correct
  direction $(-1,1)$, yet a proof seems to require a more detailed
  investigation of the phase portrait, which is beyond the scope of
  the paper.
\end{remark}

Our final auxiliary result concerns the solution of an ODE that will
appear as linearization of the original equation.

\begin{lemma}[Solution of linear equation]
  \label{lem:solution-linear-eq}
  For $0 \leq T < \bar \tau$ let $Y=(Y_1,Y_2) \colon
  [T,\bar\tau) \to \bbR^2$ solve the equation
  \begin{equation}
    \label{eq:lem-lin-eq:lin-eq}
    \deriv{t}
    \begin{pmatrix}
      Y_1 \\ Y_2
    \end{pmatrix}
    -
    \frac{\beta}{\beta+1} \frac{1}{\bar \tau - t}
    \begin{pmatrix}
      2 & -1 \\ -1 & 2
    \end{pmatrix}
    \begin{pmatrix}
      Y_1 \\ Y_2
    \end{pmatrix}
    =
    \begin{pmatrix}
      F_1 \\ F_2
    \end{pmatrix}
  \end{equation}
  with initial data $Y(T) = Y^0 \in \bbR^2$ and continuous $F =
  (F_1,F_2) \colon [T,\bar\tau) \to \bbR^2$.
  Then the following statements are true.
  \begin{enumerate}
  \item If $F_1 \equiv F_2 \equiv 0$, then
    \begin{equation*}
      \begin{pmatrix}
        Y_1 \\ Y_2
      \end{pmatrix}
      (t)
      =
      \Phi(t;T,\bar\tau) Y^0,
    \end{equation*}
    where
    \begin{equation*}
      \Phi(t;T,\bar\tau)
      =
      \frac{1}{2} \left( \frac{\bar\tau-T}{\bar\tau-t}
      \right)^{\frac{\beta}{\beta+1}}
      \begin{pmatrix}
        1 & 1 \\ 1 & 1
      \end{pmatrix}
      +
      \frac{1}{2} \left( \frac{\bar\tau-T}{\bar\tau-t}
      \right)^{\frac{3\beta}{\beta+1}}
      \begin{pmatrix}
        +1 & -1 \\ -1 & +1
      \end{pmatrix}.
    \end{equation*}

  \item If $Y^0=(0,0)$ and $F_1, F_2$ are constant functions, then
    \begin{equation*}
      \begin{pmatrix}
        Y_1 \\ Y_2
      \end{pmatrix}
      (t)
      =
      (\bar\tau-t) A - (\bar\tau-T) \Phi(t;T,\bar\tau) A
    \end{equation*}
    where $A=(A_1,A_2)$ is the solution of
    \begin{equation*}
      -
      \begin{pmatrix}
        A_1 \\ A_2
      \end{pmatrix}
      =
      \frac{\beta}{\beta+1}
      \begin{pmatrix}
        2 & -1 \\ -1 & 2
      \end{pmatrix}
      \begin{pmatrix}
        A_1 \\ A_2
      \end{pmatrix}
      +
      \begin{pmatrix}
        F_1 \\ F_2
      \end{pmatrix}.
    \end{equation*}

  \item If $Y^0 = (0,0)$, then
    \begin{equation*}
      |Y_p(t)|
      \leq
      C
      \int_T^t \left(
        \frac{\bar\tau-s}{\bar\tau-t}\right)^{\frac{3\beta}{\beta+1}}
      \left( |F_1(s)|+|F_2(s)| \right) \,\dint{s}
    \end{equation*}
    for $p=1,2$ and all $t \in [T,\bar\tau)$.
  \end{enumerate}
\end{lemma}

\begin{proof}
  It is easily checked that $\Phi(t;T,\bar\tau)$ is a fundamental
  matrix for the homogeneous linear equation associated with
  \eqref{eq:lem-lin-eq:lin-eq} and that $\Phi(T;T,\bar\tau)$ is
  the identity matrix. Hence, the first claim follows immediately and
  the second by a direct computation.
  Finally, the variation of constants formula
  \begin{equation*}
    Y(t)
    =
    \int_T^t \Phi(t;T,\bar\tau) \Phi(s;T,\bar\tau)^{-1} F(s) \,\dint{s}
    =
    \int_T^t \Phi(t;s,\bar\tau) F(s) \,\dint{s}
  \end{equation*}
  implies the third assertion.
\end{proof}

% -----------------------------------------------------------------------------
% - Iterative estimates
% -----------------------------------------------------------------------------

\subsection{Iterative estimates}
\label{sec:iterative-estimates}

The aim of this section is to derive the main iterative estimates for
the non-uniqueness result. To this end, let $x^N$ and $\bar x^N$
be solutions to the truncated problem with initial data with and
without some $\eps^N \in (-1/8,1/8)$, respectively.
For simplicity of notation we drop the upper index $N$.

We denote by $\tau_k$ and $\bar \tau_k$, $k=-3N,\ldots,3N$ the
vanishing times of $x$ and $\bar x$, respectively, and define
\begin{equation*}
  T_j
  =
  \max \left\{ \tau_{3j+1}, \tau_{3j+2},
    \bar \tau_{3j+1} \right\}
\end{equation*}
for $j=0,\ldots,N-1$; for convenience we set $T_N=0$.
Lemma \ref{lem:van-times-continuity} immediately yields
\begin{equation*}
  c R_j^{\beta+1}
  \leq
  T_j
  \leq
  C R_j^{\beta+1},
\end{equation*}
for $j>N_*$, and by Proposition \ref{pro:idata-equal-van-times}, the
remark thereafter and the definition of $(R_m)$ we have
\begin{equation}
  \label{eq:T_j-diff}
  \min \{\tau_{3(j-1)+1}, \tau_{3(j-1)+2},
  \bar \tau_{3(j-1)+1} \} - T_{j}
  \geq
  c R_{j}^{\beta+1}
\end{equation}
as well as
\begin{equation*}
  \frac{\min \{\tau_{3(j-1)+1}, \tau_{3(j-1)+2},
    \bar \tau_{3(j-1)+1} \}}{T_j}
  \geq
  c
  >
  1
\end{equation*}
for $j>N_*$.
As a consequence, the vanishing times of $3(j-1)+p$ and $3j+p$ are
separated from each other and we may consider non-adjacent small
particles separately.

In the following we restrict our attention to the time interval
$[T_j,T_{j-1}]$ and we refer to the vanishing particles $3(j-1)+p$ and
their neighbors $3(j-1)$, $3j$ as the \emph{active particles}; the
others we call \emph{inactive particles}. To study the change of $x -
\bar x$ we group the inactive particles according to size and
location and set
\begin{alignat*}{2}
  &\| x - \bar x\|_{s,j}
  &\,:=
  &\sup_{\substack{N_* \leq k \leq j-2\\p=1,2}} \frac{\left|
      x_{3k+p} - \bar x_{3k+p} \right|}{R_k^{\beta+1}},
  \\
  &\| x - \bar x\|_{l,j}
  &\,:=
  &\sup_{N_* \leq k \leq j-2} \left| x_{3 k} - \bar x_{3 k} \right|,
  \\
  &\| x - \bar x\|_{r,j}
  &\,:=
  &\sup_{j+1 \leq k \leq N} \left| x_{3 k} - \bar x_{3 k} \right|.
\end{alignat*}
Moreover, we denote by 
\begin{equation*}
  D_{p,j} := x_{3(j-1)+p} - \bar x_{3(j-1)+p},
  \qquad p=0,1,2,3
\end{equation*}
the differences of the active particles.

The following theorem is the key result of this section and summarizes
the propagation of a perturbation from the particles $3j$ at time
$T_j$ to the particles $3(j-1)$ in the time interval $[T_j,T_{j-1}]$.

\begin{theorem}[Iterative estimates]
  \label{thm:iterative-estimate}
  Let $\beta \geq \beta_*$ where $\beta_*$ is as in (\ref{eq:beta}).
  There exist $\delta_* \in (0,1)$, $N_* \in \bbN$, $\gamma_*
  \in (0,1/3)$, $a_*>0$ and $C_*>0$ such that for all $N>N_*$, $\delta
  \in (0,\delta^*)$ and $\gamma \in (0,\gamma^*)$ the following
  property is true.
  If for some $j=N_*+1,\ldots,N-1$ the particles $3j$ satisfy
  \begin{equation}
    \label{eq:iterative-ass-active-large-right}
    |D_{3,j}(T_j)| = \delta
  \end{equation}
  and the other particles are estimated by
  \begin{equation}
    \label{eq:iterative-ass-active-other}
    |D_{0,j}(T_j)| \leq \delta \delta_*,
    \qquad
    |D_{1,j}(T_j)| \leq R_{j-1}^{\beta+1} \delta \delta_*,
    \qquad
    |D_{2,j}(T_j)| \leq R_{j-1}^{\beta+1} \delta \delta_*
  \end{equation}
  as well as
  \begin{equation}
    \label{eq:iterative-ass-inactive}
    \|x - \bar x\|_{s,j}(T_j) \leq \delta \delta_*,
    \qquad
    \|x - \bar x\|_{l,j}(T_j) \leq \delta \delta_*,
    \qquad
    \|x - \bar x\|_{r,j}(T_j) \leq 2\delta,
  \end{equation}
  then there is
  \begin{equation*}
    \bar \delta
    =
    \bar \delta(\delta)
    \in
    \left[\tfrac{3}{4} a_* \left(
        R_{j-1}^{4\beta+1} \delta \right)^\frac{1}{3\beta+1}, \tfrac{5}{4}
      a_* \left( R_{j-1}^{4\beta+1} \delta \right)^\frac{1}{3\beta+1}
    \right]
  \end{equation*}
  such that we have
  \begin{equation*}
    |D_{3,j-1}(T_{j-1})|
    =
    |D_{0,j}(T_{j-1})|
    =
    \bar \delta
  \end{equation*}
  and
  \begin{equation*}
    |D_{0,j-1}(T_{j-1})| \leq \bar \delta \delta_*,
    \qquad
    |D_{1,j-1}(T_{j-1})| \leq R_{j-2}^{\beta+1} \bar \delta \delta_*,
    \qquad
    |D_{2,j-1}(T_{j-1})| \leq R_{j-2}^{\beta+1} \bar \delta \delta_*
  \end{equation*}
  as well as
  \begin{equation*}
    \|x - \bar x\|_{s,j-1}(T_{j-1}) \leq \bar \delta \delta_*,
    \qquad
    \|x - \bar x\|_{l,j-1}(T_{j-1}) \leq \bar \delta \delta_*,
    \qquad
    \|x - \bar x\|_{r,j-1}(T_{j-1}) \leq 2 \bar \delta,
  \end{equation*}
  provided that $C_* \delta \leq \bar \delta$.
\end{theorem}

For the rest of the section we drop the index $j$ wherever it is not
necessary and write $\|\cdot\|_s$, $D_p$ and so on.
Moreover, it is convenient to abbreviate
\begin{equation*}
  D_{0}^\beta
  :=
  x_{3(j-1)}^{-\beta} - \bar x_{3(j-1)}^{-\beta}
  \qquad\text{and}\qquad
  D_{3}^\beta
  :=
  x_{3j}^{-\beta} - \bar x_{3j}^{-\beta}.
\end{equation*}
If not stated otherwise, we also assume $N_* \leq j<N$ where $N_*$ is
sufficiently large so that all previous results apply.

The next lemma bounds the change of particles up to time $T_{j-1}$,
and Lemmas \ref{lem:inactive-I}--\ref{lem:active-large-I} provide a
few simple estimates for the particle differences.

\begin{lemma}[Constancy of left inactive particles and all large
  particles]
  \label{lem:constancy}
  For all $t \in [0,T_{j-1}]$ we have
  \begin{align*}
    %|x_k(t) - 1| &\leq C \eps R_j^{\beta+1}
    %& &\text{for } k=-3N,\ldots,-1
    %\\
    %|x_{3k}(t) - x_{3k}(0)| &\leq C \gamma^{j-1-k} R_{k}
    %& &\text{for } k=N_*,\ldots,j-2,
    %\\
    |x_{3k+p}(t) - x_{3k+p}(0)| &\leq C \gamma^{j-1-k} R_{k}
    & &\text{for } k=N_*,\ldots,j-2,\quad p=0,1,2,
    \\
    |x_{3k}(t) - x_{3k}(0)| &\leq C R_{j-1}
    & &\text{for } k=j-1,\ldots,N-1.
  \end{align*}
  In particular, choosing $\gamma$ sufficiently small, we obtain
  \begin{equation*}
    x_{3k+p}(t) \geq \tfrac{1}{4} R_k
  \end{equation*}
  for all $t \in [0,T_{j-1}]$ and $k=N_*,\ldots,j-2$.
  Similar inequalities hold for $\bar x$.
\end{lemma}

\begin{proof}
  From Lemma \ref{lem:van-times-continuity} and inequality
  \eqref{eq:T_j-diff} we know that for all $t \in [0,T_{j-1}]$ we have
  \begin{equation*}
    x_{3k+p}(t)
    \geq
    c(\tau_{3k+p}-T_{j-1})^{\frac{1}{\beta+1}}
    \geq c R_{j-1}
  \end{equation*}
  for $k=N_*,\ldots,j-2$ and $p=1,2$.  Since clearly also
  \begin{equation*}
    x_{3k}(t) \geq \tfrac{1}{2} \geq R_{j-1}
  \end{equation*}
  for $k = 0,\ldots,j-1$, we conclude that
  \begin{equation*}
    | x_{3k+p}(t) - x_{3k+p}(0) |
    \leq
    \int_0^{T_{j-1}} |\dot x_{3k+p}(s)| \,\dint{s}
    \leq
    C R_{j-1}^{-\beta} T_{j-1}
    \leq
    C R_{j-1}
    \leq
    C \gamma^{j-1-k} R_{k}
  \end{equation*}
  for $k=N_*,\ldots,j-2$ and $p=0,1,2$. This proves the first
  inequality.

  For the third note that up to time $T_{j-1}$ each particle $x_{3k}$,
  $k \geq j$ has seen at most four different small neighbors whose
  contributions to its velocity are of size
  \begin{equation*}
    \int_0^\tau y^{-\beta} \,\dint{t}
    \leq
    C \int_{0}^\tau \left( \tau - t \right)^{-\frac{\beta}{\beta+1}}
    \,\dint{t}
    \leq
    C \tau^{\frac{1}{\beta+1}}
  \end{equation*}
  where $y$ denotes a neighbor with vanishing time $\tau$. Since
  moreover large particles remain large throughout the evolution we
  deduce
  \begin{equation*}
    | x_{3k}(t) - x_{3k}(0) |
    \leq
    C \left(
      R_{j-1}^{\beta+1} + R_k + R_{k-1}
    \right)
    \leq C R_{j-1}
  \end{equation*}
  for all $t \in [0,T_{j-1}]$ and $k \geq j$.
  Finally, for $k=3(j-1)$ the particles $x_{3(j-1)+p}$, $p=1,2$ vanish
  before the time $T_{j-1}$, whereas we have $x_{3(j-2)+2} \geq c
  R_{j-1}$. Thus, combining the two arguments from above finishes
  the proof for $x$, and the proof for $\bar x$ is identical.
\end{proof}

\begin{lemma}[Estimates for inactive particles]
  \label{lem:inactive-I}
  We have 
  \begin{align}
    \label{eq:inactive-I-small}
    \| x - \bar x \|_s(t)    
    &\leq
    C \| x - \bar x \|_s(T_j)
    +
    C \frac{t-T_j}{R_{j-2}^{\beta+1}} \| x - \bar x \|_l(T_j)
    +
    \frac{C}{R_{j-2}^{\beta+1}}
    \int_{T_j}^t
    | D_0^\beta(s) |
    %\big| x_{3(j-1)}^{-\beta} - \bar x_{3(j-1)}^{-\beta} \big|(s)
    \,\dint{s},
    \\
    \label{eq:inactive-I-large}
    \| x - \bar x \|_l(t)
    &\leq
    C \| x - \bar x \|_l(T_j)
    +
    C (t-T_j) \| x - \bar x \|_s(T_j)
    +
    C \gamma^{\beta+1}
    \int_{T_j}^t
    | D_0^\beta(s) |
    \,\dint{s}
  \end{align}
  and
  \begin{align}
    \label{eq:inactive-I-right}
    \| x - \bar x \|_r(t)
    \leq
    C \| x - \bar x \|_r(T_j)
    +
    C \int_{T_j}^t
    | D_3^\beta(s) |
    %\big| x_{3j}^{-\beta} - \bar x_{3j}^{-\beta} \big|(s)
    \,\dint{s}
  \end{align}
  for all $t \in [T_j,T_{j-1}]$.
\end{lemma}

\begin{proof}
  From the Constancy Lemma \ref{lem:constancy} we know that up to the
  time $T_{j-1}$ the large and small particles on the left of the
  active ones remain close to their initial value $1$ and $R_k$,
  respectively. In the difference of the equations for $x$ and
  $\bar x$ we may thus linearize and find
  \begin{equation*}
    |x_{3k}^{-\beta} - \bar x_{3k}^{\beta}|
    \leq
    C |x_{3k} - \bar x_{3k}|
    \leq
    C \| x - \bar x\|_l
  \end{equation*}
  as well as
  \begin{equation*}
    |x_{3k+p}^{-\beta} - \bar x_{3k+p}^{\beta}|
    \leq
    C \frac{|x_{3k+p} - \bar x_{3k+p}|}{R_k^{\beta+1}}
    \leq
    C \| x - \bar x\|_s
  \end{equation*}
  for $k=N_*,\ldots,3(j-2)$ and $p=1,2$.  Consequently, we obtain
  \begin{equation}
    \label{eq:lem:inactive-I-large-1}
    \tderiv{t} \| x - \bar x\|_l
    \leq
    C \| x - \bar x\|_l + C \| x - \bar x\|_s
  \end{equation}
  and
  \begin{equation}
    \label{eq:lem:inactive-I-small-1}
    \tderiv{t} \| x - \bar x\|_s
    \leq
    \frac{C}{R_{j-2}^{\beta+1}} \| x - \bar x\|_l
    +
    \frac{C}{R_{j-2}^{\beta+1}} \| x - \bar x\|_s
    +
    \frac{1}{R_{j-2}^{\beta+1}} |D_0^\beta(s)|
    %\big| x_{3(j-1)}^{-\beta} - \bar x_{3(j-1)}^{-\beta} \big|
  \end{equation}
  for all $t \in [T_j,T_{j-1}]$, where the last term in
  \eqref{eq:lem:inactive-I-small-1} stems from the equation for
  $3(j-2)+2$.
  Applying Gronwall's inequality to \eqref{eq:lem:inactive-I-large-1}
  and using $t-T_j \leq 1$ to incorporate the
  resulting exponential terms into the constants, we infer that
  \begin{equation}
    \label{eq:lem:inactive-I-large-2}
    \| x - \bar x \|_l(t)
    \leq
    C \| x - \bar x \|_l(T_j)
    +
    C \int_{T_j}^t \| x - \bar x \|_s(s) \,\dint{s}.
  \end{equation}
  Similarly, we deduce from \eqref{eq:lem:inactive-I-small-1} and
  $t-T_j \leq R_{j-1}^{\beta+1} \leq R_{j-2}^{\beta+1}$ that
  \begin{equation}
    \label{eq:lem:inactive-I-small-2}
    \| x - \bar x \|_s(t)
    \leq
    C \| x - \bar x \|_s(T_j)
    +
    \frac{C}{R_{j-2}^{\beta+1}}
    \int_{T_j}^t \| x - \bar x \|_l(s) 
    +
    |D_0^\beta(s)|
    %\big| x_{3(j-1)}^{-\beta} - \bar x_{3(j-1)}^{-\beta} \big|(s)
    \,\dint{s},
  \end{equation}
  and combining \eqref{eq:lem:inactive-I-large-2},
  \eqref{eq:lem:inactive-I-small-2} and
  \begin{equation*}
    \int_{T_j}^t \int_{T_j}^s \| x - \bar x \|_s(r) \,\dint{r} \,\dint{s}
    =
    \int_{T_j}^t \left(t-r\right) \| x - \bar x \|_s(r) \,\dint{r}
    \leq
    R_{j-1}^{\beta+1} \int_{T_j}^t \| x - \bar x \|_s(r) \,\dint{r}
  \end{equation*}
  we arrive at
  \begin{multline*}
    \| x - \bar x \|_s(t)
    \leq
    C \| x - \bar x \|_s(T_j)
    +
    C \frac{t-T_j}{R_{j-2}^{\beta+1}} \| x - \bar x \|_l(T_j)
    \\
    +
    %C \left(\frac{R_{j-1}}{R_{j-2}}\right)^{\beta+1}
    C \gamma^{\beta+1}
    \int_{T_j}^t \| x - \bar x \|_s(s) \,\dint{s} 
    +
    \frac{C}{R_{j-2}^{\beta+1}}
    \int_{T_j}^t
    |D_0^\beta(s)|
    %\big| x_{3(j-1)}^{-\beta} - \bar x_{3(j-1)}^{-\beta} \big|(s)
    \,\dint{s}.
  \end{multline*}
  Now, a Gronwall argument for $t \mapsto \int_{T_j}^t \| x - \bar
  x \| \,\dint{s}$ proves \eqref{eq:inactive-I-small}. The
  inequalities \eqref{eq:inactive-I-large} and
  \eqref{eq:inactive-I-right} are derived analogously.
\end{proof}

\begin{lemma}[Estimates for active large particles I]
  \label{lem:active-large-I}
  We have
  \begin{multline}
    \label{eq:active-large-I-left}
    |D_0(t)|
    \leq
    C |D_0(T_j)|
    +
    C (t-T_j) \left( \| x - \bar x \|_s(T_j)
      +
      %C \left(\frac{R_{j-1}}{R_{j-2}}\right)^{\beta+1}
      \gamma^{\beta+1}
      \| x - \bar x \|_l(T_j) \right)
    \\
    +
    C \int_{T_j}^t \big| x_{\sigma_+(3(j-1))}^{-\beta}
    - \bar x_{\sigma_+(3(j-1))}^{-\beta} \big|(s) \,\dint{s}
  \end{multline}
  and
  \begin{equation}
    \label{eq:active-large-I-right}
    |D_3(t)|
    \leq
    C |D_3(T_j)|
    +
    C (t-T_j) \| x - \bar x \|_r(T_j)
    +
    C \int_{T_j}^t \big| x_{\sigma_-(3j)}^{-\beta}
    - \bar x_{\sigma_-(3j)}^{-\beta} \big|(s) \,\dint{s}
  \end{equation}
  for all $t \in [T_j,T_{j-1}]$.
\end{lemma}

\begin{proof}
  Similar to the proof of the previous lemma we linearize in the
  derivative of $D_0$ to find
  \begin{equation*}
    \tderiv{t} |D_0|
    \leq
    C \| x - \bar x \|_s + C | D_0 |
    +
    \big| x_{\sigma_+(3(j-1))}^{-\beta}
    - \bar x_{\sigma_+(3(j-1))}^{-\beta} \big|.
  \end{equation*}
  Note that $\sigma_+(3(j-1))$ may take the values $3(j-1)+p$,
  $p=1,2,3$ in the time interval $[T_j,T_{j-1}]$ and may be different
  for $x$ and $\bar x$.
  With Gronwall's inequality we get
  \begin{equation}
    \label{eq:lem:active-large-I-1}
    |D_0(t)|
    \leq
    C | D_0(T_j)|
    +
    C \int_{T_j}^t \| x - \bar x \|_s \,\dint{s}
    +
    C \int_{T_j}^t \big| x_{\sigma_+(3(j-1))}^{-\beta}
    - \bar x_{\sigma_+(3(j-1))}^{-\beta} \big| \,\dint{s},
  \end{equation}
  and from \eqref{eq:inactive-I-small} with $|D_0^\beta| \leq C|D_0|$
  we obtain
  \begin{multline}
    \label{eq:lem:active-large-I-2}
    \int_{T_j}^t \| x - \bar x \|_s(s) \,\dint{s}
    \leq
    C (t-T_j) \left( \| x - \bar x \|_s(T_j)
      +
      %C \left(\frac{R_{j-1}}{R_{j-2}}\right)^{\beta+1}
      C \gamma^{\beta+1}
      \| x - \bar x \|_l(T_j) \right)
    \\
    +
    C \gamma^{\beta+1}
    %C \left(\frac{R_{j-1}}{R_{j-2}}\right)^{\beta+1}
    \int_{T_j}^t |D_0(s)| \,\dint{s}.
  \end{multline}
  Combining \eqref{eq:lem:active-large-I-2} with
  \eqref{eq:lem:active-large-I-1} and applying Gronwall's inequality
  to $t \mapsto \int_{T_j}^t |D_0(s)| \,\dint{s}$ proves
  \eqref{eq:active-large-I-left}. The derivation of
  \eqref{eq:active-large-I-right} is similar.
\end{proof}

Next, we consider the active small particles, and our aim is to derive
an approximate solution formula for $D_p$, $p=1,2$. However, since the
equation for $D_p$ contains $D_0^\beta$ or $D_3^\beta$, we need
precise estimates for the latter.

\begin{lemma}[Estimates for $D_0^\beta$ and $D_3^\beta$]
  \label{lem:active-large-II}
  Suppose that
  \begin{equation}
    \label{eq:closeness}
    |D_p(t)|
    \leq
    \tfrac{1}{2} \bar x_{3(j-1)+p}(t),
    \qquad
    p=1,2
  \end{equation}
  in some time interval $[T_j,T_*]$ where $T_j < T_* \leq \min \set{
    \tau_{3(j-1)+1}, \tau_{3(j-1)+2}, \bar \tau_{3(j-1)+1}}$.
  Then we have
  \begin{multline}
    \label{eq:active-large-II-left}
    |D_0^\beta(t)-D_0^\beta(T_j)|
    \leq
    C (t-T_j) \left( \| x - \bar x \|_s(T_j)
      +
      %C \left(\frac{R_{j-1}}{R_{j-2}}\right)^{\beta+1}
      C \gamma^{\beta+1}
      \| x - \bar x \|_l(T_j) \right)
    \\
    +
    C R_{j-1} | D_0^\beta(T_j)|
    +
    C \int_{T_j}^t \frac{|D_1(s)| \,\dint{s}}{\bar \tau_{3(j-1)+1}-s}
  \end{multline}
  and
  \begin{equation}
    \label{eq:active-large-II-right}
    |D_3^\beta(t)-D_3^\beta(T_j)|
    \leq
    C (t-T_j) \|x - \bar x\|_r(T_j)
    +
    C R_{j-1} |D_3^\beta(T_j)|
    +
    C \int_{T_j}^t \frac{|D_2(s)| \,\dint{s}}{\bar \tau_{3(j-1)+2}-s}
  \end{equation}
  for all $t \in [T_j,T_*]$.
\end{lemma}

\begin{proof}
  In
  \begin{equation}
    \label{eq:lem:large-active-II-1}
    \tderiv{t} D_0^\beta
    =
    - \beta \bar x_{3(j-1)}^{-(\beta+1)}
    \tderiv{t} D_0
    %\deriv{t} \left( x_{3(j-1)} - \bar x_{3(j-1)} \right)
    +
    \beta \left( \bar x_{3(j-1)}^{-(\beta+1)}
      - x_{3(j-1)}^{-(\beta+1)} \right)
    \dot x_{3(j-1)}
  \end{equation}
  we estimate the first term on the right hand side by
  \begin{align*}
    \left| \beta \bar x_{3(j-1)}^{-(\beta+1)} \tderiv{t} D_0 \right|
    &\leq
    C \| x - \bar x \|_s + C |D_0^\beta|
    +
    C \left| x_{3(j-1)+1}^{-\beta} - \bar x_{3(j-1)+1}^{-\beta} \right|
    \\
    &\leq
    C \| x - \bar x \|_s + C |D_0^\beta|
    +
    C \frac{|D_1|}{\bar x_{3(j-1)+1}^{\beta+1}},
  \end{align*}
  using that \eqref{eq:closeness} implies $x_{3(j-1)+1} \geq \bar
  x_{3(j-1)+1}/2$ and thus
  \begin{equation*}
    \left| x_{3(j-1)+1}^{-\beta} - \bar x_{3(j-1)+1}^{-\beta} \right|
    \leq
    C \frac{|D_1|}{\bar x_{3(j-1)+p}^{\beta+1}}.
  \end{equation*}
  In the second term of \eqref{eq:lem:large-active-II-1} we use
  \begin{equation*}
    \left| \bar x_{3(j-1)}^{-(\beta+1)} - x_{3(j-1)}^{-(\beta+1)} \right|
    =
    \left|
      \left(\bar x_{3(j-1)}^{-\beta}\right)^{\frac{\beta+1}{\beta}}
      -
      \left(x_{3(j-1)}^{-\beta}\right)^{\frac{\beta+1}{\beta}}
    \right|
    \leq
    C |D_0^\beta|,
  \end{equation*}
  where we have linearized $s \mapsto s^{(\beta+1)/\beta}$, and
  \begin{equation*}
    |\dot x_{3(j-1)}|
    \leq
    C \bar x_{3(j-1)+1}^{-\beta} 
  \end{equation*}
  noting that up to time $T_{j-1}$ the particles $x_{3(j-1)}$ and
  $x_{3(j-2)+2}$ remain large compared to $x_{3(j-1)+1}$ due to the
  Constancy Lemma \ref{lem:constancy}.
  As a consequence, \eqref{eq:lem:large-active-II-1} yields
  \begin{equation*}
    \left| \tderiv{t} D_0^\beta \right|
    \leq
    C \bigg( \| x - \bar x \|_s
      +
      \frac{|D_1|}{\bar x_{3(j-1)+1}^{\beta+1}}
      +
      \frac{|D_0^\beta|}{\bar x_{3(j-1)+1}^{\beta}}
    \bigg)
  \end{equation*}
  and after rewriting this estimate as
  \begin{equation*}
    \left| \tderiv{t}\big(D_0^\beta-D_0^\beta(T_j) \big) \right|
    \leq
    C \bigg( \| x - \bar x \|_s +
      \frac{|D_1|}{\bar x_{3(j-1)+1}^{\beta+1}} +
      \frac{|D_0^\beta-D_0^\beta(T_j)|}{\bar x_{3(j-1)+1}^{\beta}} +
      \frac{|D_0^\beta(T_j)|}{\bar x_{3(j-1)+1}^{\beta}}
    \bigg)
  \end{equation*}
  we use Gronwall's inequality to obtain
  \begin{multline*}
    |D_0^\beta(t)-D_0^\beta(T_j)|
    \\
    \leq
    C \int_{T_j}^t
    \bigg( \| x - \bar x \|_s +
      \frac{|D_1(s)|}{\bar x_{3(j-1)+1}(s)^{\beta+1}} +
      \frac{|D_0^\beta(T_j)|}{\bar x_{3(j-1)+1}(s)^{\beta}}
    \bigg)
    \exp\left( \int_s^t \frac{C\,\dint{r}}{\bar x_{3(j-1)+1}(r)^\beta}
      \right) \dint{s}.
  \end{multline*}
  Since the lower bound \eqref{eq:exist:lower-bound} implies
  \begin{equation*}
    \int_{T_j}^t \frac{\dint{s}}{\bar x_{3(j-1)+1}(s)^{\beta}}
    \leq
    C \int_{T_j}^t ( \bar \tau_{3(j-1)+1} - t)^{-\frac{\beta}{\beta+1}}
    \,\dint{s}
    \leq
    (\bar \tau_{3(j-1)+1} - T_j)^{\frac{1}{\beta+1}}
    \leq
    C R_{j-1},
  \end{equation*}
  we conclude
  \begin{equation*}
    |D_0^\beta(t)-D_0^\beta(T_j)|
    \leq
    C \int_{T_j}^t
    \left(
      \| x - \bar x \|_s(s)
      + 
      \frac{|D_1(s)|}{\bar x_{3(j-1)+1}(s)^{\beta+1}}
    \right) \dint{s}
    +
    C R_{j-1} |D_0^\beta(T_j)|.
  \end{equation*}
  Using now \eqref{eq:inactive-I-small} for $\|x - \bar
  x\|_s$ and Gronwall's inequality for $t \mapsto \int_{T_j}^t 
  |D_0^\beta(s)-D_0^\beta(T_j)| \,\dint{s}$ finishes the proof of
  \eqref{eq:active-large-II-left}. The proof of
  \eqref{eq:active-large-II-right} is similar.
\end{proof}

\begin{lemma}[Linearized equation for active small particles]
  \label{lem:active-small-linearized-eq}
  Suppose that \eqref{eq:closeness} holds up to time $T_*>T_j$. Then
  we have
  \begin{multline}
    \label{eq:active-small-linearized-eq}
    \deriv{t}
    \begin{pmatrix}
      D_1 \\ D_2
    \end{pmatrix}
    -
    \frac{\beta}{\beta+1} \frac{1}{\bar \tau_{3(j-1)+1}-t}
    \begin{pmatrix}
      2 & -1 \\ -1 & 2
    \end{pmatrix}
    \begin{pmatrix}
      D_1 \\ D_2
    \end{pmatrix}
    \\
    =
    \begin{pmatrix}
      \omega_1 \\ \omega_2
    \end{pmatrix}
    +
    \begin{pmatrix}
      D_0^\beta - D_0^\beta(T_j)
      \\
      D_3^\beta - D_3^\beta(T_j)
    \end{pmatrix}
    +
    \begin{pmatrix}
      D_0^\beta(T_j)
      \\
      D_3^\beta(T_j)
    \end{pmatrix}
  \end{multline}
  where
  \begin{equation*}
    |\omega_p|
    \leq
    C \frac{|D_1|^2 + |D_2|^2}
    {(\bar \tau_{3(j-1)+1}-t)^{\frac{\beta+2}{\beta+1}}}
    +
    C \frac{|D_1| + |D_2|}{(\bar \tau_{3(j-1)+1}-t)^{\frac{1}{\beta+1}}},
    \qquad
    p=1,2
  \end{equation*}
  for all $t \in [T_j,T_*]$.
\end{lemma}

\begin{proof}
  We employ \eqref{eq:closeness} to linearize the differences
  $x_{3(j-1)+p}^{-\beta} - \bar x_{3(j-1)+p}^{-\beta}$ around
  $\bar x_{3(j-1)+p}$ in
  \begin{equation*}
    \deriv{t}
    \begin{pmatrix}
      D_1 \\ D_2
    \end{pmatrix}
    =
    \begin{pmatrix}
      -2 & 1 \\ 1 & -2
    \end{pmatrix}
    \begin{pmatrix}
      x_{3(j-1)+1}^{-\beta} - \bar x_{3(j-1)+1}^{-\beta} \\
      x_{3(j-1)+2}^{-\beta} - \bar x_{3(j-1)+2}^{-\beta}
    \end{pmatrix}
    +
    \begin{pmatrix}
      D_0^\beta \\ D_3^\beta
    \end{pmatrix},
  \end{equation*}
  which is
  \begin{equation*}
    x_{3(j-1)+p}^{-\beta} - \bar x_{3(j-1)+p}^{-\beta}
    =
    -\beta \frac{D_p}{\bar x_{3(j-1)+p}^{\beta+1}}
    +
    O \Bigg( \frac{D_p^2}{\bar x_{3(j-1)+p}^{\beta+2}} \Bigg).
  \end{equation*}
  Then, replacing $\bar x_{3(j-1)+p}^{-(\beta+1)}$ with the power
  law obtained in Lemma \ref{lem:asymptotics}, that is
  \begin{equation*}
    \bar x_{3(j-1)+p}^{-(\beta+1)}
    =
    \frac{1}{\beta+1} \frac{1}{\bar \tau_{3(j-1)+p}-t}
    +
    O \left( (\bar \tau_{3(j-1)+p} - t)^{-\frac{1}{\beta+1}} \right),
  \end{equation*}
  and using the lower bound \eqref{eq:exist:lower-bound} to get
  \begin{equation*}
    \bar x_{3(j-1)+p}^{-(\beta+2)}
    \leq
    C \left( \bar \tau_{3(j-1)+p} - t \right)^{-\frac{\beta+2}{\beta+1}}
  \end{equation*}
  proves the claim.
\end{proof}

\begin{proposition}[Approximate solution for active small particles]
  \label{pro:active-small-approx}
  %Assume
  %\eqref{eq:iterative-ass-inactive}--\eqref{eq:iterative-ass-active-small}.
  %
  Given $\delta_* \in (0,1)$ and $\delta \in (0,\delta_*)$ there are
  constants $N_*=N_*(\beta,\delta_*)$ and $C=C(\beta,\delta_*)$ with
  the following property.
  If
  \eqref{eq:iterative-ass-active-large-right}--\eqref{eq:iterative-ass-inactive}
  are satisfied for some $j>N_*$ then we have
  \begin{equation*}
    %\label{eq:active-small-approx-van-time}
    \tau_{3(j-1)+p}
    >
    \bar \tau_{3(j-1)+1} - C
    \left( R_{j-1}^{4\beta+1} \delta \right)^\frac{\beta+1}{3\beta+1}
    =:
    \tau^*_{j-1}.
  \end{equation*}
  Moreover, \eqref{eq:closeness} holds for all $t \in
  [T_j,\tau^*_{j-1}]$ and we have
  \begin{equation*}
    \begin{pmatrix}
      D_1 \\ D_2
    \end{pmatrix}
    (t)
    =
    -
    \frac{D_3(T_j)}{\bar x_{3j}(T_j)^{\beta+1}}
    \phi(t; T_j, \bar \tau_{3(j-1)+1})
    +
    \begin{pmatrix}
      \rho_1 \\ \rho_2
    \end{pmatrix}
    (t)
  \end{equation*}
  where
  \begin{equation*}
    | \rho_p(t) |
    \leq
    C R_{j-1}^{\beta+1} \delta 
    \left(
      \frac{\bar\tau_{3(j-1)+1}-T_j}{\bar \tau_{3(j-1)+1}-t}
    \right)^{\frac{3\beta}{\beta+1}}
    \delta_*
  \end{equation*}
  and
  \begin{align*}
    \phi(t;T_j,\bar \tau_{3(j-1)+1})
    =
    &- ( \bar \tau_{3(j-1)+1}-t )
    \begin{pmatrix}
      C_1 \\ C_2
    \end{pmatrix}
    \\
    &+
    C_3 ( \bar \tau_{3(j-1)+1} - T_j )
    \left( \frac{\bar\tau_{3(j-1)+1}-T_j}{\bar\tau_{3(j-1)+1}-t}
    \right)^{\frac{\beta}{\beta+1}}
    \begin{pmatrix}
      1 \\ 1
    \end{pmatrix}
    \\
    &-
    C_4 ( \bar \tau_{3(j-1)+1} - T_j )
    \left( \frac{\bar\tau_{3(j-1)+1}-T_j}{\bar\tau_{3(j-1)+1}-t}
    \right)^{\frac{3\beta}{\beta+1}}
    \begin{pmatrix}
      1 \\ -1
    \end{pmatrix}
  \end{align*}
  with positive constants $C_1,C_2,C_3,C_4$.
\end{proposition}

\begin{proof}
  The Constancy Lemma \ref{lem:constancy} for time $T_j$ and $k=j-1$
  yields $\bar x_{3(j-1)+p}(T_j) \geq R_{j-1}/2$, whereas
  \eqref{eq:iterative-ass-active-other} implies $|D_p(T_j)| \leq
  R_{j-1}/8$ for $p=1,2$, all $\delta < \delta_* < 1$ and $j > N_*$ provided
  that we choose $N_*$ sufficiently large. Thus, \eqref{eq:closeness}
  is true for some $T_* \in (T_j,\bar \tau_{3(j-1)+1})$ and Lemmas
  \ref{lem:active-large-II} and \ref{lem:active-small-linearized-eq}
  apply.
  Moreover, \eqref{eq:iterative-ass-active-other} also provides
  \begin{equation}
    \label{eq:prop-approx-cont-ass}
    |D_1(t)|+|D_2(t)|
    \leq
    M R_{j-1}^{\beta+1} \delta
    \left( \frac{\bar\tau_{3(j-1)+1}-T_j}{\bar\tau_{3(j-1)+1}-t}
    \right)^{\frac{3\beta}{\beta+1}}
  \end{equation}
  for $t \in [T_j,T_*]$ and some constant $M>2$ that will be chosen
  below.

  We now proceed to solve the linearized equation
  \eqref{eq:active-small-linearized-eq} separately for the different
  contributions on its right hand side and the initial data. Then, we
  determine the time up to which our approximation is valid and the
  leading order terms of the solution.
  To simplify the notation, we abbreviate
  \begin{equation*}
    \bar \tau = \bar \tau_{3(j-1)+1},
    \qquad
    T = T_j
    \qquad\text{and}\qquad
    R = R_{j-1}
  \end{equation*}
  throughout the proof.

  \emph{Step 1: Initial data $D_p(T)$.}
  According to Lemma \ref{lem:solution-linear-eq}(1) the contribution
  of the initial data $D_p(T_j)$ to the solution of
  \eqref{eq:active-small-linearized-eq} with right hand side equal to
  zero is given by
  \begin{equation*}
    Z^{ID}(t) = \Phi(t;T, \bar \tau)
    \begin{pmatrix}
      D_1(T) \\ D_2(T)
    \end{pmatrix}.
  \end{equation*}
  By assumption \eqref{eq:iterative-ass-active-other} and the formula
  for $\Phi$ this is easily estimated as
  \begin{equation*}
    |Z^{ID}_1|+|Z^{ID}_2|
    \leq
    2 \left( \frac{\bar\tau-T}{\bar\tau-t}
    \right)^{\frac{3\beta}{\beta+1}}
    \big( |D_1(T)| + |D_2(T)| \big)
    \leq
    4 R^{\beta+1} \delta
    \left( \frac{\bar\tau-T}{\bar\tau-t}
    \right)^{\frac{3\beta}{\beta+1}} \delta_*.
  \end{equation*}

  \emph{Step 2: Right hand side $(D_0^\beta(T),D_3^\beta(T))$.}
  Here we use Lemma \ref{lem:solution-linear-eq}(2) separately for
  $(D_0^\beta(T),0)$ and $(0,D_3^\beta(T))$. Calling the
  solutions $Z^{R1}$ and $Z^{R2}$, respectively, we obtain after a
  straightforward computation that
  \begin{equation*}
    |Z^{R1}_p|
    % \leq
    % (\bar\tau-T) |D_0^\beta(T)|
    % \left(
    %   1
    %   +
    %   2 
    %   \left( \frac{\bar\tau-T}{\bar\tau-t}
    %   \right)^{\frac{3\beta}{\beta+1}}
    % \right)
    \leq
    C R^{\beta+1} \delta 
    \left( \frac{\bar\tau-T}{\bar\tau-t}
    \right)^{\frac{3\beta}{\beta+1}} \delta_*
    % \end{equation*}
    \qquad\text{and}\qquad
    % \begin{equation*}
    |Z_p^{R2}|
    \leq
    C R^{\beta+1} \delta 
    \left( \frac{\bar\tau-T}{\bar\tau-t}
    \right)^{\frac{3\beta}{\beta+1}}.
  \end{equation*}

  \emph{Step 3: Right hand side $(D_0^\beta(t)-D_0^\beta(T),
    D_3^\beta(t)-D_3^\beta(T))$.}
  Using the assumptions \eqref{eq:iterative-ass-inactive} and
  \eqref{eq:iterative-ass-active-large-right}--\eqref{eq:iterative-ass-active-other}
  for $p=0,3$ as well as the estimate \eqref{eq:prop-approx-cont-ass},
  we obtain from Lemma \ref{lem:active-large-II} that
  \begin{align*}
    |D_0^\beta(t) - D_0^\beta(T)|
    &\leq
    C R \delta \delta_*
    +
    C M R^{\beta+1} \delta
    \int_{T_j}^t
    \left( \frac{\bar\tau-T}{\bar\tau-s}
    \right)^{\frac{3\beta}{\beta+1}}
    \frac{\dint{s}}{\bar \tau-s}
    \\
    &\leq
    C R \delta \delta_*
    +
    C M R^{\beta+1} \delta
    \left( \frac{\bar\tau-T}{\bar\tau-t}
    \right)^{\frac{3\beta}{\beta+1}}
  \end{align*}
  and
  \begin{equation*}
    |D_3^\beta(t) - D_3^\beta(T)|
    \leq
    C R \delta
    +
    C M R^{\beta+1} \delta
    \left( \frac{\bar\tau-T}{\bar\tau-t}
    \right)^{\frac{3\beta}{\beta+1}}.
  \end{equation*}
  By Lemma \ref{lem:solution-linear-eq}(3) the corresponding solution
  $Z^{R0}$ may thus be estimated as
  \begin{align*}
    |Z^{R0}_p|
    &\leq
    C \delta
    \int_{T}^t \left(
      \frac{\bar\tau-s}{\bar \tau-t}
    \right)^{\frac{3\beta}{\beta+1}}
    \left(
      R + M R^{\beta+1}
      \left(
        \frac{\bar\tau-T}{\bar \tau-s}
      \right)^{\frac{3\beta}{\beta+1}}
    \right)
    \dint{s}
    \\
    &\leq
    C R^{\beta+1} \delta
    \left(
      \frac{\bar\tau-T}{\bar \tau-t}
    \right)^{\frac{3\beta}{\beta+1}}
    \left(
      R + M R^{\beta+1}
    \right).
  \end{align*}

  \emph{Step 4: Right hand side $\omega_p$.}
  Plugging \eqref{eq:prop-approx-cont-ass} into the estimate for
  $\omega_p$ from Lemma \ref{lem:active-small-linearized-eq} yields
  \begin{equation*}
    |\omega_p(s)|
    \leq
    \frac{C M^2 R^{2(\beta+1)} \delta^2}
    {(\bar \tau-s)^\frac{\beta+2}{\beta+1}}
    \left(
      \frac{\bar\tau-T}{\bar \tau-s}
    \right)^{\frac{6\beta}{\beta+1}}
    +
    \frac{C M R^{\beta+1} \delta}{(\bar \tau-s)^\frac{1}{\beta+1}}
    \left(
      \frac{\bar\tau-T}{\bar \tau-s}
    \right)^{\frac{3\beta}{\beta+1}}.
  \end{equation*}
  Considering both terms separately, we deduce from Lemma
  \ref{lem:solution-linear-eq}(3) that the associated solutions
  $Z^{\omega1}$ and $Z^{\omega2}$ satisfy
  \begin{align*}
    |Z^{\omega1}_p|
    &\leq
    C M^2 R^{2(\beta+1)} \delta^2
    \left(
      \frac{(\bar\tau-T)^2}{\bar \tau-t}
    \right)^{\frac{3\beta}{\beta+1}}
    \int_{T}^t (\bar\tau-s)^{-\frac{4\beta+2}{\beta+1}}
    \,\dint{s}
    \\
    &\leq
    C M^2 R^{2(\beta+1)} \delta^2
    \left(
      \frac{\bar\tau-T}{\bar \tau-t}
    \right)^{\frac{6\beta}{\beta+1}}
    \frac{1}{(\bar\tau-t)^\frac{1}{\beta+1}}
  \end{align*}
  and
  \begin{align*}
    |Z^{\omega2}_p|
    &\leq
    C M R^{\beta+1} \delta 
    \left(
      \frac{\bar\tau-T}{\bar \tau-t}
    \right)^{\frac{3\beta}{\beta+1}}
    \int_{T}^t (\bar\tau-s)^{-\frac{1}{\beta+1}}
    \,\dint{s}
    \\
    &\leq
    C M R^{\beta+1} \delta 
    \left(
      \frac{\bar\tau-T}{\bar \tau-t}
    \right)^{\frac{3\beta}{\beta+1}}
    R^\beta.
  \end{align*}

  \emph{Step 5: Continuation.}
  From the previous estimates we find that $(D_1,D_2)$ satisfies
  \begin{align*}
    |D_p|
    &\leq
    |Z^{ID}_p| + |Z^{R1}_p| + |Z^{R2}_p| + |Z^{R0}_p| +
    |Z^{\omega1}_p| + |Z^{\omega2}_p|
    \\
    &\leq
    R^{\beta+1} \delta
    \left(
      \frac{\bar\tau-T}{\bar \tau-t}
    \right)^{\frac{3\beta}{\beta+1}}
    C
    \left( \delta_* + 1 + R + M R^{\beta+1} + M R^\beta \right)
    +
    | Z_p^{\omega1} |.
  \end{align*}
  Since $\delta_* < 1$ we may choose $N_* \in \bbN$ and $M>0$, both
  depending only on $\beta$, such that
  \begin{equation*}
    C \left(\delta_* +1 + R + M R^{\beta+1} + M R^\beta \right)
    \leq
    \frac{M}{4}
  \end{equation*}
  for all $j > N_*$.
  Moreover, we can estimate $|Z_p^{\omega1}|$ similarly by $\nu M$,
  where $0<\nu<1/4$ will be chosen below, provided that
  \begin{equation}
    \label{eq:prop-approx-nu}
    C M R^{\beta+1} \delta
    \left(
      \frac{\bar\tau-T}{\bar \tau-t}
    \right)^{\frac{3\beta}{\beta+1}}
    \frac{1}{(\bar\tau-t)^\frac{1}{\beta+1}}
    \leq
    \nu,
  \end{equation}
  and due to $\bar\tau-T \leq C R^{\beta+1}$ this inequality is
  true if
  \begin{equation}
    \label{eq:prop-approx-time}
    \bar \tau-t
    \geq
    \left(
      \frac{C M}{\nu} R^{4\beta+1} \delta
    \right)^\frac{\beta+1}{3\beta+1}.
  \end{equation}
  Hence, for such $t \geq T$ also inequality
  \eqref{eq:prop-approx-cont-ass} holds and together with
  \eqref{eq:prop-approx-nu} implies
  \begin{equation*}
    |D_1|+|D_2|
    \leq
    \nu (\bar \tau-t)^{\frac{1}{\beta+1}}.
  \end{equation*}
  As a consequence, also \eqref{eq:closeness} and Steps 1--4 are valid
  up to the time $\tau^*_{j-1}$ defined by equality in
  \eqref{eq:prop-approx-time} if $\nu$ is sufficiently small.

  \emph{Step 6: Solution formula for $(D_1,D_2)$.}
  Now all contributions to $(D_1,D_2)$ except for $Z_p^{R2}$ are
  estimated by
  \begin{equation*}
    C M R^{\beta+1} \delta
    \left(
      \frac{\bar\tau-T}{\bar \tau-t}
    \right)^{\frac{3\beta}{\beta+1}}
    \left(
      \delta_* + R + M R^{\beta+1} + \nu + M R^\beta
    \right).
  \end{equation*}
  Therefore, by increasing $N_*$ and decreasing $\nu$ -- both thus depending
  on $\delta_*$ -- we obtain
  \begin{equation*}
    D_p = Z_p^{R2} + \rho_p
    \qquad\text{where}\qquad
    |\rho_p|
    \leq 
    C M R^{\beta+1} \delta
    \left(
      \frac{\bar\tau-T}{\bar \tau-t}
    \right)^{\frac{3\beta}{\beta+1}}
    \delta_*.
  \end{equation*}
  The solution formula then follows from the linearization
  \begin{equation*}
    D_3^\beta(T)
    =
    - \beta
    \frac{D_3(T)}{\bar x_{3j}(T)^{\beta+1}}
    %\left( x_{3j}(T) - \bar x_{3j}(T) \right)
    +
    O(\delta^2)
  \end{equation*}
  and Lemma \ref{lem:solution-linear-eq}(2), taking into account that
  the contribution of $O(\delta^2) \leq O(\delta \delta_*)$ can be
  estimated similar to $Z^{R1}$ and incorporated into $\rho_p$.
\end{proof}

\begin{corollary}
  \label{cor:active-small-leading-order}
  In the setting of Proposition \ref{pro:active-small-approx},
  we have
  \begin{equation*}
    |D_p(t)|
    \leq
    C R_{j-1}^{\beta+1} \delta
    \left( \frac{\bar \tau_{3(j-1)+1}-T_j}{\bar \tau_{3(j-1)+1}-t}
    \right)^{\frac{3\beta}{\beta+1}}
  \end{equation*}
  for $p=1,2$ and
  \begin{multline}
    \label{eq:active-small-approx}
    \begin{pmatrix}
      x_{3(j-1)+1} \\ x_{3(j-1)+2}
    \end{pmatrix}
    (t)
    =
    (\beta+1)^{\frac{1}{\beta+1}}
    \left(\bar \tau_{3(j-1)+1} - t \right)^{\frac{1}{\beta+1}}
    \begin{pmatrix}
      1 \\ 1
    \end{pmatrix}
    \\
    +
    C_4 \frac{D_3(T_j)}{\bar x_{3j}(T_j)^{\beta+1}}
    \left( \bar \tau_{3(j-1)+1} - T_j \right)
    \left( \frac{\bar \tau_{3(j-1)+1}-T_j}{\bar \tau_{3(j-1)+1}-t}
    \right)^{\frac{3\beta}{\beta+1}}
    \begin{pmatrix}
      1 \\ -1
    \end{pmatrix}
    \\
    + O \left(
      (\bar \tau_{3(j-1)+1} - t)
      +
      R_{j-1}^{\beta+1} \delta
      \left( \frac{\bar \tau_{3(j-1)+1}-T_j}{\bar \tau_{3(j-1)+1}-t}
      \right)^{\frac{\beta}{\beta+1}}
      +
      |\rho_1| + |\rho_2|
    \right)
  \end{multline}
  for all $t \in [T_j,\tau_{j-1}^*]$.
\end{corollary}

\begin{proof}
  Since $D_3(T_j)=\delta$ the first claim is an immediate consequence
  of Proposition \ref{pro:active-small-approx}, while the second in
  addition uses
  \begin{equation*}
    \bar x_{3(j-1)+p}(t)
    =
    ( \beta+1 )^{\frac{1}{\beta+1}}
    \left( \bar \tau_{3(j-1)+1} -t \right)^{\frac{1}{\beta+1}}
    +
    O \left( \bar \tau_{3(j-1)+1} - t \right),
  \end{equation*}
  as proved in Lemma \ref{lem:asymptotics}.
\end{proof}

\begin{corollary}
  \label{cor:diff-van-time-estimate}
  Assume that $\beta \geq \beta_*$. In the setting of Proposition
  \ref{pro:active-small-approx} we have either
  \begin{align*}
    \tau_{3(j-1)+1} - \bar \tau_{3(j-1)+1}
    &=
    (S_*+1) A_*
    \left( R_{j-1}^{4\beta+1} \delta \right)^{\frac{\beta+1}{3\beta+1}}
    \big( 1 + o(1)_{R_{j-1} \to 0} \big),
    \\
    \tau_{3(j-1)+2} - \bar \tau_{3(j-1)+1}
    &=
    (S_*-1) A_*
    \left( R_{j-1}^{4\beta+1} \delta \right)^{\frac{\beta+1}{3\beta+1}}
    \big( 1 + o(1)_{R_{j-1} \to 0} \big)
  \end{align*}
  if $D_3(T_j)=+\delta$ or the same equations with $\tau_{3(j-1)+1}$
  and $\tau_{3(j-1)+2}$ exchanged if $D_3(T_j)=-\delta$.
  Here
  \begin{equation*}
    A_*^{\frac{3\beta+1}{\beta+1}}
    =
    \frac{C_4}{B_* (\beta+1)^{\frac{4\beta+1}{\beta+1}}}
    \bar x_{3j}(T_j)^{-(\beta+1)}
    >
    0
  \end{equation*}
  where $S_*$ and $B_*$ are the constants from Lemma
  \ref{lem:matching}. In particular, since $|S_*|<1$ we have
  \begin{align*}
    \tau_{3(j-1)+2} < \bar \tau_{3(j-1)+1} < \tau_{3(j-1)+1}
    &\qquad\text{if}\qquad D_3(T_j)=+\delta,
    \\
    \tau_{3(j-1)+1} < \bar \tau_{3(j-1)+1} < \tau_{3(j-1)+2}
    &\qquad\text{if}\qquad D_3(T_j)=-\delta,    
  \end{align*}
  provided that $N_*$ is sufficiently large.
\end{corollary}

\begin{proof}
  In the following we abbreviate
  \begin{equation*}
    \bar \tau = \bar \tau_{3(j-1)+p},
    \qquad
    T = T_j
    \qquad\text{and}\qquad
    R = R_{j-1},
  \end{equation*}
  so that \eqref{eq:active-small-approx} reads
  \begin{multline*}
    \begin{pmatrix}
      x_{3(j-1)+1} \\ x_{3(j-1)+2}
    \end{pmatrix}
    (t)
    =
    (\beta+1)^{\frac{1}{\beta+1}}
    \left(\bar\tau-t\right)^{\frac{1}{\beta+1}}
    \begin{pmatrix}
      1 \\ 1
    \end{pmatrix}
    + C_4 \frac{D_3(T)}{\bar x_{3j}(T)^{\beta+1}}
    \frac{(\bar\tau-T)^{\frac{4\beta+1}{\beta+1}}}
    {(\bar\tau-t)^{\frac{3\beta}{\beta+1}}}
    \begin{pmatrix}
      1 \\ -1
    \end{pmatrix}
    \\
    +
    O \left( (\bar\tau-t) +
      R^{\beta+1} \delta \left( \frac{\bar\tau-T}{\bar\tau-t}
      \right)^{\frac{\beta}{\beta+1}}
      +
      |\rho_1| + |\rho_2|
    \right).
  \end{multline*}
  Assuming that $D_3(T)=+\delta$ we let
  \begin{equation*}
    y_p(s)
    =
    \frac{x_{3(j-1)+p} \left( \bar \tau + A_* (R^{4\beta+1}
        \delta)^{\frac{\beta+1}{3\beta+1}} s \right) }
    {A_*^{\frac{1}{\beta+1}} ( R^{4\beta+1} \delta )^{\frac{1}{3\beta+1}}}
  \end{equation*}
  and obtain 
  \begin{align*}
    \begin{pmatrix}
      y_1 \\ y_2
    \end{pmatrix}
    (s)
    &=
    (\beta+1)^{\frac{1}{\beta+1}} (-s)^{\frac{1}{\beta+1}}
    \begin{pmatrix}
      1 \\ 1
    \end{pmatrix}
    +
    C_4
    \frac{(\bar \tau-T)^{\frac{4\beta+1}{\beta+1}}}
    {R^{4\beta+1}}
    \bar x_{3j}(T)^{-(\beta+1)}
    A_*^{-\frac{3\beta+1}{\beta+1}} (-s)^{-\frac{3\beta}{\beta+1}}
    \begin{pmatrix}
      1 \\ -1
    \end{pmatrix}
    \\
    &\qquad+
    \omega_1(R,\delta) (-s)
    +
    \omega_2(R,\delta) (-s)^{-\frac{\beta}{\beta+1}}
    +
    \omega_3(R,\delta) (-s)^{-\frac{3\beta}{\beta+1}}
  \end{align*}
  where $\omega_p(R,\delta)$, $p=1,2,3$ tends to zero as $R \to 0$ or
  $\delta \to 0$.
  Furthermore, with the asymptotics
  \begin{equation}
    \label{eq:cor-diff-van-time:asymp-time}
    \bar \tau-T = \frac{R^{\beta+1}}{\beta+1} (1+o(1)_{R\to0}),
  \end{equation}
  which follow from Lemma \ref{lem:asymptotics}, and the definition
  of $A_*$ we arrive at
  \begin{multline*}
    \begin{pmatrix}
      y_1 \\ y_2
    \end{pmatrix}
    (s)
    =
    (\beta+1)^{\frac{1}{\beta+1}} (-s)^{\frac{1}{\beta+1}}
    \begin{pmatrix}
      1 \\ 1
    \end{pmatrix}
    +
    B_*
    (-s)^{\frac{3\beta}{\beta+1}}
    \begin{pmatrix}
      1 \\ -1
    \end{pmatrix}
    \\
    +
    \omega_1(R,\delta) (-s)
    +
    \omega_2(R,\delta) (-s)^{-\frac{\beta}{\beta+1}}
    +
    \omega_3(R,\delta) (-s)^{-\frac{3\beta}{\beta+1}}.
  \end{multline*}

  Our goal is to apply Lemma \ref{lem:matching} with $y_p$, so we now
  check its requirements.
  First, \eqref{eq:active-small-approx} and the above considerations
  are valid for $\bar \tau + A_* (R^{4\beta+1}
  \delta)^{(\beta+1)/(3\beta+1)}s \in [T,\tau_{j-1}^*]$, which by
  \eqref{eq:cor-diff-van-time:asymp-time} and the definition of
  $\tau_{j-1}^*$ means
  \begin{equation*}
    - \frac{1}{A_*} (R^\beta \delta)^{-\frac{\beta+1}{3\beta+1}}
    \lesssim
    s
    \leq - \frac{C}{A_*}.
  \end{equation*}
  Moreover, we have
  \begin{equation*}
    \tderiv{s} y_p
    =
    -2 y_p^{-\beta} + y_{3-p}^{-\beta} + F_p
  \end{equation*}
  where
  \begin{equation*}
    F_p
    =
    A_*^{\frac{\beta}{\beta+1}}
    \left( R^{4\beta+1}\delta \right)^{\frac{3\beta}{\beta+1}}
    x_{3(j+p-2)}.
  \end{equation*}
  By taking $N_*$ sufficiently large we can make $F_p$ and $\omega_p$
  small and arrange $s \in [-S,-C/A_*]$ for some $S$ that is large but
  of order $1$. Then, the claim follows from Lemma \ref{lem:matching},
  and the case $(-\delta)$ works by interchanging $y_1$ and $y_2$.
\end{proof}

Corollaries \ref{cor:active-small-leading-order} and
\ref{cor:diff-van-time-estimate} allow us to estimate the integrals in
Lemma \ref{lem:active-large-I}.

\begin{lemma}[Estimates for active large particles II]
  \label{lem:active-large-III}
  Let $\beta \geq \beta_*$. In the setting of Proposition
  \ref{pro:active-small-approx}, we have
  \begin{align*}
    \int_{T_j}^t
    \left| x_{\sigma_+(3(j-1))}^{-\beta} -
      \bar x_{\sigma_+(3(j-1))}^{-\beta} \right| \,\dint{s}
    &\leq
    C \left( R^{4\beta+1} \delta \right)^{\frac{1}{3\beta+1}}
    \\\intertext{and}
    \int_{T_j}^t
    \left| x_{\sigma_-(3j)}^{-\beta} -
      \bar x_{\sigma_-(3j)}^{-\beta} \right| \,\dint{s}
    &\leq
    C \left( R^{4\beta+1} \delta \right)^{\frac{1}{3\beta+1}}
  \end{align*}
  for all $t \in [T_j,T_{j-1}]$.
\end{lemma}

\begin{proof}
  Since the proofs of both inequalities are identical, we consider only
  the first. For simplicity of notation we abbreviate
  \begin{equation*}
    \bar \tau = \bar \tau_{3(j-1)+1},
    \qquad
    \tau^* = \tau^*_{j-1}
    \qquad\text{and}\qquad
    R = R_{j-1}.
  \end{equation*}

  For $t \in [T_j,\tau^*]$ we have $\sigma_+(3(j-1)) = 3(j-1)+1$,
  hence Corollary \ref{cor:active-small-leading-order} and Lemma
  \ref{lem:asymptotics} yield
  \begin{equation*}
    \left| x_{\sigma_+(3(j-1))}^{-\beta} -
      \bar x_{\sigma_+(3(j-1))}^{-\beta} \right|
    \leq
    \frac{C |D_1|}{\bar x_{3(j-1)+1}^{\beta+1}}
    \leq
    \frac{C R^{\beta+1} \delta}{\bar \tau - t}
    \left( \frac{\bar \tau - T_j}{\bar \tau - t}
    \right)^{\frac{3\beta}{\beta+1}}.
  \end{equation*}
  By integration and using Corollary \ref{cor:diff-van-time-estimate}
  we find
  \begin{align*}
    \int_{T_j}^t
    \left| x_{\sigma_+(3(j-1))}^{-\beta} -
      \bar x_{\sigma_+(3(j-1))}^{-\beta} \right| \,\dint{s}
    &\leq
    C R^{\beta+1} \delta
    \left(\bar \tau - T_j\right)^{\frac{3\beta}{\beta+1}}
    \left( \bar \tau - \tau^* \right)^{-\frac{3\beta}{\beta+1}}
    \\
    &\leq
    C \left( R^{4\beta+1} \delta \right)^{\frac{1}{3\beta+1}}.
  \end{align*}

  For $t \in [\tau^*, T_{j-1}]$, however, we have to deal with the
  fact that $\sigma_+(3(j-1))$ changes. Here, the lower bound
  \eqref{eq:exist:lower-bound} implies
  \begin{equation*}
    \int_{\tau^*}^t
    x_{3(j-1)+p}^{-\beta} \,\dint{s}
    \leq
    C \int_{\tau^*}^t
    \left( \tau_{3(j-1)+p} - t \right)
    \chi_{\set{t<\tau_{3(j-1)+p}}}
    \,\dint{s}
    \leq
    C \left( \tau_{3(j-1)+p} - \tau^* \right)^\frac{1}{\beta+1}
  \end{equation*}
  for $p=1,2$ and similar inequalities for $\bar x_{3(j-1)+p}$.
  Since the estimate for the possibly large neighbor
  $\sigma_+(3(j-1)) = 3j$ is even better, the definition of $\tau^*$
  and Corollary \ref{cor:diff-van-time-estimate} yield
  \begin{align*}
    \int_{\tau^*}^t
    \left| x_{\sigma_+(3(j-1))}^{-\beta} -
      \bar x_{\sigma_+(3(j-1))}^{-\beta} \right| \,\dint{s}
    \leq
    C \left( R^{4\beta+1} \delta \right)^{\frac{1}{3\beta+1}}
  \end{align*}
  for all $t \in [\tau^*,T_{j-1}]$.
\end{proof}

We now combine the estimates from Lemmas \ref{lem:inactive-I} and
\ref{lem:active-large-I} with the consequences of Proposition
\ref{pro:active-small-approx} in order to complete the estimates for
all non-vanishing particles.

\begin{corollary}[Estimates for non-vanishing particles]
  \label{cor:remaining-particles}
  Let $\beta \geq \beta_*$.  In the setting of Proposition
  \ref{pro:active-small-approx} we have
  \begin{align*}
    |D_0(t)|
    &\leq
    C  \left( R^{4\beta+1} \delta \right)^{\frac{1}{3\beta+1}}
    +
    C \delta \delta_*,
    \\
    |D_3(t)|
    &\leq
    C \left( R^{4\beta+1} \delta \right)^{\frac{1}{3\beta+1}}
    +
    C \delta,
    \\
    \| x - \bar x \|_s(t)
    &\leq
    C \gamma^{\beta+1}
    \left( R^{4\beta+1} \delta \right)^{\frac{1}{3\beta+1}}
    +
    C \delta \delta_*,
    \\
    \| x - \bar x \|_l(t)
    &\leq
    C R_{j-1}^{\beta+1} \gamma^{\beta+1}
    \left( R^{4\beta+1} \delta \right)^{\frac{1}{3\beta+1}}
    +
    C \delta \delta_*,
    \\
    \| x - \bar x \|_r(t)
    &\leq
    C R_{j-1}^{\beta+1} 
    \left( R^{4\beta+1} \delta \right)^{\frac{1}{3\beta+1}}
    +
    C \delta
  \end{align*}
  for all $t \in [T_j,T_{j-1}]$.
\end{corollary}

\begin{proof}
  Using Lemma \ref{lem:active-large-III} and the assumptions of
  Proposition \ref{pro:active-small-approx} in the estimates from
  Lemma \ref{lem:active-large-I} we find
  \begin{align*}
    |D_0(t)|
    &\leq
    C \delta \delta_*
    +
    C (t-T_j) \left( \delta \delta_* + \gamma^{\beta+1} \delta \delta_* \right)
    +
    C \left( R^{4\beta+1} \delta \right)^{\frac{1}{3\beta+1}},
    \\
    |D_3(t)|
    &\leq
    C \delta + C (t-T_j) \delta
    +
    C \left( R^{4\beta+1} \delta \right)^{\frac{1}{3\beta+1}}
  \end{align*}
  for all $t \in [T_{j-1},T_j]$. Then, Lemma \ref{lem:inactive-I}
  yields
  \begin{align*}
    \| x - \bar x \|_s(t)
    &\leq
    C \delta \delta_* + C \frac{t-T_j}{R_{j-2}^{\beta+1}} \delta \delta_*
    + \frac{C}{R_{j-2}^{\beta+1}} \int_{T_j}^t |D_0(s)| \,\dint{s}
    \\
    &\leq
    C \delta \delta_* + C \frac{t-T_j}{R_{j-2}^{\beta+1}} \delta \delta_*
    +
    % C \gamma^{\beta+1} \frac{t-T_j}{R_{j-2}^{\beta+1}} \delta \delta_*
    % +
    C \frac{t-T_j}{R_{j-2}^{\beta+1}}
    \left( R^{4\beta+1} \delta \right)^{\frac{1}{3\beta+1}},
    \\
    \| x - \bar x \|_l(t)
    &\leq
    C \delta \delta_* + C (t-T_j) \delta \delta_*
    +
    C \gamma^{\beta+1} (t-T_j)
    \left( R^{4\beta+1} \delta \right)^{\frac{1}{3\beta+1}}
    \\
    \intertext{and}
    \| x - \bar x \|_r(t)
    % &\leq
    % C \delta + C \int_{T_j}^t |D_3(s)| \,\dint{s}
    % \\
    &\leq
    C \delta
    +
    C (t-T_j)
    \left( R^{4\beta+1} \delta \right)^{\frac{1}{3\beta+1}}
  \end{align*}
  for all $t \in [T_{j-1},T_j]$.
\end{proof}

Corollary \ref{cor:remaining-particles} bounds the change of particle
differences from above. However, we still need to study how exactly
the perturbation is transported from the particles $3j$ to $3(j-1)$ in
the time interval $[T_j,T_{j-1}]$. To this end, Lemma \ref{lem:step-x}
describes the transfer of mass from $x_{3j}(T_j)$ to
$x_{3(j-1)}(T_{j-1})$ and $\bar x_{3j}(T_j)$ to $\bar
x_{3(j-1)}(T_{j-1})$, respectively, while Corollary
\ref{cor:mass-transfer} contains the corresponding result for $D_0$
and $D_3$.
Of particular interest is the value $r_j$ below, which is the
difference between $x$ and $\bar x$ and which is responsible for
the amplification of the perturbation.

\begin{lemma}
  \label{lem:step-x}
  Denote by
  \begin{equation}
    \label{eq:amplifier}
    r_j =
    \begin{dcases}
      +\tfrac{1}{6} x_{3(j-1)+2}( \tau_{3(j-1)+1} )
      &\text{if } \tau_{3(j-1)+1} < \tau_{3(j-1)+2},
      \\
      -\tfrac{1}{6} x_{3(j-1)+1}( \tau_{3(j-1)+2} )
      &\text{if } \tau_{3(j-1)+1} > \tau_{3(j-1)+2},
      \\
      0
      &\text{if } \tau_{3(j-1)+1} = \tau_{3(j-1)+2}
    \end{dcases}
  \end{equation}
  the scaled size of the remaining particle when the first one of
  $x_{3(j-1)+p}$, $p=1,2$ vanishes.
  Then we have
  \begin{align*}
    %\label{eq:step-x-j-1}
    x_{3(j-1)}(T_{j-1}) - x_{3(j-1)}(T_j)
    &=
    \tfrac{2}{3} x_{3(j-1)+1}(T_j) + \tfrac{1}{3} x_{3(j-1)+2}(T_j)
    + r_j
    \\
    &+
    \int_{T_j}^{\min(\tau_{3(j-1)+p})}
    x_{3(j-2)+2}^{-\beta} - \tfrac{4}{3} x_{3(j-1)}^{-\beta}
    + \tfrac{1}{3} x_{3j}^{-\beta} \,\dint{t}
    \\
    &+
    \int_{\min(\tau_{3(j-1)+p})}^{\max(\tau_{3(j-1)+p})}
    x_{3(j-2)+2}^{-\beta} - \tfrac{3}{2} x_{3(j-1)}^{-\beta}
    + \tfrac{1}{2} x_{3j}^{-\beta} \,\dint{t}
    \\
    &+
    \int_{\max(\tau_{3(j-1)+p})}^{T_{j-1}}
    x_{3(j-2)+2}^{-\beta} - 2 x_{3(j-1)}^{-\beta}
    + x_{3j}^{-\beta} \,\dint{t}
  \end{align*}
  and
  \begin{align*}
    %\label{eq:step-x-j}
    x_{3j}(T_{j-1}) - x_{3j}(T_j)
    &=
    \tfrac{1}{3} x_{3(j-1)+1}(T_j) + \tfrac{2}{3} x_{3(j-1)+2}(T_j)
    - r_j
    \\
    &+
    \int_{T_j}^{\min(\tau_{3(j-1)+p})}
    \tfrac{1}{3} x_{3(j-1)}^{-\beta} - \tfrac{4}{3} x_{3j}^{-\beta}
    + x_{3(j+1)}^{-\beta} \,\dint{t}
    \\
    &+
    \int_{\min(\tau_{3(j-1)+p})}^{\max(\tau_{3(j-1)+p})}
    \tfrac{1}{2} x_{3(j-1)}^{-\beta} - \tfrac{3}{2} x_{3j}^{-\beta}
    + x_{3(j+1)}^{-\beta} \,\dint{t}
    \\
    &+
    \int_{\max(\tau_{3(j-1)+p})}^{T_{j-1}}
    x_{3(j-1)}^{-\beta} - 2 x_{3j}^{-\beta}
    + x_{3(j+1)}^{-\beta} \,\dint{t},
  \end{align*}
  where the maximum and minimum in the integral boundaries are taken
  over $p=1,2$ and integrals with identical lower and upper boundary
  are $0$.
\end{lemma}

\begin{proof}
  The assertions follow from direct computations considering the two
  cases $\tau_{3(j-1)+1} \leq \tau_{3(j-1)+2}$ and $\tau_{3(j-1)+1} >
  \tau_{3(j-1)+2}$ separately. Since the arguments in both cases are
  the same we only go through the first, where $T_j < \tau_{3(j-1)+1}
  \leq \tau_{3(j-1)+2} \leq T_{j-1}$.

  For $t \in (T_j,\tau_{3(j-1)+1})$ we have $x_k(t)>0$ for $k \leq 3j$
  and $k=3(j+1)$, while $x_{3j+1}(t) = x_{3j+2}(t) = 0$. Hence, the
  equations for $x_k$, $k = 3(j-1)$, $3(j-1)+1$, $3(j-1)+2$, $3j$ read
  \begin{align}
    \label{eq:x-0-a}
    \dot x_{3(j-1)}
    &=
    x_{3(j-2)+2}^{-\beta} - 2 x_{3(j-1)}^{-\beta} + x_{3(j-1)+1}^{-\beta},
    \\
    \label{eq:x-1-a}
    \dot x_{3(j-1)+1}
    &=
    x_{3(j-1)}^{-\beta} - 2 x_{3(j-1)+1}^{-\beta} + x_{3(j-1)+2}^{-\beta},
    \\
    \label{eq:x-2-a}
    \dot x_{3(j-1)+2}
    &=
    x_{3(j-1)+1}^{-\beta} - 2 x_{3(j-1)+2}^{-\beta} + x_{3j}^{-\beta},
    \\
    \label{eq:x-3-a}
    \dot x_{3j}
    &=
    x_{3(j-1)+2}^{-\beta} - 2 x_{3j}^{-\beta} + x_{3(j+1)}^{-\beta}.
  \end{align}
  Adding twice \eqref{eq:x-1-a} to \eqref{eq:x-2-a} gives
  \begin{equation}
    \label{eq:x-add}
    2 \dot x_{3(j-1)+1} + \dot x_{3(j-1)+2}
    =
    2 x_{3(j-1)}^{-\beta} - 3 x_{3(j-1)+1}^{-\beta} + x_{3j}^{-\beta},
  \end{equation}
  and using \eqref{eq:x-add} to eliminate $x_{3(j-1)+1}^{-\beta}$ in
  \eqref{eq:x-0-a} we obtain
  \begin{equation*}
    \dot x_{3(j-1)}
    =
    - \tfrac{2}{3} \dot x_{3(j-1)+1} - \tfrac{1}{3} \dot x_{3(j-1)+2}
    + x_{3(j-2)+2}^{-\beta} - \tfrac{4}{3} x_{3(j-1)}^{-\beta}
    + \tfrac{1}{3} x_{3j}^{-\beta}.
  \end{equation*}
  Integration from $T_j$ to $\tau_{3(j-1)+1}$ and the vanishing of
  $x_{3(j-1)+1}(\tau_{3(j-1)+1})$ then yield
  \begin{multline}
    \label{eq:x-0-one}
    x_{3(j-1)}(\tau_{3(j-1)+1}) - x_{3(j-1)}(T_j)
    =
    \tfrac{2}{3} x_{3(j-1)+1}(T_j) + \tfrac{1}{3} x_{3(j-1)+2}(T_j)
    \\
    - \tfrac{1}{3} x_{3(j-1)+2}(\tau_{3(j-1)+1})
    + \int_{T_j}^{\tau_{3(j-1)+1}}
    x_{3(j-2)+2}^{-\beta} - \tfrac{4}{3} x_{3(j-1)}^{-\beta}
    + \tfrac{1}{3} x_{3j}^{-\beta} \,\dint{t}.
  \end{multline}
  Similarly, adding twice \eqref{eq:x-2-a} to \eqref{eq:x-1-a}, using the
  result to eliminate $x_{3(j-1)+2}^{-\beta}$ in \eqref{eq:x-3-a}, and
  integrating, we find
  \begin{multline}
    \label{eq:x-3-one}
    x_{3j}(\tau_{3(j-1)+1}) - x_{3j}(T_j)
    =
    \tfrac{1}{3} x_{3(j-1)+1}(T_j) + \tfrac{2}{3} x_{3(j-1)+2}(T_j)
    \\
    - \tfrac{2}{3} x_{3(j-1)+2}(\tau_{3(j-1)+1})
    + \int_{T_j}^{\tau_{3(j-1)+1}}
    \tfrac{1}{3} x_{3(j-1)}^{-\beta} - \tfrac{4}{3} x_{3j}^{-\beta}
    + x_{3(j+1)}^{-\beta} \,\dint{t}.
  \end{multline}

  Second, for $\tau_{3(j-1)+1} < \tau_{3(j-1)+2}$ and $t \in
  (\tau_{3(j-1)+1}, \tau_{3(j-1)+2})$ the non-vanished particles
  $k=3(j-1)$, $3(j-1)+2$, $3j$ satisfy
  \begin{align}
    \label{eq:x-0-b}
    \dot x_{3(j-1)}
    &=
    x_{3(j-2)+2}^{-\beta} - 2 x_{3(j-1)}^{-\beta} + x_{3(j-1)+2}^{-\beta},
    \\
    \label{eq:x-2-b}
    \dot x_{3(j-1)+2}
    &=
    x_{3(j-1)}^{-\beta} - 2 x_{3(j-1)+2}^{-\beta} + x_{3j}^{-\beta},
  \end{align}
  and \eqref{eq:x-3-a}. Again, using \eqref{eq:x-2-b} to
  eliminate $x_{3(j-1)+2}^{-\beta}$ in \eqref{eq:x-0-b} and
  \eqref{eq:x-3-a} and integrating, we obtain
  \begin{multline}
    \label{eq:x-0-two}
    x_{3(j-1)}(\tau_{3(j-1)+2}) - x_{3(j-1)}(\tau_{3(j-1)+1})
    =
    \\
    \tfrac{1}{2} x_{3(j-1)+2}(\tau_{3(j-1)+1})
    +
    \int_{\tau_{3(j-1)+1}}^{\tau_{3(j-1)+2}}
    x_{3(j-2)+2}^{-\beta} - \tfrac{3}{2} x_{3(j-1)}^{-\beta}
    + \tfrac{1}{2} x_{3j}^{-\beta} \,\dint{t}
  \end{multline}
  and
  \begin{multline}
    \label{eq:x-3-two}
    x_{3j}(\tau_{3(j-1)+2}) - x_{3j}(\tau_{3(j-1)+1})
    =
    \\
    \tfrac{1}{2} x_{3(j-1)+2}(\tau_{3(j-1)+1})
    +
    \int_{\tau_{3(j-1)+1}}^{\tau_{3(j-1)+2}}
    \tfrac{1}{2} x_{3(j-1)}^{-\beta} - \tfrac{3}{2} x_{3j}^{-\beta}
    + x_{3(j+1)}^{-\beta} \,\dint{t}.
  \end{multline}
  Finally, integrating the equations for $x_{3(j-1)}$ and $x_{3j}$
  over $t \in (\tau_{3(j-1)+2},T_{j-1})$ yields
  \begin{align}
    \label{eq:x-0-three}
    x_{3(j-1)}(T_{j-1}) - x_{3(j-1)}(\tau_{3(j-1)+2})
    &=
    \int_{\tau_{3(j-1)+2}}^{T_{j-1}}
    x_{3(j-2)+2}^{-\beta} - 2 x_{3(j-1)}^{-\beta}
    + x_{3j}^{-\beta} \,\dint{t},
    \\
    x_{3j}(T_{j-1}) - x_{3j}(\tau_{3(j-1)+2})
    &=
    \label{eq:x-3-three}
    \int_{\tau_{3(j-1)+2}}^{T_{j-1}}
    x_{3(j-1)}^{-\beta} - 2 x_{3j}^{-\beta}
    + x_{3(j+1)}^{-\beta} \,\dint{t},
  \end{align}
  and adding \eqref{eq:x-0-one}, \eqref{eq:x-0-two},
  \eqref{eq:x-0-three} as well as \eqref{eq:x-3-one},
  \eqref{eq:x-3-two}, \eqref{eq:x-3-three} proves the claim.
\end{proof}

Lemma \ref{lem:step-x} holds as well for $\bar x$ with $r_j=0$
identically. Taking the difference of $x$ and $\bar x$ we obtain
the following result.

\begin{corollary}
  \label{cor:mass-transfer}
  Suppose that either $\tau_{3(j-1)+1} < \bar \tau_{3(j-1)+1} <
  \tau_{3(j-1)+2}$ or $\tau_{3(j-1)+2} < \bar \tau_{3(j-1)+1} <
  \tau_{3(j-1)+1}$. Then we have
  \begin{equation*}
    D_0(T_{j-1})
    =
    D_0(T_{j})
    +
    \tfrac{2}{3} D_1(T_j)
    +
    \tfrac{1}{3} D_2(T_j)
    +
    r_j
    +
    \calI_j
  \end{equation*}
  and
  \begin{equation*}
    D_3(T_{j-1})
    =
    D_3(T_j)
    +
    \tfrac{1}{3} D_1(T_j)
    +
    \tfrac{2}{3} D_2(T_j)
    -
    r_j
    +
    \calJ_j,
  \end{equation*}
  where $r_j$ is as in \eqref{eq:amplifier} and 
  \begin{align*}
    \calI_j
    &=
    \calI_j^1 + \calI_j^2 + \calI_j^3
    \\
    &=
    \int_{T_j}^{\min(\tau_{3(j-1)+p})}
    \left( x_{3(j-2)+2}^{-\beta} - \bar x_{3(j-2)+2}^{-\beta} \right)
    -
    \tfrac{4}{3}
    D_0^\beta
    %\left( x_{3(j-1)}^{-\beta} - \bar x_{3(j-1)}^{-\beta} \right)
    +
    \tfrac{1}{3}
    D_3^\beta
    %\left( x_{3j}^{-\beta} - \bar x_{3j}^{-\beta} \right)
    \,\dint{t}
    \\
    &\quad+
    \int_{\min(\tau_{3(j-1)+p})}^{\bar \tau_{3(j-1)+1}}
    x_{3(j-2)+2}^{-\beta} - \bar x_{3(j-2)+2}^{-\beta} 
    - \tfrac{3}{2} x_{3(j-1)}^{-\beta}
    + \tfrac{4}{3} \bar x_{3(j-1)}^{-\beta}
    + \tfrac{1}{2} x_{3j}^{-\beta}
    - \tfrac{1}{3} \bar x_{3j}^{-\beta}
    \,\dint{t}
    \\
    &\quad+
    \int_{\bar \tau_{3(j-1)+1}}^{T_{j-1}}
    x_{3(j-2)+2}^{-\beta} - \bar x_{3(j-2)+2}^{-\beta}
    - \tfrac{3}{2} x_{3(j-1)}^{-\beta}
    + 2 \bar x_{3(j-1)}^{-\beta}
    + \tfrac{1}{2} x_{3j}^{-\beta}
    - \bar x_{3j}^{-\beta}
    \,\dint{t}
  \end{align*}
  as well as
  \begin{align*}
    \calJ_j
    &=
    \calJ_j^1 + \calJ_j^2 + \calJ_j^3
    \\
    &=
    \int_{T_j}^{\min(\tau_{3(j-1)+p})}
    \tfrac{1}{3}
    D_0^\beta
    %\left( x_{3(j-1)}^{-\beta} - \bar x_{3(j-1)}^{-\beta} \right)
    -
    \tfrac{4}{3}
    D_3^\beta
    %\left( x_{3j}^{-\beta} - \bar x_{3j}^{-\beta} \right)
    +
    \left( x_{3(j+1)}^{-\beta} - \bar x_{3(j+1)}^{-\beta} \right)
    \,\dint{t}
    \\
    &\quad+
    \int_{\min(\tau_{3(j-1)+p})}^{\bar \tau_{3(j-1)+1}}
    \tfrac{1}{2} x_{3(j-1)}^{-\beta}
    - \tfrac{1}{3} \bar x_{3(j-1)}^{-\beta}
    - \tfrac{3}{2} x_{3j}^{-\beta}
    + \tfrac{4}{3} \bar x_{3j}^{-\beta} 
    + x_{3(j+1)}^{-\beta} - \bar x_{3(j+1)}^{-\beta}
    \,\dint{t}
    \\
    &\quad+
    \int_{\bar \tau_{3(j-1)+1}}^{T_{j-1}}      
    \tfrac{1}{2} x_{3(j-1)}^{-\beta}
    - \bar x_{3(j-1)}^{-\beta}
    - \tfrac{3}{2} x_{3j}^{-\beta}
    + 2 \bar x_{3j}^{-\beta}
    + x_{3(j+1)}^{-\beta} - \bar x_{3(j+1)}^{-\beta}
    \,\dint{t}.
  \end{align*}
\end{corollary}

Using our previous estimates for the various terms in Corollary
\ref{cor:mass-transfer} we can finally characterize $D_0(T_{j-1})$ and
$D_3(T_{j-1})$.

\begin{proposition}
  \label{pro:active-large}
  Suppose that $\beta \geq \beta_*$. There is a constant $a_*>0$ such
  that in the setting of Proposition \ref{pro:active-small-approx} we
  have
  \begin{equation*}
    |r_j|
    =
    a_* \left( R_{j-1}^{4\beta+1} \delta \right)^{\frac{1}{3\beta+1}}
    \left(1+ o(1)_{R_{j-1} \to 0} \right).
  \end{equation*}
  Moreover, we have
  \begin{align*}
    | D_0(T_{j-1}) - r_j |
    &\leq
    \left( R_{j-1}^{4\beta+1} \delta \right)^{\frac{1}{3\beta+1}}
    o(1)_{\gamma\to0,R_{j-1}\to0}
    +
    O \left( R_{j-1}^{\beta+1} \delta (1+\delta_*) \right)
    +
    \delta \delta_* (1+R_{j-1}^{\beta+1})
    \\
    \intertext{and}
    | D_3(T_{j-1}) + r_j |
    &\leq
    \left( R_{j-1}^{4\beta+1} \delta \right)^{\frac{1}{3\beta+1}}
    o(1)_{R_{j-1}\to0}
    +
    O \left( R_{j-1}^{\beta+1} \delta (1+\delta_*) \right)
    +
    \delta + R_{j-1}^{\beta+1} \delta \delta_*.
  \end{align*}
\end{proposition}

\begin{proof}
  First, to determine the size of $r_j$ note that after one of the
  small particles $x_{3(j-1)+p}$, $p=1,2$ has vanished, the remaining
  one has large neighbors only. Therefore, the power law from Lemma
  \ref{lem:asymptotics} (formally doubling the small particle) and
  Corollary \ref{cor:diff-van-time-estimate} imply
  \begin{align*}
    |r_j|
    &=
    \frac{1}{6} (\beta+1)^{\frac{1}{\beta+1}}
    \left| \tau_{3(j-1)+2}-\tau_{3(j-1)+1} \right|^{\frac{1}{\beta+1}}
    \left( (1+ O\big(R_{j-1}^{\beta+1}\big) \right)
    \\
    &=
    \frac{1}{3} (\beta+1)^{\frac{1}{\beta+1}}
    A_* \left( R_{j-1}^{4\beta+1} \delta \right)^{\frac{1}{3\beta+1}}
    \left( 1+ o(1)_{R_{j-1} \to 0} \right).
  \end{align*}
  As $A_*$ depends only on $\beta$ and $\bar x_{3j}(T_j)$, and
  as the latter satisfies $\bar x_{3j}(T_j) = 1 + O(R_{j-1})$ by the
  Constancy Lemma \ref{lem:constancy} we obtain
  \begin{equation*}
    |r_j|
    =
    a_*
    \left( R_{j-1}^{4\beta+1} \delta \right)^{\frac{1}{3\beta+1}}
    \left( 1+ o(1)_{R_{j-1} \to 0} \right)
  \end{equation*}
  for a constant $a_*$.

  Next, we estimate the other contributions to $D_0$ and $D_3$ from
  Corollary \ref{cor:mass-transfer} separately, and since the proofs
  for both differences are similar, we consider only $D_0$.
  First, by assumptions \eqref{eq:iterative-ass-active-other} we
  clearly have
  \begin{equation*}
    \left|
      D_0(T_j) + \tfrac{2}{3} D_1(T_j) + \tfrac{1}{3} D_2(T_j)
    \right|
    \leq
    \delta \delta_* + R_{j-1}^{\beta+1} \delta \delta_*.
  \end{equation*}
  Furthermore, linearizing in $\calI_j^1$ we have
  \begin{equation*}
    |\calI_j^1|
    \leq
    C \int_{T_j}^{\min(\tau_{3(j-1)+p})}
    \frac{1}{R_{j-1}^{\beta+1}} \| x - \bar x\|_s
    +
    |D_0| + |D_3|
    \,\dint{t}
  \end{equation*}
  and using Corollary \ref{cor:remaining-particles} we find
  \begin{align*}
    | \calI_j^1 |
    &\leq
    C \int_{T_j}^{\min(\tau_{3(j-1)+p})}
    \left( R_{j-1}^{4\beta+1} \delta \right)^{\frac{1}{3\beta+1}}
    \left( 1 + \tfrac{\gamma^{\beta+1}}{R_{j-2}^{\beta+1}} \right)
    +
    C \delta (1+\delta_*)
    \,\dint{t}
    \\
    &\leq
    C
    \left( R_{j-1}^{4\beta+1} \delta \right)^{\frac{1}{3\beta+1}}
    \left( \gamma^{2(\beta+1)} + R_{j-1}^{\beta+1} \right)
    +
    C R_{j-1}^{\beta+1} \delta ( 1 + \delta_*).
  \end{align*}
  Finally, $\calI_j^2$ and $\calI_j^3$ are easily estimated using the
  Constancy Lemma \ref{lem:constancy} for $x_{3(j-2)+2}$ and the fact
  that all other particles that appear in the integrands remain
  large. The upshot is that
  \begin{align*}
    |\calI_j^2| + |\calI_j^3|
    &\leq
    \left( C R_{j-2}^{-\beta} + 1 \right)
    \left( | \tau_{3(j-1)+1} - \bar \tau_{3(j-1)+1} |
      + | \tau_{3(j-1)+2} - \bar \tau_{3(j-1)+1} | \right)
    \\
    &\leq
    C R_{j-1}^{-\beta}
    \left( R_{j-1}^{4\beta+1} \delta \right)^{\frac{\beta+1}{3\beta+1}}
    \\
    &=
    C \left( R_{j-1}^{4\beta+1} \delta \right)^{\frac{1}{3\beta+1}}
    \left( R_{j-1}^{\beta} \delta \right)^{\frac{\beta}{3\beta+1}},
  \end{align*}
  where we have used Corollary \ref{cor:diff-van-time-estimate} in
  the second inequality.
\end{proof}

Theorem \ref{thm:iterative-estimate} is now a consequence of Corollary
\ref{cor:remaining-particles} and Proposition \ref{pro:active-large},
once we choose $\gamma_*$ sufficiently small and $N_*$, $C_*$
sufficiently large.

% -----------------------------------------------------------------------------
% - Proof of Non-Uniqueness
% -----------------------------------------------------------------------------

\subsection{Proof of Theorem \ref*{thm:non-uniqueness}}
\label{sec:proof-theorem}

We now look for an appropriate sequence $(\eps^N)$ and send $N \to
\infty$.
Clearly, by restricting ourselves to a not relabeled subsequence we
may assume that $R^N_{j,p}$ converges to some $R_{j,p} \geq R_j/2$ as
$N \to \infty$ for all $j \in \bbN$ and $p=1,2$.

Suppose for the moment that we already have some $\eps^N$ for large
$N$ and that we can apply Theorem \ref{thm:iterative-estimate}
iteratively to obtain a sequence $\delta_{j-1} = \bar \delta
(\delta_j)$ for $j=N_*+1,\ldots,N$ and $\delta_N = \eps^N$.
Then we have
\begin{equation*}
  \frac{3}{4} a_*
  \left( R_{j-1}^{4\beta+1} \delta_j \right)^{\frac{1}{3\beta+1}}
  \leq
  \delta_{j-1}
  \leq
  \frac{5}{4} a_*
  \left( R_{j-1}^{4\beta+1} \delta_j \right)^{\frac{1}{3\beta+1}},
\end{equation*}
and after taking logarithms these inequalities become
\begin{equation}
  \label{eq:choice-of-delta_N-1}
  \frac{\ln \delta_j}{3\beta+1}
  +
  C_1
  +
  C_3 \ln R_{j-1}
  \leq
  \ln \delta_{j-1}
  \leq
  \frac{\ln \delta_j}{3\beta+1}
  +
  C_2
  +
  C_3 \ln R_{j-1},
\end{equation}
where $C_1 = \ln ( 3a_*/4)$, $C_2 = \ln (5a_*/4)$ and $C_3 =
(4\beta+1) / (3\beta+1)$.  Iterating \eqref{eq:choice-of-delta_N-1} we
arrive at
\begin{equation}
  \label{eq:choice-of-delta_N-2}
  \frac{\ln \eps^N}{(3\beta+1)^{N+1-j}}
  +
  S_1(N,j)
  \leq
  \ln \delta_{j-1}
  \leq
  \frac{\ln \eps^N}{(3\beta+1)^{N+1-j}}
  +
  S_2(N,j)
\end{equation}
where
\begin{equation*}
  S_p(N,j)
  =
  \sum_{k=0}^{N-j} \frac{C_p + C_3 \ln R_{j+k-1}}{(3\beta+1)^k}
  =
  \sum_{k=0}^{N-j} \frac{C_p + C_3 (j+k-1) \ln \gamma}{(3\beta+1)^k}
\end{equation*}
due to $R_j = \gamma^j$. For fixed $j$ the sums $S_p(N,j)$ converge
absolutely as $N \to \infty$ and we have
\begin{equation}
  \label{eq:growth-of-sums}
  |S_p(N,j)|
  \leq
  C + C | \ln \gamma| (j+1)
\end{equation}
for all $N \geq j$ with constants independent of $N$.
In particular, for $j=N_*+1$ we have
\begin{equation*}
  \frac{\ln \eps^N}{(3\beta+1)^{N-N_*}}
  +
  S_1(N,N^*+1)
  \leq
  \ln \delta_{N_*}
  \leq
  \frac{\ln \eps^N}{(3\beta+1)^{N-N_*}}
  +
  S_2(N,N_*+1),
\end{equation*}
which shows that for fixed $\delta_{N_*}>0$ we can find $\eps^N>0$ for
$N>N_*$ so that
\begin{equation}
  \label{eq:def-of-eps_N}
  \ln \eps^N
  \sim
  (3\beta+1)^{N-N_*} \ln \delta_{N_*}
\end{equation}
and so that the iterative procedure $\delta_{j-1} = \bar \delta
(\delta_j)$ for $j=N_*+1,\ldots,N$ yields the prescribed value
$\delta_{N_*}$, provided that the assumptions of Theorem
\ref{thm:iterative-estimate} are satisfied.
Moreover, from \eqref{eq:choice-of-delta_N-2} we know that
\begin{multline*}
  - 3\beta \frac{\ln \eps^N}{(3\beta+1)^{N+1-j}}
  + S_1(N,j) - S_2(N,j+1)
  \\
  \leq
  \ln \frac{\delta_{j-1}}{\delta_j}
  \leq
  - 3\beta \frac{\ln \eps^N}{(3\beta+1)^{N+1-j}}
  - S_1(N,j+1) + S_2(N,j),
\end{multline*}
and using \eqref{eq:def-of-eps_N}--\eqref{eq:growth-of-sums} we
infer that
\begin{equation}
  \label{eq:choice-of-delta_N-3}
  (3\beta+1)^{j-1+N_*} |\ln \delta_{N_*}| - C(j+1)
  \leq
  \ln \frac{\delta_{j-1}}{\delta_j}
  \leq
  (3\beta+1)^{j-1+N_*} |\ln \delta_{N_*}| + C(j+1)
\end{equation}
for $j=N_*+1,\ldots,N-1$. Since the first term on the left and on the
right of \eqref{eq:choice-of-delta_N-3} grows faster in $j$ than the
second we can choose $\delta_{N_*}$ so small that $|\ln \delta_{N_*}|
\gg C(N_*+1)$ and
\begin{equation*}
  \ln \frac{\delta_{j-1}}{\delta_j}
  \geq
  \ln C_*
  \qquad
  \Longleftrightarrow
  \qquad
  \delta_{j-1} \geq C_* \delta_j
\end{equation*}
for $j=N_*+1,\ldots,N-1$, where $C_*$ is from Theorem
\ref{thm:iterative-estimate}.
We also obtain $C_* \delta_N \leq \delta_{N-1}$ by using
\eqref{eq:def-of-eps_N} in \eqref{eq:choice-of-delta_N-1}.
Hence, for sufficiently small $\delta_{N_*}$ and $\eps^N$ chosen
appropriately as in \eqref{eq:def-of-eps_N}, Theorem
\ref{thm:iterative-estimate} applies for all $j=N_*+1,\ldots,N$ and
all $N>N_*+1$.  The iteration over $j$ yields
\begin{equation}
  \label{eq:almost-there}
  | x^N_{3 N_*}(T^N_{N_*}) - \bar x^N_{3 N_*}(T^N_{N_*}) |
  =
  \delta_{N_*} > 0
\end{equation}
for all $N>N_*+1$.

Using our compactness results from Section \ref{sec:existence} as well
as $\eps^N \to 0$ as $N \to \infty$, we pass to the limit along a
subsequence of $N$ and find two solutions $x$ and $\bar x$ of
\eqref{eq:master-eq} with the same initial data. Moreover, since $c
R_{N_*} \leq T_{N_*}^N \leq C R_{N_*}$, we also have $T_{N_*}^N \to
T_{N_*} > 0$ as $N \to \infty$, and \eqref{eq:almost-there} yields
\begin{equation*}
  | x_{3 N_*}(T_{N_*}) - \bar x_{3 N_*}(T_{N_*}) |
  =
  \delta_{N_*} > 0.
\end{equation*}
Then by continuity of $x_{3 N_*}$ and $\bar x_{3 N_*}$ both differ
in a time interval around $T_{N_*}$.

%%%%%%%%%%%%%%%%%%%%%%%%%%%%%%%%%%%%%%%%%%%%%%%%%%%%%%%%%%%%%%%%%%%%%%%%%%%%%%% 
%%%%%%%%%%%%%%%%%%%%%%%%%%%%%%%%%%%%%%%%%%%%%%%%%%%%%%%%%%%%%%%%%%%%%%%%%%%%%%%
%%%
%%% Acknowledgement
%%%
%%%%%%%%%%%%%%%%%%%%%%%%%%%%%%%%%%%%%%%%%%%%%%%%%%%%%%%%%%%%%%%%%%%%%%%%%%%%%%%
%%%%%%%%%%%%%%%%%%%%%%%%%%%%%%%%%%%%%%%%%%%%%%%%%%%%%%%%%%%%%%%%%%%%%%%%%%%%%%%

\section*{Acknowledgment}

The authors were supported by the German Research Foundation through
the CRC 1060.

%%%%%%%%%%%%%%%%%%%%%%%%%%%%%%%%%%%%%%%%%%%%%%%%%%%%%%%%%%%%%%%%%%%%%%%%%%%%%%% 
%%%%%%%%%%%%%%%%%%%%%%%%%%%%%%%%%%%%%%%%%%%%%%%%%%%%%%%%%%%%%%%%%%%%%%%%%%%%%%%
%%%
%%% References
%%%
%%%%%%%%%%%%%%%%%%%%%%%%%%%%%%%%%%%%%%%%%%%%%%%%%%%%%%%%%%%%%%%%%%%%%%%%%%%%%%%
%%%%%%%%%%%%%%%%%%%%%%%%%%%%%%%%%%%%%%%%%%%%%%%%%%%%%%%%%%%%%%%%%%%%%%%%%%%%%%%

\bibliography{crs_paper}
\bibliographystyle{alpha} %{abbrv}

\end{document}